\documentclass[11pt]{amsart} 
\usepackage{amsfonts,amssymb,verbatim,amsmath,amsthm,latexsym,textcomp,amscd}
\usepackage{latexsym,amsfonts,amssymb,epsfig,verbatim}
\usepackage{amsmath,amsthm,amssymb,latexsym,graphics,textcomp,hyperref,enumerate}
\usepackage{mathtools}
\usepackage[capitalize]{cleveref}
\usepackage{comment}
\usepackage{mnotes}
\usepackage[utf8]{inputenc}
\usepackage{xcolor}
\usepackage{eucal}

\newcommand{\reals}{\mathbb{R}}

\newcommand{\naturals}{\mathbb{N}}
\renewcommand{\P}{\mathcal{P}}

\newcommand{\hyperbolic}{\mathbb{H}}

\newcommand{\PSL}{{\rm PSL}_2(\mathbb{C})}
\newcommand{\SL}{{\rm SL}_2(\mathbb{C})}

\newcommand{\ie}{i.e.\ }

\newcommand{\len}{\ell}

\newcommand{\Cos}{\mathrm{Cos}}

\renewcommand{\v}{\mathbf{e}}
\newcommand{\0}{\mathbf 0}
\newcommand{\e}{\hat{\mathbf{e}}}

\newcommand{\dth}{d_{\mathrm{Th}}}
\newcommand{\dwp}{d_{\mathrm{WP}}}

\newcommand{\Log}{\mathrm{Log}}

\newcommand{\T}{\mathcal{T}}
\newcommand{\ml}{\mathcal{ML}}
\newcommand{\pml}{\mathcal{PML}}
\newcommand{\tec}{Teichm\"uller }
\renewcommand{\wp}{Weil--Petersson }

\newcommand{\R}{\mathbb R}
\newcommand{\trace}{{\rm tr}}

\newtheorem{theorem}{Theorem}[section]
\newtheorem{proposition}[theorem]{Proposition}
\newtheorem{lemma}[theorem]{Lemma}
\newtheorem{corollary}[theorem]{Corollary}

\theoremstyle{definition}
\newtheorem{definition}[theorem]{Definition}
\newtheorem{remark}[theorem]{Remark}

\newtheorem{conjecture}[theorem]{Conjecture}
\newtheorem{question}[theorem]{Question}

\includecomment{notes}
\specialcomment{notes}{\color{red}}{\color{black}}

\title[The earthquake metric on Teichm\"{u}ller space]
{The earthquake metric on Teichm\"{u}ller space}
\author[Y. Huang, K. Ohshika, H. Pan and A. Papadopoulos]{Yi Huang, Ken'ichi Ohshika, Huiping Pan, and Athanase Papadopoulos}
\address{Yi Huang, Yau Mathematical Sciences Center, Tsinghua University, Haidian District, Beijing, China.}
\email{yihuangmath@tsinghua.edu.cn}
\address{Ken'ichi Ohshika, 
Department of Mathematics,
Gakushuin University,
Mejiro, Toshima-ku, Tokyo, Japan.}
\email{ohshika@math.gakushuin.ac.jp}

\address{Huiping Pan,
School of mathematics, South China University of Technology, 510641, Guangzhou, China}
\email{panhp@scut.edu.cn} 

\address{Athanase Papadopoulos, Institut de Recherche Mathématique Avancée (Université de Strasbourg et CNRS),
7 rue René Descartes,
67084 Strasbourg Cedex France.}
\email {papadop@math.unistra.fr}
\date{\today}
\begin{document}
\sloppy
\maketitle

\begin{abstract}
This is the first paper to systematically study the earthquake metric, an asymmetric Finsler metric on Teichm\"uller space introduced by Thurston. We provide proofs for several assertions of Thurston and establish new properties of this metric, among which are incompleteness,  asymptotic distance to the boundary and comparisons with the Thurston metric and the Weil--Petersson metric.  In doing so, we propose a  novel asymmetric generalisation of the notion of completion for symmetric metrics, which we call the FD-completion, and prove that for the earthquake metric the FD-completion and various symmetrised metric completions coincide with the Weil--Petersson completion. We also answer a question of Thurston by giving an interpretation of this metric arising from a global minimisation problem, namely, the earthquake magnitude minimisation problem. At several points of this paper, we formulate a certain number of open problems which will show that the earthquake metric constitutes a promising subject.

\noindent  \emph{Keywords.---} hyperbolic surfaces, Teichm\"uller theory, earthquakes, earthquake metric, Finsler metrics, Thurston metric, Weil--Petersson metric, moduli space, augmented Teichm\"uller space, Deligne-Mumford compactification.

\noindent    \emph{AMS classification.---}  30F60, 32G15, 30F10, 52A21.
\end{abstract}

\tableofcontents

\section{Introduction}

Fenchel--Nielsen twists are deformations of hyperbolic metrics on a surface $S$ obtained by cutting the surface along simple closed geodesics and reglueing with displacement. They constitute special examples of a general family of hyperbolic surface deformations which cut, twist and reglue with respect to \emph{measured geodesic laminations}. Thurston \cite{ThE} first defined these generalised metric deformations, and dubbed them \emph{earthquakes}. 

The well-known \emph{earthquake theorem} (\cite[Theorem~2]{Ker} and \cite[III.1.3.1.]{ThE}) is Thurston's first important result on earthquakes. It asserts that  for any ordered pair of distinct hyperbolic metrics on $S$ there is a unique (left) earthquake\footnote{We say ``earthquake" without qualifying adjectives to mean the left earthquake.} deforming from the initial metric to the terminal metric. This result was famously employed in Kerckhoff's solution \cite{Ker} to the Nielsen realisation problem: Kerckhoff showed that any finite subgroup of the mapping class group of $S$ can be realised as an isometry (sub)group of some hyperbolic metric on $S$.

Kerckhoff's proof made use of \emph{convexity} (specifically, of geodesic lengths with respect to earthquakes) --- a favourite tool in the Thurston school. Another context in which convexity and earthquakes intermingle is described in \cite[Theorem~5.2]{ThM} (also see \cref{Th:Thurston-Ball} in the present paper), where Thurston showed that the derivatives of unit-speed earthquakes at each point in Teichm\"uller space $x\in\T(S)$ prescribes the boundary sphere for a convex body in $T_x\T(S)$ which contains the origin. Using this convex body, he defined a weak norm (\cref{defn:weaknorm}) to take value $1$ on this sphere, and dubbed it the \emph{earthquake norm}, which we denote by ${\|\cdot\|}_e$. The earthquake norm therefore defines a Finsler metric on Teichm\"uller space, which we call the \emph{earthquake metric} $d_e$ (\cref{defn:earthquakemetric}). 

The earthquake norm has an intriguing interpretation: Barbot and Fillastre spurred by a comment of Seppi, realised that the additive symmetrisation of the earthquake norm measures the volume of the convex core of certain associated quasi-Fuchsian co-Minkowski surfaces, see \cite[p. 650]{Barbot-Fillastre}. Yet the earthquake metric remains mysterious. Thurston muses: 

\begin{quote} \small Perhaps there is some global minimization problem which the earthquake norm measures, but I don't know [\ldots] any global interpretation for distances with respect to the earthquake norm.
\end{quote}

In response, we show in \cref{thm:magnitude} that the earthquake metric \emph{does} have a minimisation formulation in terms of earthquake \emph{magnitudes}  (\cref{defn:magnitude}). Our nomenclature retroactively reinterprets the following words by Thurston \cite[p.~22]{ThM}):
\begin{quote}\small
 The distance between two elements of Teichm\"{u}ller space in the earthquake norm is the infimum of the total magnitude of a sequence of earthquakes transforming one to the other.
\end{quote}

\noindent
In other words, the earthquake metric quantifies the following:

\begin{center}
\textit{How efficiently can we earthquake from one hyperbolic metric to another?}
\end{center}

We establish basic geometric properties of this metric, and discover connections with the Thurston (asymmetric) metric and the Weil--Petersson metric. We believe that these connections will give the earthquake metric a secure footing in Teichm\"uller and moduli space theory.

We mentioned the paper  \cite{Barbot-Fillastre} by Barbot and Fillastre in which the earthquake metric appears in the form of a norm. The two authors write, on p. 650 of their paper: 
\begin{quote}\small
[\ldots] Note that the earthquake norm also induces
an asymmetric distance on Teichmüller space, but, to the best of our knowledge,
nothing is known about this distance.
\end{quote}
The present paper is the first contribution in this direction.

Throughout this paper, $S=S_{g,n}$ denotes a Riemann surface of genus $g\geq0$, with $n\geq0$ cusps, and negative Euler characteristic. Let us state our results in some detail, anticipating some of the notation set out precisely in the next section.

\subsection{Earthquake metric}

Kerckhoff \cite[Proposition~2.6]{kerckhoff1985earthquakes} showed that every tangent vector $v\in T_x\T(S)$ at a point $x\in\T(S)$ can be uniquely expressed as an \emph{infinitesimal earthquake}, i.e. a tangent vector of the form
\[
\v_{\lambda}(x):=\left.\tfrac{\mathrm{d}}{\mathrm{d}t}\right|_{t=0} E_{t\lambda}(x)
\] 
of an earthquake path $E_{t\lambda}(x)$ shearing with respect to a measured lamination $\lambda\in\ml(S)$. Thurston \cite[Theorem~5.2]{ThM} further showed that the set
\begin{align*}
\left\{
v\in T_x\T(S)\;\mid\;
v=\v_{\lambda}(x)\text{ for }\ell_\lambda(x)=1
\right\}
\end{align*}
is the boundary of a convex ball in $T_x\T(S)$ containing the origin in its interior, and hence we can define a weak norm as follows.

\begin{definition}[earthquake norm]
\label{defn:earthquakenorm}
For every $x\in\T(S)$, the function
\begin{align*}
{\|\cdot\|}_e: T_x\T(S)\to\mathbb{R}_{\geq0},\quad
\|\v_{\lambda}(x)\|_e:=\len_x(\lambda)
\end{align*}
defines a weak norm (\cref{defn:weaknorm}) on $T_x\T(S)$  called the \emph{earthquake norm}. 
\end{definition}

\begin{definition}[earthquake metric]
\label{defn:earthquakemetric}
The earthquake norm induces a notion of length for $C^1$ paths in $\T(S)$. This determines a path metric  on $\T(S)$  (\cref{defn:finslermetric}) where the distance from $x$ to $y$ is defined as the infimum of the length, with respect to the earthquake norm, of piecewise $C^1$ paths from $x$ to $y$. We refer to this distance metric as the \emph{earthquake metric} $d_e$.
\end{definition}

\begin{remark}[left/right earthquake metric]
\label{rmk:rightearthquake}
Working with right earthquakes rather than left earthquakes, we get a distance function $d^\#_e$ instead of $d_e$. This is also a Finsler metric, satisfying $d^\#_e(x,y):=d_e(y,x)$ --- it is also called the reverse metric of $d_e$ (\cref{defn:conjmetric}). Its underlying norm is ${\|v\|}_e^\#={\|-v\|}_e$. We refer to $d_e$ and $d^\#_e$ respectively as the left and right earthquake metrics respectively when we need to distinguish between the two, and omit ``left'' when only dealing with $d_e$.
\end{remark}

The left and right earthquake metrics induce natural topologies on Teichm\"uller space. The following is an immediate corollary of \cref{prop:finslertopology}:

\begin{proposition}[usual topology]
\label{prop:topology:finsler}
The left and right earthquake metrics $d_e$ and $d_e^\#$ both induce the usual topology on Teichm\"uller space.
\end{proposition}

\subsection{Magnitude}
We  propose a natural energy function for quantifying \emph{efficient} earthquaking, and arrive at an alternative formulation of $d_e$.

\begin{definition}[magnitude]
\label{defn:magnitude}  Consider a piecewise earthquake path consisting of earthquake segments joining $x_1,\ldots,x_{m+1}$ such that $x_{i+1}=E_{\lambda_i}(x_i)$, for $i=1,\ldots,m$. We refer to the quantity ${\sum}_{i=1}^m \ell_{\lambda_i}(x_i)$ as the \emph{magnitude} of the given piecewise earthquake path.
\end{definition}

\begin{remark}
The magnitude of an earthquake segment is independent of the choice of homothetic representative. In particular, the magnitude is both 
\begin{itemize}
\item
the time taken for traversing an earthquake segment when the measured lamination $\lambda$ is normalised to have hyperbolic length $1$, and
\item
the hyperbolic length of the measured lamination $\lambda$ needed to traverse the given earthquake segment in precisely time $1$.
\end{itemize}
%This homothety invariance means that one can even vary the length of $\lambda$ whilst earthquaking, and express the magnitude as an integral of the length of $\lambda$, integrated over the time taken to travel the path.
\end{remark}

%In particular, the magnitude of an earthquake path is precisely equal to its earthquake length, which is defined independently of the path parametrisation.

\begin{theorem}[magnitude minimisation]
\label{thm:magnitude}
For arbitrary $x,y\in\T(S)$, the earthquake distance $d_e(x,y)$ is equal to the infimum of the magnitudes of all piecewise earthquake paths from $x$ to $y$.
\end{theorem}

%This result answers Thurston's minimization question which we recalled above. Yi: I commented this out because I think it's not necessary to advertise it here.

\subsection{Duality to the Thurston metric}
In \cite[Section~16.3.4.5]{Barbot-Fillastre}, Barbot and Fillastre explain a duality, first hinted at by Thurston on the last lines of page~20 and on page~21 of \cite{ThM}:

\begin{theorem}[infinitesimal duality]
\label{duality}
At each point $x \in \T(S)$, there is a linear isometry between the normed spaces $(T_x \T(S), {\|\cdot \|}_e)$ and ${(T^*_x \T(S), {\|\cdot \|}^*_{\mathrm{Th}})}$. The isometry is induced by the Weil--Petersson symplectic duality.
\end{theorem}

\begin{remark}
\cref{duality} entwines the earthquake, the Thurston, and the Weil--Petersson geometries of Teichm\"uller space. In light of upcoming results such as connecting the earthquake and Weil--Petersson metrics, it is tempting to wonder if this helps to explain the similarities between the two metrics.
\end{remark}

For expositional clarity and completeness, we give a proof of \cref{duality} via \cref{lem:WP}. Using \cref{duality}, we obtain the rigidity theorem.

\begin{theorem}[mapping class group rigidity]
\label{thm:mcg rigid}
The isometry group of $(\T(S), d_e)$ is the extended mapping class group, with the usual action. This statement holds except when $S=S_{0,4},S_{0,5},S_{1,1}$ or $S_{2,0}$, where the isometry group is the extended mapping class group modulo hyperelliptic involution.
\end{theorem}

In fact, we obtain the above rigidity result from an infinitesimal version of this theorem (see \cref{geometrically rigid}), which parallels Royden's infinitesimal rigidity theorem for the Teichm\"uller metric \cite{royden}.

\begin{remark}
Lipschitz constants for stretch maps underpinning the Thurston metric and the magnitudes for earthquakes are quantities associated with different deformation phenomena:
\begin{itemize}
\item
the former measures tensile stress accumulated during elastic deformations of hyperbolic surfaces,
\item
whereas the latter measures accumulated shear stress which result in brittle deformation of hyperbolic surfaces.
\end{itemize}
It is tempting, although heavily speculative, to wonder if the Finsler structure of the Thurston and earthquake metrics are consequences of some general physical principle regarding energy minimisation, and further, whether the infinitesimal duality has an energy-based interpretation. 
\end{remark}

\subsection{Earthquake paths and earthquake metric geodesics} \label{s:horo}
Mirzakhani \cite{Mirzakhani2008} and Calderon--Farre \cite{calderonfarre} showed that the earthquake flow is a natural conceptual hyperbolic analogue of the Teichm\"uller horocycle flow. This narrative roughly accords with Thurston's observation in \cite[p.~22]{ThM} that earthquakes are \emph{not} geodesics for this metric. In particular, he asserts that on the Teichm\"uller space of $1$-cusped tori,
\begin{center}
\textit{Earthquakes along simple curves are approximately horocycles}.
\end{center}
He remarked that the inefficiency of horocycles in navigating Teichm\"uller space suggests the following statement, of which we provide a proof, for general hyperbolic surfaces:

\begin{theorem}\label{thm:longquake:nongeo} Sufficiently long left earthquake paths in $\T(S)$ cannot be geodesics with respect to the left earthquake metric.
\end{theorem}

We presently have very little understanding regarding geodesics of the earthquake metric, and much is open for discovery. Here is a sampling of some of the questions we wish to be able to answer:

\begin{question}
Is it possible to give an explicit (i.e. in coordinates) description of even a single geodesic segment for the earthquake metric?
\end{question}

\begin{question}
Can short earthquake segments ever be geodesic?
\end{question}

\begin{question}
Is the reverse-time parameterisation of some geodesic ever also a geodesic for the earthquake metric?
\end{question}

\begin{question}\label{item}
Is the earthquake metric geodesically convex, i.e. can two arbitrary points in Teichm\"uller space  be joined by a geodesic?
\end{question}

\subsection{Norm and metric comparisons}

%{\color{orange}
%Thurston mused in \cite{ThM}, that
%\begin{quote}
%``\textit{Perhaps some day we will understand how all the different structures on Teichm\"uller space fit together, but that day has not yet arrived.}''
%\end{quote}
%This fabled day still seems far in the distance, and we add to the difficult of this task by introducing a new metric structure. Fortunately, we shall see that the earthquake metric has a potentially bridging role, with links to both the Thurston Lipschitz metric (via infinitesimal duality \cref{duality}) and the Weil--Petersson metric (via \cref{cor:e<wp}, \cref{thm:completion} and \cref{thm:coarse:geometry}).} 

To help familiarising and contextualising the earthquake metric, we first study comparisons among the following Finsler norms on Teichm\"uller space: 
\begin{itemize}
\item
the earthquake norm ${\|\cdot\|}_e$, 
\item
the Weil--Petersson norm ${\|\cdot\|}_{\mathrm{WP}}$, 
\item
the Thurston norm ${\|\cdot\|}_{\mathrm{Th}}$, and 
\item
the Teichm\"uller norm ${\|\cdot\|}_{\mathrm T}$.  
\end{itemize}

Before stating  the theorem, let us introduce a notation. 
For any $x\in\T(S)$,  let $\ell_{\mathrm{sys}}(x)$ be the \emph{systolic length} of the hyperbolic surface $x$, that is, the length of a shortest  closed geodesic on $x$. 
With this notation in preparation, our result is:
\begin{theorem}
\label{thm:e:wp:th:m} 
There are positive constants $C_0,C_1,C_2$, depending only on the topology of $S$, such that for any $x\in\T(S)$ and any $v\in \mathrm{T}_x\T(S)$, 
\begin{align*}
C_0 \ell_{\mathrm{sys}}(x){\Log}(\tfrac{1}{\ell_{\mathrm{sys}}(x)})\|v\|_{\mathrm{Th}}
\leq 	
\|v\|_e 
\leq 
C_1\|v\|_{\mathrm{WP}} 
\leq  
C_2 \|v\|_{\mathrm{Th}} ,
\end{align*}
where $\Log (t):=\max\{1,\log t\}$.
\end{theorem}

There is one more metric we shall consider on Teichm\"uller space, namely, the Teichm\"uller metric; we denote it by $d_T$, and we denote the norm of a vector $v$ for this metric by $\|v\|_T$.
With this notation, combining \cref{thm:e:wp:th:m} with Wolpert's \cite[Lemma 3.1]{Wolpert1979} and an estimate due to Burns--Masur--Wilkinson \cite[Lemma 5.4]{BMW2012},  we can form the following \emph{cycle} of comparisons:
 
\begin{corollary}\label{thm:various}
	There are   positive constants $C_0,C_1,C_2,C_3$ depending only on the topology of $S$, such that for any $x\in\T(S)$ and any $v\in T_x\T(S)$, 
\begin{align*}
C_0 \ell_{\mathrm{sys}}(x)\Log\frac{1}{\ell_{\mathrm{sys}}(x)}\|v\|_{\mathrm{Th}}\leq 	\|v\|_e{\leq} C_1\|v\|_{\mathrm{WP}}{\leq}  C_2 \|v\|_{\mathrm{Th}}{\leq}2C_2\|v\|_T{\leq}\frac{C_3\|v\|_{\mathrm{WP}}}{\ell_{\mathrm{sys}}(x)} . 
\end{align*}
\end{corollary}

\cref{thm:various} induces the following metric comparisons.

\begin{corollary}\label{cor:e<wp}
There are constants $C_1,C_2$ depending only on the topology of $S$ such that for any two points $x,y$ in $\T(S)$, we have
\begin{align*}
d_e(x,y)
\leq C_1 d_{\mathrm{WP}}(x,y)
\leq C_2 d_{\mathrm{Th}}(x,y)\leq 2C_2 d_T(x,y).
\end{align*}
\end{corollary}

\subsection{Metric incompleteness and completion}

Since the Weil--Petersson metric is incomplete, one intuitively expects the uniform domination of $d_e$ by $d_{\mathrm{WP}}$, given in \cref{cor:e<wp}, to be an obstruction to any natural notion of completeness for the earthquake metric --- this includes a notion called forward \emph{FD-completeness}, defined for asymmetric metrics, which we introduce in \cref{sec:fdcompletion}:

\begin{proposition}
\label{thm:incomplete}
The earthquake metric is not forward FD-complete.
\end{proposition}

In \cref{sec:fdcompletion} and \cref{appendix:FD} we also define the forward \emph{finite distance-series completion (FD-completion)} --- a novel metric completion for asymmetric metric spaces. 
The FD-completion generalises the Cauchy completion for symmetric metric spaces (\cref{thm:Fd=cauchy}), and naturally contains the original metric space as a forward-dense subset (\cref{prop:natisoinclude}). Lipschitz maps in the original space can be extended to the FD-completion in a functorial way (\cref{prop:FDmorphisms}).

%Back to the earthquake metric in consideration, we show:
\begin{theorem}[completion]\label{thm:completion}
The forward FD-completions of $\T(S)$ with respect to both the left and right earthquake metrics coincide with the Weil--Petersson metric completion, i.e.\ the augmented Teichm\"uller space, obtained by adding boundary strata corresponding to Teichm\"{u}ller spaces of nodal Riemann surfaces.
\end{theorem} 

Combining \cref{thm:completion}, \cite{harvey1974chabauty}, and \cite{masur1976extension}, we further obtain:

\begin{corollary}[Deligne--Mumford compactification]
\label{thm:dmcompact}
The forward FD-completion of the moduli space $\mathcal{M}(S)$ with respect to the earthquake metric coincides with the Deligne--Mumford compactification.
\end{corollary}

\begin{remark}
The fact that the earthquake metric, which is defined purely in terms of hyperbolic geometry, is able to access the Deligne--Mumford compactification with its complex analytic/algebraic geometric roots, may seem unexpected. This is reminiscent of a result of Wolpert in \cite{wolpert1986thurston}, saying that the Weil--Petersson metric, which has complex analytical origins, is equal to a certain Riemannian metric first constructed by Thurston, defined in terms of the behaviour of the hyperbolic lengths of closed geodesics.
\end{remark}

\begin{remark}
We shall show in \cref{thm:universalcompletion} that analogous results concerning augmented Teichm\"uller space  also hold for symmetrisation (\cref{defn:metricsymm} and \cref{defn:finslersymm}) of the earthquake metric. Among these, we shall highlight the Finsler symmetrisations $D^{(1)}$ and $D^{(\infty)}$ (see \cref{defn:finslersymm}) as potential objects for closer study. The former has a geometric interpretation as its Finsler norm in terms of co-Minkowskian surfaces \cite[p. 650]{Barbot-Fillastre}, and the latter infimises magnitude without favouring left or right earthquakes.
%{\color{orange}Since the earthquake metric is asymmetric, we may consider its symmetrisation. By Corollary \ref{cor:e<wp}, the symmetrised  metric is not complete as well. We prove that the (Cauchy) completion of the symmetrised  metric  also coincides with the Weil--Petersson completion of $\T(S)$ (Theorem \ref{thm:universalcompletion}).}
\end{remark}

\begin{theorem}[metric extension]\label{thm:!extension}
The left earthquake metric $d_e$ has a unique continuous extension $\bar{d}_e$ to an asymmetric metric on the forward FD-completion $\overline{\T(S)}$. Likewise, the right earthquake metric $d_e^\#$ uniquely extends to an asymmetric metric $\bar{d}_e^\#$ on the forward FD-completion $\overline{\T(S)^\#}$ of $(\T(S), d_e^\#)$. 
Both $(\overline{\T(S)},\bar{d}_e)$ and $(\overline{\T(S)}^\#,\bar{d}_e^\#)$ are Busemannian metric spaces.
\end{theorem}

We introduce the term ``Busemannian'' in \cref{defn:busemannian}.

\begin{question}
By  \cref{thm:dmcompact}, the completion locus of the earthquake metric consists of strata of lower-dimensional Teichm\"uller spaces (of possibly disconnected surfaces). Does the extended metric on each stratum coincide with the intrinsic earthquake metric for that Teichm\"uller space?
\end{question}

\subsection{Distance to the boundary}

The first inequality in \cref{cor:e<wp} says a great deal more than just the incompleteness of the earthquake metric. It asserts that any point in $\T(S)$ is a bounded distance away from the completion locus. 

\begin{corollary}[Bounded diameter]\label{cor:ms:diameter}
The moduli space $\mathcal{M}(S)$ has bounded diameter with respect to the earthquake metric.
\end{corollary}

One line of inquiry is to ponder:

\begin{question}
What is the diameter of $\mathcal{M}(S)$ with respect to the earthquake metric?
\end{question}

Another is to study the behaviour of the earthquake metric as one approaches the completion locus:

\begin{theorem}\label{thm:dist:boundary:new}
Let $\overline{\mathcal{M}(S)}$ denote the completion of the moduli space with respect to the earthquake metric. There exists  a constant $C_S>1$, depending only on the topology of $S$, such that for any $x\in\mathcal{M}(S)$,

\begin{align}
2\ell_{\mathrm{sys}}(x)
\mathrm{Log}
(\tfrac{1}{\ell_{\mathrm{sys}}(x)})
\leq 
d_e(x,\partial \overline{\mathcal{M}(S)})
\leq 
2C_S\cdot\ell_{\mathrm{sys}}(x)
\mathrm{Log}(\tfrac{1}{\ell_{\mathrm{sys}}(x)}),
\end{align}
where, as before, $\mathrm{Log}(t):=\max\{1,\log(t)\}$. Moreover, we have the following asymptotic behaviour as $\ell_{\mathrm{sys}}(x)\to 0$
\begin{align}
\frac
{d_e(x,\partial \overline{\mathcal{M}(S)})}
%\sim
{2\ell_{\mathrm{sys}}(x)
\mathrm{Log}(\tfrac{1}{\ell_{\mathrm{sys}}(x)})}
\to 1.
\end{align}
Both inequalities hold with $d_e(x,\partial \overline{\mathcal{M}(S)})$ replaced by $d_e(\partial \overline{\mathcal{M}(S)},x)$.
\end{theorem} 

\begin{remark}
Although \cref{thm:dist:boundary:new} is phrased in terms of the moduli space, the result is still true if we replace $\mathcal{M}(S)$ by $\T(S)$.
\end{remark}

\subsection{Further comparisons with the Weil--Petersson metric}

Yet another line of inquiry suggested by \cref{cor:e<wp} is whether $d_e$ and $d_{\mathrm{WP}}$ are bi-Lipschitz equivalent. We show in \cref{prop:notbilipschitz} that this is false, essentially because $d_e$ and $\dwp$ behave very differently in the thin part of Teichm\"uller space. Such fine structure is ignored by coarse geometry, and we have the following  positive result:

\begin{theorem}[Earthquake vs. Weil--Petersson]
\label{thm:coarse:geometry}
The earthquake metric is quasi-isometric to the Weil--Petersson metric. More precisely, there are constants $D_1,D_2>0$ which depend only on the topology of $S$, such that
\begin{align*}
\forall x,y\in\T(S),\quad 
&D_1 \dwp (x,y)-D_2\leq d_e(x,y)\leq C \dwp (x,y)
\end{align*}
where $C$ is the constant from \cref{cor:e<wp}. \end{theorem}

The proof of \cref{thm:coarse:geometry} also shows that $(\T(S),d_e)$ is quasi-isometric to the pants graph $\mathcal{P}(S)$ --- the infinite graph whose vertices are pants decompositions of $S$, with an edge between any two pants decompositions if and only if the decompositions are related by an elementary move. 
Indeed, Brock's result \cite[Theorem~1.1]{Brock2003} saying that $(\T(S),d_{\mathrm{WP}})$ is quasi-isometric to $\mathcal{P}(S)$ implies the claim. 

\begin{remark}[navigating Teichm\"uller space]
\cref{thm:coarse:geometry} suggests that (at least) for ``long'' paths in Teichm\"uller space endowed with the earthquake metric, a potential strategy for efficiently traversing this metric space is to use the pants complex as a guide, running between regions of Teichm\"uller space with very short pants decompositions.
\end{remark}

\subsection{Asymmetry of the earthquake metric}

It is known that the Thurston norm is  asymmetric at some points in $\T(S)$ (see \cite[Figure 1]{ThM} and \cite[Theorem 5.3]{Theret}). This confers a similar norm-wise asymmetry to the earthquake norm via \cref{duality}. 
We shall prove a stronger asymmetry using the infinitesimal rigidity (\cref{geometrically rigid}), that is, we shall prove that the earthquake norm, hence the Thurston norm also, is \emph{nowhere} symmetric.

\begin{theorem}[norm asymmetry]
\label{prop:asymmetric}
The earthquake norm is asymmetric at every point of $\T(S)$. That is:  for any $x\in\T(S)$, there exists a tangent vector $v\in T_x\T(S)$ such that $\|v\|_e\neq\|-v\|_e$. 
\end{theorem}

As a direct consequence of \cref{prop:asymmetric}, we obtain:

\begin{corollary}\label{cor:metricasym}
The earthquake metric is not symmetric, that is, there exist $x$ and $y$ in $\T(S)$ such that $d_e(x,y)\neq d_e(y,x)$.  Equivalently, the left earthquake metric and the right earthquake metric are not equal.
\end{corollary}

\begin{remark}
Curiously, we have not been able to find a direct argument which proves \cref{cor:metricasym}.
\end{remark}
 The following ``equivalences'' we obtain (often with the Weil--Petersson metric as a bridge) between the left and right earthquake metrics compensate for the asymmetry of the (left) earthquake metric:
\begin{itemize}
\item
\cref{prop:topology:finsler} says that they are homeomorphic;

\item
\cref{thm:coarse:geometry} shows that they are quasi-isometric;

\item
\cref{thm:completion} shows that they have the same forward FD-completion;

\item
and \cref{thm:dist:boundary:new} shows that they have similar behaviour in terms of distances from the completion locus.

\end{itemize}

And one is tempted to pose the following question.

\begin{question}
Is the left earthquake metric bi-Lipschitz equivalent to the right earthquake metric?
\end{question}

Actually, a positive answer to the above question would suffice to show that the earthquake metric is geodesic. However, there is every likelihood that the two metrics are \emph{not} bi-Lipschitz equivalent.

\subsection*{Organisation of the paper}

Our paper is structured as follows.
\begin{itemize}
\item
In \cref{s:background}, we review some of the necessary background Teichm\"uller theory and a small amount of asymmetric metric theory. 

\item
We investigate the properties of the earthquake norm in \cref{s:enorm}, showing that it defines a Finsler metric.

\item
In \cref{s:regularity}, we strengthen the known regularity of $E_{t\lambda}(x)$ with respect to both time $t$ and space $x$ derivatives. We use this to show that the earthquake metric is (locally) Lipschitz continuous.

\item
We then utilise the established regularity to show, in \cref{sec:magnitude}, an energy-based formulation of the earthquake metric.

\item
In \cref{sec:infinitesimal}, we show that the the isometry group of the earthquake metric is equal to the extended mapping class group. The infinitesimal version of this result can be used to show that the earthquake metric is asymmetric.
 
 \item In \cref{sec:comparison},  we compare the earthquake metric with the Thurston metric, the Weil--Petersson metric, and the Teichm\"uller metric on $\T(S)$, both infinitesimally and globally.
\item
In \cref{s:geodesy}, we investigate the geodesic properties of the earthquake metric, showing a negative result, namely, that (long) earthquakes are not geodesics of this metric.

\item
We then go on to study the behaviour of the earthquake metric in the thin part of Teichm\"uller space, determining the magnitude of earthquakes needed to deform a marked surface with short systole to a nodal Riemann surface.

\item
In \cref{s:completion}, we define FD-completions of asymmetric metric spaces, and show that the FD-completion of Teichm\"uller space equipped with the earthquake metric is equal to the Weil--Petersson completion. 

\item
Finally, in \cref{s:vsWP}, we show that the earthquake metric is coarsely equivalent to the Weil--Petersson metric. 

\end{itemize}
At the end of this paper, we have included two appendices, the first outlining general facts regarding Finsler metrics of low regularity, and the second establishing various foundational properties of the FD-completion of an asymmetric metric space.

\subsection*{Acknowledgements} We thank Harry Baik, Jialong Deng, Xiaolong Han, Hideki Miyachi, Nicolai Reshetikhin, Andrea Seppi, Weixu Su, Ivan Telpukhovskiy, Marc Troyanov, Yunhui Wu and Wenyuan Yang for useful and/or encouraging conversations regarding this work. The first listed author (YH) wishes to acknowledge support by Beijing Natural Science Foundation Grant no. 04150100122. The second listed author (KO) is supported by the JSPS Fund for the Promotion of Joint International Research 18KK0071. The third listed author (HP) is supported by the National Natural Science Foundation of China NSFC 12371073 and Guangzhou Basic and Applied Basic Research Foundation 2024A04J3636.  The authors acknowledge the support of the Erwin Schr\"odinger international Institute for Mathematics and Physics.

\newpage

\section{Background}\label{s:background}

In this section, in an effort to make the paper self-contained, we have collected a certain number of fundamental ideas, definitions and properties which explain the results of this paper and set them in the appropriate context.

\subsection{Teichm\"uller and moduli space}

Given an oriented connected surface of genus $g\geq 0$ with $n\geq 0$ punctures, $S=S_{g,n}$, we let 
\begin{itemize}
\item
$\T(S)$ denote the \emph{Teichm\"uller space} of $S$ --- the space of complete hyperbolic metrics on $S$ up to isometries isotopic to the identity map on $S$;
\item
$\mathcal{M}(S)$ denote the \emph{moduli space} of $S$ --- the space of isometry classes of complete hyperbolic metrics on $S$.
\end{itemize}

For expositional simplicity, we may refer to $x\in\T(S)$ simply as a marked (hyperbolic) surface, as opposed to an equivalence class of hyperbolic metrics. Likewise, we refer to a point in $\mathcal{M}(S)$ as a (hyperbolic) surface.

The topology of $\T(S)$ is well understood: the space is analytically equivalent to $\mathbb{R}^{6g-6+2n}$. The moduli space $\mathcal{M}(S)$ has greater topological complexity, and one approach to it is to understand it as a quotient of $\T(S)$ by the \emph{mapping class group}
\[
\mathrm{Mod}(S):=
\frac{\{\text{ orientation-preserving diffeomorphisms of }S\ \}}
{\{\text{ diffeomorphisms of }S\text{ isotopic to the identity}\ \}}.
\]
It is well known that
\begin{itemize}
\item 
the mapping class group $\mathrm{Mod}(S)$ acts (non-freely) properly discontinuously on $\T(S)$,
\item
the moduli space $\mathcal{M}(S)$, which is the quotient space, is an orbifold, 
\item
and since the Teichm\"uller space $\T(S)$ is simply connected, it is the orbifold universal cover of $\mathcal{M}(S)$.
\end{itemize}

There are several modern introductions to these topics, see e.g.  \cite{Abikoff, FLP, FLPtranslation, IT}.
The reader interested in the history of moduli and Teichm\"uller spaces may refer to \cite{AJP-History, JP-Historical}.

\subsection{Lengths of geodesics and measured laminations}

The (hyperbolic) length of a non-contractible closed curve $\gamma$ on $S$, with respect to a hyperbolic metric $x\in\T(S)$, is defined as the length of the unique shortest closed curve (specifically, a geodesic) in the free homotopy class of $\gamma$.
From now on, by a  closed curve, we always mean a non-contractible one.
 Linearity gives us a natural extension of the notion of length to weighted curves $c\gamma$ and sums of weighted curves $\sum c_i\gamma_i$. The measured lamination space $\mathcal{ML}(S)$ is, in some precise sense, a completion of the space of weighted  \emph{simple} (i.e.\ non-self-intersecting) closed curves. The elements of $\mathcal{ML}(S)$ are geodesic laminations equipped with transverse measures, see \cite[\S 1]{Thurston-BAMS}.

The hyperbolic length function for weighted simple closed curves extends continuously to a positively homogeneous function on $\ml(S)$. For a general measured geodesic lamination $\lambda$, its length $\len_x(\lambda)$ 
is the total mass of the measure on the surface defined as the product of the Lebesgue measure along the leaves of $\lambda$ and the transverse measure of $\lambda$.

\begin{remark}[length notation]\label{rmk:lengthnotation}
In our most general context, a length function is a map
\[
\ell:\ml(S)\times\T(S)\to\mathbb{R}_{\geq0},
\]
assigning to any pair $(x,\lambda)\in\T(S)\times\ml(S)$ the length of the unique geodesic representative of $\lambda$ with respect to the metric $x$. In almost all contexts, however, we fix either the lamination $\lambda$ or the metric $x$, and we respectively denote these two functions by:
\begin{itemize}
\item
$\len_x:\ml(S)\to\mathbb{R}_{\geq0},\quad \len_x(\lambda):=\ell(\lambda,x)$;

\item
$\ell_\lambda:\T(S)\to\mathbb{R}_{\geq0},\quad \ell_\lambda(x):=\ell(\lambda,x)$.
\end{itemize}
\end{remark}

The hyperbolic length function is continuous, and is positively homogeneous in its first parameter, and the pairing $\ell$ between Teichm\"uller space and the measured lamination space induces an embedding of $\T(S)$ into the space $\mathbb{R}_{\geq0}^{\mathcal{ML}(S)}$ of real functions on $\mathcal{ML}(S)$ by ``duality'' in the following sense: for each $\lambda\in\mathcal{ML}(S)$, the $\lambda$-coordinate for $x\in\T(S)$ is given by $\ell_\lambda(x)$. Thurston showed that the $\mathbb{R}_{>0}$-projectivisation of this embedding is still an embedding, and that the boundary of Teichm\"uller space may be interpreted, using this embedding, as the space of projective measured laminations $\mathcal{PML}(S):=(\mathcal{ML}(S)-\{0\})/_{\R_{>0}}$.

\subsection{Fenchel--Nielsen twists and earthquakes}
The left Fenchel--Nielsen twist with respect to a weighted simple closed curve $c\gamma$ defines a flow on Teichm\"uller space, whereby the time $\ell_\gamma$ map of the flow corresponds to the action of the right Dehn twist along $\gamma$ (since the action of a mapping class on the Teichm\"uller  space is  precomposition by the inverse). Thurston showed that Fenchel--Nielsen twist flows, which shear along simple closed geodesics, generalise to flows with respect to measured laminations in the following way (see \cite{Ker}):
\begin{itemize}
\item
extend Fenchel--Nielsen twist flows to  flows for (positively) weighted simple closed curves $c\gamma$ via $(t,x)\mapsto E_{tc\gamma}(x)$, for $x\in \mathcal{T}(S)$ and $t\in\mathbb{R}_{>0}$.

\item
 for any sequence $(c_i\gamma_i)$ of weighted simple closed curves converging to a measured lamination $\lambda$, show that their corresponding twist flows converge, and that this limit is independent of the choice of the initial sequence $(c_i\gamma_i)$ of weighted simple closed curves. 
\end{itemize}
We refer to any such limiting flow as the \emph{(left) earthquake flow with respect to $\lambda$}:
\[
(t,x)\mapsto \lim_{i\to\infty}E_{tc_i\gamma_i}(x)=:E_{t\lambda}(x).
\]
Kerckhoff \cite{kerckhoff1985earthquakes} showed that for any $\lambda\in\ml(S)$, the earthquake flow $t\mapsto E_{t\lambda}(x)$ is analytic with respect to $t$. The infinitesimal earthquake
\[
\v(\lambda,x)
:=
\left.
\tfrac{d}{dt}
\right|_{t=0}E_{t\lambda}(x)
\]
is the basis for defining the earthquake norm (\cref{defn:earthquakenorm}), and subsequently, the earthquake metric (\cref{defn:earthquakemetric}).

\begin{remark}[earthquake vector notation]\label{rmk:earthquake-normalized}
We adopt notation similar to length function notation (see \cref{rmk:lengthnotation}) to distinguish situations when we wish to regard earthquake vectors as functions of $\ml(S)$ versus $\T(S)$:
\begin{itemize}
\item
$\v_x:\ml(S)\to T_x\T(S),\quad \v_x(\lambda):=\v(\lambda,x)$;

\item
$\v_\lambda:\T(S)\to T\T(S),\quad \v_\lambda(x):=\v(\lambda,x)$
\end{itemize}
\end{remark}

\begin{remark}
Thurston provides in \cite{ThE} a different construction/proof for earthquake flows in greater generality. In particular, this construction  works for the universal Teichm\"uller space.
\end{remark}

%\subsection{Dynamics on Teichm\"uller space}

%\begin{itemize}
%\item
%Mention a bit of Bonahon's geodesic currents story  as build a dictionary between Teichm\"uller space and ML and PML vs the hyperbolic plane in Minkowski space and the light cone and the ideal boundary circle,
%\item
%lead that to a discussion about geodesic flow vs horocyclic flows vs the stretch flow vs the left and right earthquake flows
%\item
%use that to talk about limits of stretch flow, earthquake flows (we use this in the proof of \cref{thm:longquake:nongeo}) and the idea of exponentiation maps.
%\end{itemize}

%Yi: I commented this part out because we no longer use that idea for the proof of \cref{thm:longquake:nongeo}.

\subsection{Asymmetric and Busemannian metric spaces}

Consider the following quote from Gromov in the introduction of \cite{gromov1999metric}:
\begin{quote}\small
 [\ldots] Besides, one insists that the distance function be symmetric, that is, $d(x,y)=d(y,x)$. (This unpleasantly limits many applications [\ldots])
\end{quote}
Many metrics are naturally asymmetric due to global minimisation problems innately bearing such asymmetry. For example, the (minimal) amount of energy that one expends ascending a staircase is (generally) different from the amount needed when descending. Nevertheless, such quantities satisfy the triangle inequality, and thus benefit from the wealth of metric space theory which mathematics has produced. There are various versions of asymmetric metric space theory. For our purposes, we adopt the following:

\begin{definition}[{asymmetric metric space, see, e.g.,\cite{mennucci2013asymmetric}}]
\label{defn:asymmetric}
A set $X$ endowed with a function $d:X\times X\to [0,\infty]$ is called an \emph{asymmetric metric space} if for all $x,y,z\in X$
\begin{enumerate}
\item $d(x,x)=0$;
\item $d(x,y)=d(y,x)=0\Rightarrow x=y$;
\item $d(x,z)\leq d(x,y)+d(y,z)$.
\end{enumerate}
\end{definition}

The adjective ``asymmetric" is used by Thurston for a metric which does not satisfy the symmetry axiom in his paper \cite{ThM}, which is where he introduced the Thurston (asymmetric) metric. Before Thurston, asymmetric metrics were used and studied in various contexts.  Thurston described his metric as ``Finsler asymmetric". As a matter of terminology, Finsler metrics do not necessarily assume the symmetry axiom. A Finsler metric which is symmetric is usually referred to as a ``reversible Finsler metric".

\begin{remark}
Although semantically confusing, an asymmetric metric is allowed to be symmetric. In any case, we show that the earthquake metric, like the Thurston metric, is non-symmetric. 
\end{remark}

In this paper, we shall make use of the following notions of symmetrisation of an asymmetric metric:

\begin{definition}[Symmetrisations]
Given an asymmetric metric metric $d$, its 
\emph{sum-symmetrisation} $d^{\mathrm{sum}}$ is defined as
\[d^{\mathrm{sum}} (x,y):=\tfrac{1}{2}(d(x,y)+d(y,x))
\]
and its 
\emph{max-symmetrisation} $d^{\mathrm{max}}$ is defined as
\[d^{\mathrm{max}} (x,y):=\max \{d(x,y),d(y,x)\}
.\]

\end{definition}

\begin{definition}[reverse metric]
\label{defn:conjmetric}
Given an asymmetric metric space $(X,d)$, the function $d^\#: X\times X\to[0,\infty]$ defined by $d^\#(x,y):=d(y,x)$ is (also) an asymmetric metric on $X$. We refer to $d^\#$ as the \emph{reverse metric} of $d$, and call $(X,d^\#)$ the \emph{reverse metric space} of $(X,d)$.
\end{definition}

\begin{definition}[topologies of asymmetric metric spaces]
Asymmetric metrics $d$ induce three natural choices of topologies on $X$:
\begin{itemize}
\item
the \emph{forward topology} is generated by \emph{forward open balls}, i.e.\  sets of the form $B^+(x,\epsilon):=\{y\in X\mid d(x,y)<\epsilon\}$, for arbitrary $\epsilon>0$;

\item
the \emph{backward topology} is generated by \emph{backward open balls}, i.e.\  sets of the form $B^-(x,\epsilon):=\{y\in X\mid d(y,x)<\epsilon\}$, for arbitrary $\epsilon>0$;

\item
the \emph{symmetric topology} is generated by the forward and backward open balls.
\end{itemize}
\end{definition}

\begin{remark}
It is fairly straightforward to see that the symmetric topology of an asymmetric metric $d$ agrees with the topology induced by the max-symmetrisation of $d$.
\end{remark}

\begin{definition}[limits]
Given an asymmetric metric space $(X,d)$, we say that $x\in X$ is a \emph{forward limit} of a sequence $(x_n)$ if $d(x_n,x)\to0$. Backward limits are analogously defined.
\end{definition}

\begin{remark}
There is an alternative notion of a forward limit stated as requiring that   $d(x_n,y)\to d(x,y)$ for all $y \in X$. This is a stricter condition than what we want.
\end{remark}

Asymmetric metrics, including Finsler, were studied extensively by Herbert Busemann, in several papers and books. Busemann is one of the most important promoters of metric geometry. We shall see later that the metric spaces we consider all fall within a strict subclass of asymmetric metric spaces which appear in his work. Busemann refers to metrics which we call asymmetric as ``not necessarily symmetric'' in \cite{Busemann-synthetic} --- this is semantically clearer, but more cumbersome to say. Nonetheless, it is helpful to have a name for this more general framework. We consider that the following name is deserved:

\begin{definition}[Busemannian metric space]
\label{defn:busemannian}
A set $X$ endowed with a function $d:X\times X\to [0,\infty)$ is called an \emph{Busemannian metric space}
 if
\begin{itemize}
\item for all $x,y\in X$, $d(x,y)=0\Leftrightarrow x=y$;
\item for all $x,y,z\in X$, $d(x,z)\leq d(x,y)+d(y,z)$.
\item for all $x\in X$ and for all sequences $(x_n)$ in $X$, we have $d(x_n,x)\to 0$ if and only if $d(x,x_n)\to 0$.
\end{itemize}
\end{definition}

Spaces satisfying the three properties of \cref{defn:busemannian} are highlighted and studied in \cite{busemann1944} (of course, without the name).

\begin{remark}
The concept of Busemannian metric spaces should not be confused with objects such as Busemann (convex) spaces and Busemann G-spaces, which are geodesic metric spaces satisfying additional conditions, and are tied in with different developments in metric geometry.
\end{remark}

\begin{remark}
\label{rmk:continuityissues}
For Busemannian metric spaces, the forward, backward and symmetric topologies all agree. This follows from the fact that these topologies are first countable. We set this topology as the \emph{standard topology} on $X$. Since $d^{\max}$ is symmetric, continuity of functions defined on Busemannian metric spaces may be verified as usual in the language of limits and/or convergent sequences.
\end{remark}

\subsection{Finsler geometry of low regularity}

Thurston's paper \cite{ThM} on the Thurston metric, where he defines the earthquake metric, deals with asymmetric Finsler metric structures on Teichm\"uller space with low regularity (i.e.\  $C^0$) on the metric. To begin with, we clarify the manner in which asymmetry manifests itself at the level of the Finsler norm.

\begin{definition}[weak norm]
\label{defn:weaknorm}
We define  a \emph{weak norm} ${\|\cdot\|}$ on a vector space $V$  as a function ${\|\cdot\|}:V\to[0,\infty)$ satisfying the following properties:
\begin{enumerate}
\item $\| v\|=0$ if and only if $v=0$;
\item $\| v_1+v_2\|\leq \| v_1\| + \| v_2\|$ for all $v_1, v_2$ in $V$;
\item $\| \lambda v\|= \lambda \| v\|$ for all $v$  in $V$ and for all $\lambda >0$.
\end{enumerate}
\end{definition}

\begin{remark}As in the case of a symmetric norm, a weak norm $\| \cdot\| $ on $V$ is completely determined by its unit sphere (for Finsler norms, this is sometimes called the \emph{indicatrix}), that is, the set of vectors $v\in V$ satisfying $\| v\| =1$. 
This may be an arbitrary hypersurface bounding a compact convex body in $V$ containing the origin in its interior. 
\end{remark}

\begin{definition}[Finsler metric]
\label{defn:finslermetric}
A \emph{Finsler metric} or a \emph{Finsler norm} on a path-connected $C^1$-manifold is a continuous function ${\|\cdot\|}_X: TX\to\mathbb{R}_{\geq0}$ where the restriction of $F$ to each tangent space is a weak norm. For any $C^1$ path $p:I\rightarrow X$ we may compute its length $\mathrm{L}(p)$ with respect to ${\|\cdot\|}_X$ via $\mathrm{L}(p):=\int_I \|p'(s)\|_X\ \mathrm{d}s$. Given any $x,y\in X$, let $\mathbf{P}(x,y)$ be the set of piecewise $C^1$ paths in $X$ starting at $x$ and ending at $y$. We define the metric $d_X$ induced by the Finsler metric ${\|\cdot\|}_X$ as
\begin{align*}
d_X(x,y):=
\inf_{p\in\mathbf{P}(x,y)}
\mathrm{L}(p)
=
\inf_{p\in\mathbf{P}(x,y)}
\int_I \|p'(s)\|_X\ \mathrm{d}s.
\end{align*}
\end{definition}

\begin{remark}
General principles (see  \cref{prop:finslerbusemann} of the first appendix) tell us that Finsler metrics are always Busemannian, and the respective topologies induced by these two types of metric structures are equivalent.
\end{remark}

\subsection{The Thurston metric}
\label{s:thurstonmetric}

The \emph{Thurston (asymmetric) metric} measures the logarithmic optimal Lipchitz constant of maps between hyperbolic metrics. Specifically, given any homeomorphism $\varphi:(S,x)\to(S,y)$, define the \emph{Lipschitz constant of $\varphi$} as the quantity
\[
\mathrm{Lip}(\varphi)
:=
\sup_{\text{distinct }p,q\in S}
\frac{d_y(\varphi(p),\varphi(q))}{d_x(p,q)}\in\mathbb{R}\cup\{\infty\},
\]
where $d_x$ and $d_y$ respectively denote distance metrics on $S$ induced by (some fixed choice of isotopy class representative for) the hyperbolic metrics $x$ and $y$. The Thurston metric is defined by
\[
d_{\mathrm{Th}}(x,y)
:=
\inf_{\varphi\sim\mathrm{Id}_S}
\log\mathrm{Lip}(\varphi),
\]
where the infimum is taken over all homeomorphisms $\varphi$ isotopic to the identity map $\mathrm{Id}_S$ on $S$. Remarkably, Thurston \cite[Theorem 8.5]{ThM} demonstrated an alternative interpretation for $\dth(x,y)$ as the supremal ratio of lengths of simple closed curves with respect to $x$ and $y$. 
Let $\mathcal S$ denote the set of essential simple closed curves on $S$.
Then
\begin{align}\label{eq:thurston:ratio}
\dth(x,y)
=
\log\sup_{\gamma\in\mathcal S}
\frac{\ell_\gamma(y)}{\ell_\gamma(x)}
=
\log\max_{\gamma\in\ml( S) \setminus \{0\}}
\frac{\ell_\lambda(y)}{\ell_\lambda(x)}.
\end{align}
%\HP{previous we use the notion $m$ and $m'$ for hyperbolic surfaces, here (in this section) we use $x$ and $y$. Shall we put notations in consistency? }

Thurston further showed that $d_{\mathrm{Th}}$ is  a Finsler metric on $\T(S)$. First note that the derivative of $\log \ell_\lambda:\T(S)\to\mathbb{R}$ depends only on the projective class $[\lambda]\in\pml(S)$. Thus, for each $x \in \T(S)$, we may define a map $\iota_x:\pml(S)\to T_x^*\T(S)$
\[
[\lambda]\mapsto
\iota_x([\lambda]):=(\mathrm{d}\log \ell_\lambda)_x \in T^*_x \T(S).
\]
\begin{theorem}[\cite{ThM}, Theorem 5.1]
\label{embeds into cotangent}
For every $x \in \T(S)$, the map $\iota_x$ is an embedding of the projective lamination space $\pml(S)$ into the cotangent space $T_x^* \T(S)$. The image set $\iota_x(\pml(S))$ is the boundary of a convex open ball containing the origin in $T_x^*\T(S)$.
\end{theorem}
By taking $\iota_x(\pml(S))$ to be the indicatrix, this defines a (co-)norm on the cotangent space $T_x^*\T(S)$. The \emph{Thurston norm} on the tangent space $T_x\T(S)$ is defined by duality:
 \begin{align}\label{eq:thurston:norm}
%\forall~ v\in T_x \T(S),
%\quad
\|v\|_{\mathrm{Th}}
:=
\sup _{\lambda \in \mathcal{PML}(S)}
(\mathrm{d} \log\ell_\lambda)_x (v)
=
\sup _{\lambda \in \mathcal{PML}(S)}
\frac{v(\ell_\lambda)}{\ell_\lambda}.
\end{align}

%The induced path metric is called the \emph{Thurston metric} (\cite[Page 20]{ThM}, see also \cite[Theorem 2.3]{PapadopoulosSu2015}).

%\HP{the derivative of distance along a $C^1$ path is the length of that tangent vector, see Weixu and Papadopoulos}

%    The unit sphere of the Thurston conorm in the cotangent space $T^*_x\T(S)$ is exactly the image of $\pml(S)$ under the map $d\log\ell$.

\subsection{The Weil--Petersson metric}
\label{sec:wp}
Teichm\"uller space is naturally a complex analytic manifold (see, e.g., \cite[chapter~7]{IT}):
\begin{itemize}
\item  
identify the tangent space $T_x\T(S)$ with the space $\mathit{HB}(x)$ of \emph{harmonic Beltrami differentials}  on $x$, and
\item
set the complex structure $\sqrt{-1}:T_x\T(S)\to T_x\T(S)$ on $T_x\T(S)$ to be multiplication by $\sqrt{-1}$ on $\mathit{HB}(x)$.

%the cotangent space $T^*_x\T(S)$ identifies with the space $Q(x)$ of \emph{holomorphic quadratic differentials} on $x$.
\end{itemize}

\begin{definition}[Weil--Petersson K\"ahler metric]
Let $\rho(z)|dz|^2=x\in\T(S)$ denote a hyperbolic metric on $S$, and consider the following Hermitian inner product on $\mathit{HB}(x)$: for any $f_1,f_2\in \mathit{HB}(x)$,
\begin{align}
\label{eq:petersson}
h_{\mathrm{WP}}(f_1,f_2)
:=\int_S 
\rho(z)^2
f_1(z)\overline{f_2(z)}
\ dA,
\quad
\end{align}
where $dA$ is the area form for the metric $\rho(z)|dz|^2$ corresponding to $x$. It turns out that the collection of such Hermitian inner products varying over the tangent bundle of $\T(S)$ defines a K\"ahler metric $h_{\mathrm{WP}}$ on $\T(S)$; this is the \emph{Weil--Petersson K\"ahler metric} on $\T(S)$.
\end{definition}

\begin{itemize}
\item
The real component of $h_{\mathrm{WP}}$ defines a Riemannian metric $\langle\cdot,\cdot\rangle_{\mathrm{WP}}$, and we refer to it as the \emph{Weil--Petersson metric}. 

\item
We write ${\|\cdot\|}_{\mathrm{WP}}$ for the Weil--Petersson norm on tangent spaces, and denote the induced distance by $d_{\mathrm{WP}}$.

\item
The (constant multiple of) the imaginary component, $-2 \mathrm{Im}(h_{\mathrm{WP}})$ defines a symplectic form $\omega_{\mathrm{WP}}$, which is called the Weil-Petersson symplectic form.
\end{itemize}

\begin{remark}
We note that \cref{eq:petersson} is unconventional in the sense that the Petersson inner product is usually phrased in terms of holomorphic quadratic differentials. 
\end{remark}

The following lemma is an essential ingredient for the comparison of various norms on the Teichm\"uller space $\T(S)$.  

\begin{lemma}\label{lem:wp:nabla}
For any simple closed curve $\alpha$ on $S$, let $\nabla \ell_\alpha$ denote the Weil--Petersson gradient field for $\ell_\alpha:\T(S)\to \R$, then
\begin{align*}
\|\v_\alpha(x)\|_{\mathrm{WP}}=\tfrac{1}{2}\|(\nabla \ell_\alpha)_x\|_{\mathrm{WP}}.
\end{align*}
\end{lemma}
\begin{proof}
By \cite[Theorem 2.10]{Wolpert1982}, we see that\footnote{In this formula, the sign conventions differ from those of Wolpert's article \cite{Wolpert1982} because we use left earthquakes while Wolpert uses right earthquakes.}
\begin{align*}
d\ell_\alpha
= \omega_{\mathrm{WP}}(\cdot, \v_\alpha(x))
=-\omega_{\mathrm{WP}}(\v_\alpha(x),\cdot),
\end{align*}
where $\omega_{\mathrm{WP}}$ denotes the Weil--Petersson symplectic form. Then, for any tangent vector $v\in T_x\T(S)$, 
\begin{align*}
\langle\nabla \ell_\alpha, v\rangle_{\mathrm{WP}}
=
d\ell_\alpha(v)
=-\omega_{\mathrm{WP}}(\v_\alpha(x), v)
=2\langle\sqrt{-1}\v_\alpha(x),v\rangle_{\mathrm{WP}}.
\end{align*} 
Since the pairing is nonsingular, we see that $\v_\alpha(x)=\frac{\sqrt{-1}}{2}(\nabla \ell_\alpha)_x$. To conclude, we write
\begin{align*}
\|\v_\alpha(x)\|^2
_{\mathrm{WP}}
&=
\mathrm{Re}(h_{\mathrm{WP}}(\v_\alpha(x),\v_\alpha(x)))\\
&=
\mathrm{Re}(h_{\mathrm{WP}}
(\tfrac{\sqrt{-1}}{2}(\nabla \ell_\alpha)_x,
\tfrac{\sqrt{-1}}{2}(\nabla \ell_\alpha)_x))\\
&=
\mathrm{Re}(h_{\mathrm{WP}}
(\tfrac{1}{2}(\nabla \ell_\alpha)_x,
\tfrac{1}{2}(\nabla \ell_\alpha)_x))
=
\|\tfrac{1}{2}(\nabla \ell_\alpha)_x\|^2
_{\mathrm{WP}}.
\end{align*}
\end{proof}

\subsection{Complexified character varieties}

Let $\mathcal{X}^{\mathrm{irred}}_{\mathrm{para}}(\pi_1(S),\PSL)$ denote the \emph{character variety} consisting of conjugacy classes of type-preserving (i.e.\  all punctures, if any, have parabolic holonomy), irreducible representations of $\pi_1(S)$ into $\PSL$. Teichm\"uller space is identified with the subset of $\mathcal{X}^{\mathrm{irred}}_{\mathrm{para}}(\pi_1(S),\PSL)$ consisting of discrete faithful representations into $\mathrm{PSL}_2(\mathbb{R})\subset\PSL$. We fix, once and for all, a spin structure on $S$. This tells us how to lift a representation from $\PSL$ to $\SL$, and so we have an embedding 
\[
\T(S)
\hookrightarrow
V(S):=
\mathcal{X}^{\mathrm{irred}}_{\mathrm{para}}(\pi_1(S),\SL),
\]
where the codomain is the character variety of type-preserving, irreducible representations of $\pi_1(S)$ into $\SL$. Faltings \cite[Theorem~3]{faltings1983real} shows that $V(S)$ is a complex manifold\footnote{Gunning shows this for the $\mathrm{GL}_2(\mathbb{C})$ character variety in earlier work \cite[\S9]{gunning1967lectures}, and it seems possible to derive the $\SL$ case using the holomorphicity and rank~1-ness of determinant maps on surface group generators.}. %Note in particular that $\T(S)$ embeds as a totally real subset of $V(S)$, and there is no complex analytic compatibility between the two spaces.

\subsection{Quasi-Fuchsian space}
\label{sec:qf}
We make mention of a particular subset $\mathcal{QF}(S)$ of $V(S)$ consisting of all characters of \emph{quasi-Fuchsian representations}. A representation $\rho:\pi_1(S)\to\mathrm{PSL}_2(\mathbb{C})$ induces a group 
\[
\rho(\pi_1(S))\lneq\mathrm{PSL}_2(\mathbb{C})=\mathrm{Isom}^+(\mathbb{H}^3),
\]
which acts on $\mathbb{H}^3$. A quasi-Fuchsian representation is defined as one whose induced group $\rho(\pi_1(S))\lneq\mathrm{PSL}_2(\mathbb{C})$ has limit set in $\partial^\infty\mathbb{H}^3=\mathbb{C}P^1$ which is a quasicircle. Such a quasicircle separates $\mathbb{C}P^1$ into two discs, both admitting proper discontinuous actions of $\pi_1(S)$ via $\rho(\pi_1(S))$. This means that a quasi-Fuchsian representation induces two (marked) Riemann surface structures on $S$ with opposite orientations. Bers' simultaneous uniformisation theorem \cite{bers_simultaneous} tells us that the space of such pairs of marked Riemann surfaces precisely parametrises the quasi-Fuchsian space, and
\[
\mathcal{QF}(S)=\T(S\cup\bar{S})=\T(S)\times\T(\bar{S})=\T(S)\times\overline{\T(S)}.
\]
In particular, simultaneous uniformisation tells us that this agreement is a biholomorphism, where the former space is endowed with the  complex structure coming from $V(S)$ and the latter is given the usual complex structure on Teichm\"uller space.

Quasifuchsian representations generalise Fuchsian ones where the limit set circle $\mathbb{R}\cup\{\infty\}=\mathbb{R}P^2\subset\mathbb{C}^1=\mathbb{C}\cup\{\infty\}$ has complement naturally regarded as the upper and lower half-planes in $\mathbb{C}$. The actions of a Fuchsian representation on these two half-planes are related by reflection symmetry, hence the two marked Riemann surfaces parametrising such a Fuchsian representation are antiholomorphic. This gives a diagonal embedding of $\T(S)$ in $\mathcal{QF}(S)$ as a maximal-dimensional totally real submanifold.  In particular, this means that the usual complex structure of $\T(S)$, as referenced in the beginning of \cref{sec:wp}, is not complex analytically compatible with the natural complex structure on $\mathcal{QF}(S)\subset V(S)$. However, this diagonal embedding of $\T(S)$ into $\mathcal{QF}(S)$ \emph{is} real analytic.

%\begin{remark}
%For those familiar with quasifuchsian representations, we clarify that
%\[
%\T(S)
%\hookrightarrow
%\mathcal{QF}(S)
%\hookrightarrow
%\mathcal{X}^{\mathrm{irred}}_{\mathrm{para}}(\pi_1(S),\PSL).
%\]
%The first embedding sets Teichm\"uller space as a real-analytic diagonal in quasifuchsian space $\mathcal{QF}(S)$, and the second embedding is the embedding of an open complex manifold within another. However, the closure of $\mathcal{QF}(S)$ is not a topological manifold with boundary (\cite[p.~286]{mcmullen1998complex} and {\color{orange}[I FEEL LIKE Markovic or Gekhtman may have proved the higher dimensional version of this, but I can't find it. I think Yunhui knows about it.]}).
%\end{remark}

\subsection{Quakebends}
The character variety $V(S):=\mathcal{X}^{\mathrm{irred}}_{\mathrm{para}}(\pi_1(S),\SL)$ is a natural parameter space for deformations of Fuchsian representations under \emph{quakebends} --- a holonomy-based complexification of earthquakes introduced by Epstein--Marden \cite{epstein1987convex} via the formalism of \emph{quakebend cocycles}. This is shearing with respect to the following ``tensor product'' of $\mathbb{C}$ and measured lamination space over $\mathbb{R}_{>0}$:
\[
\ml_{\mathbb{C}}(S)
:=
\{
(z,\lambda)\in\mathbb{C}\times\ml(S)
\}
/
(\forall t\in\mathbb{R}_{>0},\ (tz,\lambda)\sim(z,t\lambda)).
\]
Conceptually speaking, to quakebend a marked hyperbolic surface $x\in\T(S)$ with respect to 
\begin{itemize}
\item
a real weighted measured lamination $(t,\lambda)$ for $t\in\mathbb{R}$, is to perform a left earthquake with respect to $t\lambda$ if $t>0$ and to perform a right earthquake if $t<0$;

\item
an imaginary weighted measured lamination $(t\sqrt{-1},\lambda)$ for $t\in\mathbb{R}$, is to bend the universal cover of $x$, regarded as a totally geodesic plane in $\mathbb{H}^3$, to obtain a $\pi_1(S)$-equivariant pleated plane;

\item
a measured lamination with general complex weight, is a combination of earthquaking with respect to the real part of $(z,\lambda)$, followed by bending with respect to the imaginary part.
\end{itemize}
McMullen \cite{mcmullen1998complex} gave an alternative approach for describing quakebends using intermediary deformations called \emph{complex earthquakes} where one (real) earthquakes and grafts flat cylinders \cite{goldman1987projective, hejhal1975monodromy, maskit1969class}, instead of bending by some angle, to produce projective surface structures, and one then goes from a projective structure to a $\PSL$ holonomy representation via \cite{hejhal1975monodromy}.

Complexifying earthquakes in the form of quakebends affords us simultaneous access to the complex-analytic, as well as to the algebraic structure of the representation varieties. McMullen's work \cite[Theorem~2.5 and Proposition~2.6]{mcmullen1998complex} clarifies the structure which we need.

\begin{theorem}
\label{thm:quakebend}
Define the \emph{bending holonomy map}
\[
\eta:
\ml_{\mathbb{C}}(S)\times\T(S)
\to
\mathcal{X}^{\mathrm{irred}}_{\mathrm{para}}(\pi_1(S),\SL)
\]
as the map which takes a complex measured lamination $(z,\lambda)\in\ml_{\mathbb{C}}(S)$ paired with (the character for) a Fuchsian representation $x\in\T(S)$ to (the character for) the induced representation obtained from quakebending $x$ with respect to $(z,\lambda)$. The map $\eta$ is continuous, and it is complex analytic with respect to the $z$ parameter in the complex measured lamination $(z,\lambda)$.
\end{theorem}

%\begin{remark}
%Although the continuity of the map is not explicitly stated in the form of a result in \cite{mcmullen1998complex}, it follows from \cite[Theorem~2.5]{mcmullen1998complex} and the final sentence on p.~292 before the pleated planes subsection in that paper.
%\end{remark}

\subsection{Trace coordinates }

\begin{definition}[trace functions]
\label{defn:trace}
For any closed curve $\alpha$ on $S$, let $\mathbf{a}\in\pi_1(S)$ denote its homotopy class which is determined up to conjugacy. Then the \emph{trace function} 
\[
\trace_\alpha:V(S)
\to
\mathbb{C},
\quad
\rho\mapsto\trace(\rho(\mathbf{a}))
\]
is well defined (i.e.\  independent of the choice of $\mathbf{a}$) and holomorphic.
\end{definition}

\begin{theorem}[trace embedding]
\label{thm:embedding}
The Teichm\"uller space $\T(S)$ embeds in $\mathbb{R}^\Gamma$, for some finite collection of simple closed curves $\Gamma=\gamma_1,\ldots,\gamma_D$ on $S$, via the map
\begin{align*}
\trace^{\Gamma}:\T(S)\to\mathbb{R}^\Gamma,\quad
x\mapsto (\trace_{\gamma_1}(x),\ldots,\trace_{\gamma_D}(x)). 
\end{align*}
For closed surfaces, Schmutz \cite{schmutz1993parametrisierung} and Okumura \cite{okumura1996global} independently show that $D:=\dim_{\mathbb{R}}\T(S)+1$ suffices, and this $D$ is known to be optimal. For punctured surfaces, Hamenst\"adt \cite{hamenstadt2003length} shows that $D=\dim_{\mathbb{R}}\T(S)+1$ still suffices.
\end{theorem}

\begin{remark}
The above embedding is real analytic. See, e.g., \cite[Proposition~6.17]{IT} for an argument for closed surfaces which also applies for punctured surfaces.

%combined with the real analyticity of canonical embedding and projection maps between $\mathbb{R}^{\Gamma}$. 
\end{remark}

\begin{remark}[trace vs. length]
\label{rmk:complexlengths}
Trace functions are more natural when working over $V(S):=\mathcal{X}^{\mathrm{irred}}_{\mathrm{para}}(\pi_1(S),\SL)$ because length functions generally take value in $\mathbb{C}/(2\pi\sqrt{-1}\mathbb{Z})$ rather than $\mathbb{C}$. However, if one is only working over $\mathcal{QF}(S)\subset V(S)$, the contractibility of $\mathcal{QF}(S)$ gives a canonical extension of $\ell_{\gamma}:\T(S)\to\mathbb{R}$ to  $\ell_{\gamma}:\mathcal{QF}(S)\to\mathbb{C}$ via analytic continuation.  This holomorphic function on $\mathcal{QF}(S)$ is called the \textit{complex length} of $\gamma$. 
\end{remark}

\subsection{Cosine formula}

In our study, we shall often need to know the behaviour of geodesic length functions (or trace functions) with respect to earthquakes. Wolpert's cosine formula \cite[Corollary 2.12]{Wolpert1982} describes what happens to geodesic length functions when one performs Fenchel--Nielsen twists, but we shall require Kerckhoff's generalisation. First, we define the total cosine function.

\begin{definition}[{total cosine, see \cite{kerckhoff1985earthquakes}}]
Given two measured geodesic laminations $\lambda$ and $\mu$ on a hyperbolic surface $(S, x)$, their \emph{total cosine}, denoted by $\Cos_x(\lambda, \mu)$, is the integral over the surface $(S,x)$
\[
\Cos_x(\lambda, \mu)
=
\int_{(S,x)}
\cos\theta\ \mathrm{d}\lambda\times \mathrm{d}\mu,
\]
of the cosine of the positive angle $\theta$ made at each intersection point of the two laminations $\lambda$ and $\mu$ when we look at $\mu$ from $\lambda$. The aforementioned integral is taken with respect to the product of the two transverse measures.
\end{definition}

\begin{lemma}[{Kerckhoff's cosine formula, \cite[Proposition 2.5]{kerckhoff1985earthquakes}}]
\label{Kerckhoff cosine}
 
For any measured laminations $\lambda$ and $\mu$ on $(S,x)$, we have
\[
\mathrm{d}\ell_\mu(\v_x(\lambda))
=
\v_x(\lambda)(\ell_{\mu})
=
\Cos_x(\lambda, \mu).
\]
\end{lemma}

\begin{corollary}%[\cite{Wolpert1982}, Theorem 2.11]
\label{reciprocal}
Let $\lambda$ and $\mu$ be two measured laminations.
Then we have 
\[
\mathrm{d}\ell_\mu(\v_x(\lambda))
=
-d\ell_\lambda(\v_x(\mu)).
\]
\end{corollary}

\newpage
\section{The earthquake norm}
\label{s:enorm}

In this section, we first prove that the earthquake metric is a well-defined Finsler metric. This requires the following ingredients:
\begin{itemize}
\item
the earthquake norm ${\|\cdot\|}_e$ on each tangent space $T_x\T(S)$ needs to be a weak norm,
\item
and these norms should vary (at least) continuously with respect to the base point $x\in\T(S)$.
\end{itemize}

The former may be obtained in several ways. We provide first a proof via an infinitesimal duality between the earthquake norm and the co-norm of the Thurston metric (\cref{duality}); we then give a more direct proof via Thurston's original strategy, which relies on the convexity of length functions. The latter of the above two requirements implicitly follows from Kerckhoff's proof of the analyticity of earthquakes \cite{kerckhoff1985earthquakes}. We show later in \cref{thm:lipregular} a stronger result, namely, that the earthquake norm is locally Lipschitz continuous.

\subsection{Infinitesimal duality to the Thurston metric}
\label{s:vsthurston}

Throughout this subsection, although we refer to ${\|\cdot\|}_e$ as the earthquake \emph{norm}, we only use its invariance with respect to multiplication by non-negative real numbers. This is an important point, as the main result of this subsection will be used to give a proof of the fact that ${\|\cdot\|}_e$ is a weak norm. The proof of the main result of this subsection depends on the following key lemma in which we use the earthquake vector notation $\v_x$ for the map defined in \cref{rmk:earthquake-normalized}
and where $\e_x:\mathcal{PML}(S)\to T_x\T(S)$ is the map defined on the quotient space by
\[
\e_x \colon \lambda\mapsto \v_x\left(\tfrac{\lambda}{\len_x(\lambda)}\right)=\tfrac{\v_x(\lambda)}{\len_x(\lambda)}.
\]

\begin{lemma}
\label{lem:WP}
Let $\omega_{\mathrm{WP}}$ be the Weil--Petersson symplectic form on the Teichm\"uller space $\mathcal{T}(S)$ and consider the linear isomorphism  $\Phi_{\mathrm{WP}}$ defined by
\begin{align}
 \Phi_{\mathrm{WP}} \colon T_x\T(S)\to T_x^*\T(S),\quad
v\mapsto \Phi_{\mathrm{WP}}(v):=\omega_{\mathrm{WP}}(\cdot,v).
\end{align}
Then, we have
\begin{align}
\Phi_{\mathrm{WP}}
\left(
\e_x(\lambda)\right)
=(d\log\ell_\lambda)_x.\label{eq:dual}
\end{align}
\end{lemma}
\begin{proof}
For any simple closed curve $\gamma$, Wolpert's duality formula for the Weil--Petersson form \cite[Theorem 2]{Wolpert1982} tells us that the Fenchel--Nielsen twist vector $\left.\tfrac{\partial}{\partial\tau_{\gamma}}\right|_x$ for $\gamma$ satisfies $\Phi_{\mathrm{WP}}(\left.\tfrac{\partial}{\partial\tau_{\gamma}}\right|_x)=(d\ell_\gamma)_x$. Since $\left.\tfrac{\partial}{\partial\tau_{\gamma}}\right|_x$ is precisely the earthquake vector field $\v_x(\gamma)$, normalising by the length of $\gamma$, we have:
\begin{align}
\Phi_{\mathrm{WP}}
\left(
\e_x(\gamma)\right)
=
\Phi_{\mathrm{WP}}
\left(
\tfrac
{\v_x(\gamma)}
{\len_x(\gamma)}
\right)
=
\frac{(\mathrm{d}\ell_\gamma)_x}
{\ell_\gamma(x)}
=
(d\log\ell_\gamma)_x.\label{eq:projembed}
\end{align}

In the general case where $\lambda$ is an arbitrary measured geodesic lamination,
the continuity of $\v_x$ and $\len_x$ ensures that for any sequence of weighted simple closed curves $c_i\gamma_i\to\lambda$, the sequence $\e_x(c_i\gamma_i)=\e_x(\gamma_i)$ tends to $\e_x(\lambda)$. Furthermore, Kerckhoff's generalised cosine formula for the derivative of general measured laminations with respect to general earthquakes (\cref{Kerckhoff cosine}) asserts that for an arbitrary tangent vector $v=\v_x(\mu)\in T_x\T(S)$,
\[
\mathrm{d}\ell_{c_i\gamma_i}(v)\
=
\v_x(\mu)(\ell_{c_i\gamma_i})
=\Cos_x(\mu,c_i\gamma_i).
\]
The continuity of $\Cos_x$ ensures that this converges to $\Cos_x(\mu,\lambda)=\mathrm{d}\ell_{\lambda}(v)$, and thus, $(d\ell_\lambda)_x$ varies continuously with respect to $\lambda$,  hence
\[
\Phi_{\mathrm{WP}}
\left(
\e_x(\lambda)
\right)
=
\lim_{i\to\infty}
\Phi_{\mathrm{WP}}
\left(
\tfrac
{\v_x(\gamma_i)}
{\len_x(\gamma_i)}
\right)
=
\lim_{i\to\infty}
(d\log\ell_{\gamma_i})_x
=
(d\log\ell_\lambda)_x,
\] 
which proves the lemma.
\end{proof}

Recall that the Thurston metric is Finsler. We denote the Thurston conorm on the cotangent space of $\T(S)$ by $\|\cdot \|^*_{\mathrm{Th}}$.\medskip

\noindent\textbf{\cref{duality}} \textrm{(infinitesimal duality)}\textbf{.} \textit{At each point $x \in \T(S)$, there is a linear isometry between the normed spaces $(T_x \T(S), \|\cdot \|_e)$ and ${(T^*_x \T(S), \|\cdot \|^*_{\mathrm{Th}})}$ induced by Weil--Petersson symplectic duality.} 

\begin{proof}
The right-hand side in \cref{eq:dual} is precisely Thurston's expression \cite[Theorem~5.1]{ThM} for the embedding of $\pml(S)$ in $T_x^*\T(S)$ as the Thurston co-norm unit sphere. Thus, $\Phi_{\mathrm WP}$ is a linear isomorphism taking the unit sphere in $T_x \T(S)$ with respect to $\| \cdot \|_e$ to that in $T_x^*\T(S)$ with respect to $\| \cdot \|_{\mathrm{Th}}^*$. This induces the required isometry.
\end{proof}

% Since $\Phi_\omega$ is a linear isomorphism, its inverse takes this unit sphere to the boundary of an open convex subset of the tangent space  $T_x^*\T(S)$ containing $0$, as desired.

%From the proof of the preceding theorem, we deduce the following corollary:

%
%\begin{corollary}
%At each point of the Teichm\"uller space of  $S$, the duality  between the tangent and cotangent spaces established by the  Weil--Petersson symplectic form sends the infinitesimal unit ball of the earthquake metric to the infinitesimal unit co-ball of the Thurston metric.
%\end{corollary}

%
%\begin{corollary}
%Infinitesimal rigidity of the earthquake metric.
%\end{corollary}
%

\subsection{Convexity of $\e_x(\pml(S))$}

Kerckhoff uses his work on lines of minima 
\cite[Theorem 2.2]{kerckhoff1992lines} to show, during the proof of \cite[Proposition~2.6]{kerckhoff1985earthquakes}, the following.

\begin{proposition}
\label{thm:earthquakehomeo}
For every $x\in\T(S)$, the map  
\[
\v_x:\ml(S)\to T_x\T(S),\quad\lambda\mapsto\v_x(\lambda)
\] 
is a homeomorphism.
\end{proposition}

Thurston strengthens the above result as follows.

\begin{theorem}[{\cite[Theorem 5.2]{ThM}}]
\label{Th:Thurston-Ball}
For every $x\in\T(S)$, the map from $\ml(S)$ to $T_x \T(S)$ defined by $\v_x \colon\ml(S)\to T_x\T(S),\ \lambda \mapsto \v_x(\lambda)$ is a homeomorphism onto $T_x\T(S)$. The quotient map $\e_x:\mathcal{PML}(S)\to T_x\T(S)$ defined by
\[
\e_x \colon \lambda\mapsto \v_x\left(\frac{\lambda}{\len_x(\lambda)}\right)=\frac{\v_x(\lambda)}{\len_x(\lambda)}
\]
embeds $\mathcal{PML}(S)$ as a hypersurface which is the boundary of a  convex subset of $T_x\T(S)$ containing the origin at its interior.
\end{theorem}

The additional conclusions needed to promote \cref{thm:earthquakehomeo} to this theorem are equivalent to asserting that the earthquake norm on the vector space $T_x \T(S)$ is a weak norm (see \cref{defn:weaknorm}).
It will turn out that the image of the embedding of $\mathcal{PML}(S)$ into $T_x\T(S)$ given by \cref{Th:Thurston-Ball} is the unit sphere for the earthquake norm.

We give two proofs for \cref{Th:Thurston-Ball}.

%, and explain an alternative proof relying on the fact that at each point of Teichm\"{u}ller space, there is an isometric duality (i.e.\ : \cref{duality}) between the tangent space equipped with the earthquake norm and the cotangent space equipped with the Thurston conorm.%

\noindent\textbf{First proof:} The linear isometry $\Phi_{\mathrm WP}$ in 
\cref{lem:WP,duality} isometrically identifies $\e_x(\pml(S))$ with the unit co-vectors in $T^*_x\T(S)$ with respect to the Thurston norm. We see that the unit ball with respect to the Thurston norm in  ${\|\cdot\|}^*_{\mathrm{Th}}$ is the boundary of an open convex ball in $T^*_x\T(S)$ since the Thurston conorm is a weak norm.
Therefore $\e_x(\pml(S))$ has the same property.\medskip

\noindent\textbf{Second proof:} We flesh out Thurston's sketch of the proof of \cref{Th:Thurston-Ball}. This proof does not depend on the isometry described in \cref{duality}, but does use Kerckhoff's generalisation of Wolpert's cosine formula (\cref{Kerckhoff cosine}) in the guise of the following proposition.

\begin{proposition}
\label{Kerckhoff}
Given an arbitrary $\lambda\in\ml(S)$, let $\mu$ be the unique measured geodesic lamination satisfying $\v_x(\mu)=-\v_x(\lambda)$. Then, 
\[
(\mathrm{d}\ell_\lambda+\mathrm{d}\ell_\mu)_x=0
\quad\text{and }\quad
\Cos_x(\lambda, \mu)=0.
\]
\end{proposition}
\begin{proof}
By \cref{lem:WP}, the condition $\v_x(\lambda)+\v_x(\mu)=0$ is equivalent to 
\[
(\mathrm{d}\ell_\lambda+\mathrm{d}\ell_\mu)_x=0.
\]
Since earthquakes do not change the length of the sheared lamination, we have $(\mathrm{d} \ell_\lambda)(\v_x(\lambda))=0$, and hence $(\mathrm{d}\ell_\mu)_x(\v_x(\lambda))=0$.
By \cref{Kerckhoff cosine}, we have
\[
0
=
(\mathrm{d}\ell_\mu)_x(\v_x(\lambda))
 =
\v_x(\lambda)(\ell_{\mu})
=
\Cos_x(\lambda,\mu),
\]
 hence $\Cos_x(\lambda, \mu)=0$. 
\end{proof}

%We now give our second proof of \cref{Th:Thurston-Ball}.

\begin{proof}[Second proof of \cref{Th:Thurston-Ball}]
Since $\e_x:\pml(S)\to T_x\T(S)$ is a homeomorphism (\cref{thm:earthquakehomeo}), its image is a hypersurface in $T_x\T(S)$. For an arbitrary measured lamination $\lambda$ on $S$, the form $(\mathrm{d}\log \ell_\lambda)_x$ is an element of the cotangent space $T^*_x\T(S)$ (see \cref{embeds into cotangent}).
Since $\mathrm{d}\log\ell(\pml(S))$ is the unit sphere with respect to the conorm ${\|\cdot\|}^*_{\mathrm{Th}}$, it is a convex hypersurface in $T^*_x\T(S)$. Convexity is equivalent to the fact for each $\lambda$ in $\mathcal{ML}(S)$,  there exists a linear function on $T^*_x\T(S)$, \ie an element $v$ of $T_x\T(S)$, which, when restricted to $\mathrm{d}\log \ell(\pml(S))$, attains its maximum at $(\mathrm{d}\log \ell_\lambda)_x$. By \cref{thm:earthquakehomeo}, there is a unique measured lamination $\mu$ such that $v=\v_x(\mu)$. Then by \cref{reciprocal}, we have
\begin{align}
\label{maximum}
\begin{split}
&(\mathrm{d}\log\ell_\lambda)(\v_x(\mu))\\
&=\frac{1}{\len_x(\lambda)}\mathrm{d}\ell_\lambda(\v_x(\mu))
=\frac{-1}{\len_x(\lambda)}\mathrm{d}\ell_\mu(\v_x(\lambda))
=-(\mathrm{d}\ell_\mu)\left(\e_x(\lambda)\right).
\end{split}
\end{align}
Therefore $-(d\ell_\mu)_x$ attains its maximum on $\e_x(\pml(S))$ at $\e_x(\lambda)$. The existence of such a linear function implies the convexity of $\e(\pml(S))$ at $\e_x(\lambda)$. By varying $\lambda$ over all measured laminations, we see that the image $\e_x(\pml(S))$ is a convex hypersurface.

It remains to show that the image of the origin of the vector space $T_x\mathcal{T}(S)$ is contained in the open subset bounded by the image of $\mathcal{PML}(S)$ by the map $\e_x$.
For an arbitrary $\lambda\in\ml(S)\setminus\{0\}$, by taking $\mu$ not disjoint from $\lambda$, each of whose leaves is either disjoint from or nearly parallel to $\lambda$, we ensure that $\v_x(\mu)(\ell_\lambda)\neq0$. By possibly replacing $\mu$ with $\mu'$ such that $\v_x(\mu)=-\v_x(\mu')$, we obtain that $\v_x(\mu)(\ell_\lambda)>0$, hence the maximum value of $(d\ell_\mu)_x$ taken at $\e_x(\lambda)$ in \cref{maximum}, which we denote by $m_\lambda$, is always positive.
Therefore, the hyperplane defined by $(d\ell_\mu)_x=m_\lambda$ has $\e(\pml(S))$ and the origin $\0$ on the same side. Since this holds for all $\lambda\in\ml(S)$, the interior of the convex ball bounded by $\e(\pml(S))$ contains $\0$.
\end{proof}

\newpage
\section{Earthquake regularity via holomorphicity}
\label{s:regularity}

In \cite{kerckhoff1985earthquakes}, Kerckhoff shows that earthquakes $E_{t\lambda}(x)$ are real analytic with respect to both the time parameter $t$ and the spatial parameter $x\in\T(S)$. For our purposes, we require  continuity of its (higher) derivatives with respect to varying $\lambda\in\ml(S)$. We obtain the following results which will be used later in the paper.
\begin{itemize}
\item
\cref{lem:taylorseries}: We derive the Taylor series expansion for earthquakes $E_{t\lambda}(x)$ with respect to time $t$ with coefficients depending continuously on $\lambda\in\ml(S)$ and $x\in\T(S)$, and

\item 
\cref{prop:spatial}: we establish the continuity of the space derivatives, i.e.\  for $x$ varying in $\T(S)$,  of $\v_{\lambda}(x)$ with respect to $\lambda$.
\end{itemize}

We rely on complex analysis to show both results. We apply the Cauchy integral formula to control the derivatives of complex extensions of earthquakes (complexifying either $t$ or $x\in\T(S)$). In \cref{sec:lipregular}, as an application of this improved understanding of the regularity of earthquakes, we show that the earthquake norm is locally Lipschitz continuous.

\subsection{Regularity of time derivatives}\label{sec:proof:preliminaries}

We first deal with the time derivative in some detail to illustrate the argument at hand.

\begin{proposition}[Taylor series for earthquakes]
\label{lem:taylorseries}
For any simple closed curve $\gamma$ on $S$ and for any $x\in\T(S)$, $t\geq0$, and any $\lambda\in\ml(S)$, we have the following ``Taylor series expansion":
\[
\trace_\gamma\left(E_{\frac{t\lambda}{\ell_\lambda(x)}}(x)\right)
=
\trace_\gamma(x)
+
\sum_{k=1}^\infty
\varphi^\gamma_k([\lambda],x)
\cdot
t^k,
\]
where the functions $\varphi^\gamma_k:\pml(S)\times\T(S)\to\mathbb{R}$ are continuous.
%\begin{enumerate}
%\item
%analytic with respect to $\T(S)$,
%\item
%continuous with respect to $\pml(S)$,
%\item
%and their derivatives with respect to $\T(S)$ are also continuous with respect to $\pml(S)$.
%\end{enumerate}
\end{proposition}

\begin{remark}
This proposition ensures that small earthquake segments are uniformly close to ``straight line segments'' in local trace coordinates. We later use this to $C^1$-approximate paths on $\T(S)$ using earthquake segments. 
\end{remark}

\begin{proof}
We know from \cref{thm:quakebend} that the function $F_{\lambda,x}(z):=\eta((z,\lambda),x)$, where $\eta:\ml_{\mathbb{C}}(S)\times\T(S)\to\mathcal{X}^{\mathrm{irred}}_{\mathrm{para}}(\pi_1(S),\SL)$ denotes the bending holonomy map, is holomorphic. Composed with the trace function $\trace_\gamma$ (\cref{defn:trace}), we have an entire function $\mathrm{Tr}_{\lambda,x}:=
\trace_\gamma\circ F_{\lambda,x}:\mathbb{C}\to\mathbb{C}$. Thus, $\mathrm{Tr}_{\lambda,x}$ is equal to its Taylor series (at $z=0$):
\begin{align*}
\mathrm{Tr}_{\lambda,x}(z)
=
\trace_{\gamma}(x)
+\sum_{k=1}^\infty\phi_{\lambda,x,k}\cdot{z^k},
\end{align*}
where the Cauchy integral formula tells us that
\begin{align}
\phi_{\lambda,x,k}
&=
\int_{|w|=1}
\frac{\mathrm{Tr}_{\lambda,x}(w)}{{2\pi\sqrt{-1}}w^{k+1}}\ \mathrm{d}w
=
\int_{|w|=1}
\frac{\trace_\gamma\circ\eta((w,\lambda),x)}{{2\pi\sqrt{-1}}w^{k+1}}\ \mathrm{d}w.
\label{eq:meroint}
\end{align}
\cref{eq:meroint} tells us that, for all $k=1,2,\ldots$, the coefficient $\phi_{\lambda,x,k}$ varies continuously with respect to $\lambda$ and $x$.

Going forward, we restrict the $z$ parameter in $\mathrm{Tr}_{\lambda,x}(z)$ to be a positive real parameter $t\geq0$. In this case, the bending holonomy map $\eta((t,\tfrac{\lambda}{\ell_\lambda(x)}),x)$ is the left earthquake map $E_{\frac{t\lambda}{\ell_\lambda(x)}}(x)$, and we have:
\[
\trace_\gamma\left(E_{\frac{t\lambda}{\ell_\lambda(x)}}(x)\right)
=
\trace_\gamma(x)
+
\sum_{k=1}^\infty
\phi_{\frac{\lambda}{\ell_\lambda(x)},x,k}\cdot {t^k}.
\]
Setting $\varphi_k^\gamma([\lambda],x):=\phi_{\frac{\lambda}{\ell_\lambda(x)},x,k}$ defines a continuous function on $\pml(S)\times\T(S)$, therefore giving us the desired result.
\end{proof}

\begin{lemma}
\label{lem:keylemma}
For any %$\delta\in(0,1)$ and 
compact $K\subset\T(S)$, there exists $C_{\gamma,K}$ such that for all $ t\in[0,1]$,
\[
\left|
\trace_\gamma\left(E_{\frac{t\lambda}{\ell_\lambda(x)}}(x)\right)
-\trace_\gamma(x)
-\varphi^\gamma_1([\lambda],x)t
\right|
\leq
C_{\gamma,K}t^2.
\]
\end{lemma}

\begin{proof}
%This essentially follows from \cref{eq:meroint}, which tells us that
This essentially follows from \cref{eq:meroint}, which tells us that
\begin{align*}
\left|
\varphi_k^\gamma([\lambda],x)
\right|
=
\left|
\phi_{\frac{\lambda}{\ell_\lambda(x)},x,k}
\right|
&= \left|
\int_{|w|=2}
\frac{|\trace_\gamma\circ\eta((w,\frac{\lambda}{\ell_\lambda(x)}),x)}{{2\pi\sqrt{-1}}w^{k+1}}\ \mathrm{d}w\right|\\
&\leq 
\int_{|w|=2}
\frac{|\trace_\gamma\circ\eta((w,\frac{\lambda}{\ell_\lambda(x)}),x)|}{{2\pi|\sqrt{-1}}w^{k+1}|}\ \mathrm{d}|w|\\
&\leq
\int_{|w|=2}
\;\;\;
\sup_{
\stackrel
{|w|=2, \lambda\in\ml(S)}
{x\in K}
}
\frac
{\left|\trace_\gamma\circ\eta((w,\frac{\lambda}{\ell_\lambda(x)}),x)\right|}
{2\pi \cdot 2^{k+1}}\ \mathrm{d}|w|\\
&=
\max_{
\stackrel
{|w|=2, [\lambda]\in\pml(S)}
{x\in K}
}
\frac{\left|
\trace_\gamma\circ\eta((w,\tfrac{\lambda}{\ell_\lambda(x)}),x)
\right|}{2^{k+1}}.
\end{align*}
We set  $$c_{\gamma,K}:=
\max_{
\stackrel
{|w|=2, [\lambda]\in\pml(S)}
{x\in K}
}
\left|
\trace_\gamma\circ\eta((w,\tfrac{\lambda}{\ell_\lambda(x)}),x)
\right|.$$ 
We note that the existence of $c_{\gamma,K}$ is  due to the compactness of the three spaces $\{|w|=2\}\subset\mathbb{C}$, $\pml(S)$ and $K\subset\T(S)$.  This implies \begin{equation*}
	\left|
\varphi_k^\gamma([\lambda],x)
\right|\leq \frac{c_{\gamma,K}}{2^{k+1}}.
\end{equation*}
\cref{lem:taylorseries} then tells us that, for $0 \leq t\leq 1$, 
\begin{align*}
\left|
\trace_\gamma\left(E_{\frac{t\lambda}{\ell_\lambda(x)}}(x)\right)
-\trace_\gamma(x)
-\varphi^\gamma_1([\lambda],x)t
\right|
&\leq
\sum_{k=2}^\infty
\frac{c_{\gamma,K}}{2^{k+1}} {t^k}
=\frac{c_{\gamma,K}}{2^3}t^2\sum_{k=0}\frac{t^k}{2^k}\\
&\leq \frac{c_{\gamma,K}}{2^3}\cdot t^2.
\end{align*}
Taking $C_{\gamma,K}:=\frac{c_{\gamma,K}}{2^3}$ suffices to give the desired result.
\end{proof}
\begin{remark}\label{rmk:right:earthquake:key}
We can extend the domain of $t\in[0,1]$ to $[-1,1]$, where $E_{\frac{t\lambda}{\ell_\lambda(x)}}(x)$ with $t\in[-1,0)$ represents the surface obtained from $x$ via  a \emph{right} earthquake wth respect to $\frac{|t|\lambda}{\ell_\lambda(x)}$.
\end{remark}

\subsection{Regularity of space derivatives}

Having proved the continuity of the time derivatives with respect to $\lambda\in\ml(S)$ in detail, we now appeal to the same arguments, but for the spatial derivatives:

\begin{proposition}\label{prop:spatial}
The earthquake vector field $\v_\lambda:\T(S)\to T\T(S)$ is real analytic, and its partial derivatives of all orders (including the $0$-th) are continuous with respect to $\lambda\in\ml(S)$.  
\end{proposition}

\begin{proof}
We use a part of Kerckhoff's proof of the analyticity of earthquakes \cite[Theorem 1]{kerckhoff1985earthquakes}. For any simple closed curve $\alpha$, Kerckhoff showed \cite[p. 25, 26]{kerckhoff1985earthquakes} that the derivative function $\mathrm{d}\ell_\alpha\circ\v_{\lambda}:\T(S)\to \R$ complex analytically extends to a neighbourhood $\mathcal{N}$ of the Fuchsian locus $\T(S)$ within the quasi-Fuchsian space $\mathcal{QF}(S)$ (see \cref{sec:qf}). The holomorphicity of $\v_\lambda:\mathcal{N}\to T\mathcal{N}\subset T\mathcal{QF}(S)$ then follows from the following two facts: 
\begin{itemize}
\item 
the complex length functions\footnote{\cref{rmk:complexlengths} explains why complex lengths are well defined over $\mathcal{QF}(S)$.} of simple closed curves are holomorphic functions on $\mathcal{QF}(S)$;
\item
$\mathcal{QF}(S)$ is parametrised by the complex length functions of finitely many simple closed curves \cite{kou_complexlength, tan1994}.\end{itemize}

Applying the Cauchy integral formula to the analytic extension ${\mathrm{d}\ell_\alpha\circ\v_{\lambda}:\mathcal{N}\to\mathbb C}$, we see that the partial derivatives (of all orders) of $\mathrm{d}\ell_\alpha\circ\v_{\lambda}$ are continuous with respect to $\lambda$  varying $\ml(S)$. In particular, the restriction of any such derivative to $\T(S)$ is also continuous with respect to $\lambda$ varying in $\ml(S)$. 
\end{proof}

It is perhaps cleaner to frame \cref{prop:spatial} via the following:

\begin{corollary}
The earthquake map $E_\lambda:\T(S)\to\T(S)$ is real analytic, and its partial derivatives of all orders are continuous with respect to $\lambda\in\ml(S)$.
\end{corollary}

\begin{proof}
Combine \cref{prop:spatial} and \cite[Theorem~1]{kerckhoff1985earthquakes}.
\end{proof}

\subsection{Local Lipschitz regularity}
\label{sec:lipregular}

In \cref{sec:infinitesimal}, we make use of Matveev and Troyanov's work \cite{MT2017} on the Myers--Steenrod theorem for Finsler metrics of low regularity to apply norm-wise results (e.g.\  asymmetricity, rigidity) to obtain information about the earthquake distance metric. To do so, we show that the earthquake norm is locally Lipschitz continuous in the following sense:

\begin{definition}[Lipschitz continuity of Finsler norms {\cite[{\S2.1}]{MT2017}}]
A Finsler norm $F$ on a domain $U\subset\mathbb{R}^N$ is \emph{Lipschitz continuous} if the restriction of $F:TU=U\times\mathbb{R}^N\to\mathbb{R}_{\geq0}$ to an \emph{open neighbourhood} of
\[
T^1 U:=\{(x,v)\in U\times \mathbb{R}^N \colon F(x,v)=1\}
\]
is a Lipschitz continuous function with respect to the Euclidean metrics on $TU=U\times\mathbb{R}^N$ and $\mathbb{R}_{\geq0}$.
\end{definition}

\begin{remark}[local Lipschitz continuity]
Although the notion of Lipschitz regularity in our context is ill-defined as there is no canonical Euclidean parameterisation of Teichm\"uller space, \emph{local Lipschitz continuity} is well defined for real analytic, hence locally Lipschitz continuous, coordinate charts on $\T(S)$.
\end{remark}

\begin{theorem}[local Lipschitz continuity of ${\|\cdot\|}_e$]
\label{thm:lipregular}
The earthquake Finsler norm ${\|\cdot\|}_e$ is locally Lipschitz continuous.
\end{theorem}

\begin{proof}%[Proof of \cref{thm:lipregular}]
 Teichm\"uller space can be covered by open balls $\{V\}$ where each $V$ is contained in a compact set $K\subset U$ where $(U,\trace^\Gamma)$ is a trace coordinate chart for the set of simple closed curves $\Gamma$ with $|\Gamma|=\dim \T(S)$. We regard $TU$ as $U\times\mathbb{R}^{6g-6+2n}$, and adopt the following notation:
\begin{itemize}
\item
let $d_1$ denote the Euclidean metric on $U$ (pulled back from ${\trace^\Gamma(U)}$);

\item
let ${\|\cdot\|}_2$ denote the Euclidean norm on the parallelised tangent space;

%$T^{op} U$ denote $\{(x,v)\in TU \colon \|v\|_2\in (0.5,2)\}$, noting that this is a neighborhood of $T^1U$, also define $T^{op}V=TV\cap T^{op}U$;

\item
we often write tangent vectors $\v_x(\lambda)\in T_x U$ in the form $(x,e(x,\lambda))\in U\times\mathbb{R}^{6g-6+2n}$.
\end{itemize}

Let $A$ be a bounded open set in $\mathbb{R}^{6g-6+2n}$ whose closure does not contain the origin such that $T^1 K\subset U\times A$. Our goal is to show that the earthquake norm is Lipschitz on $V\times A$, not with respect to the Euclidean metric, but with respect to a distance defined as:
\[
d_3((x,v),(y,w)):=d_1(x,y)+\|v-w\|_2. 
\]
To see that this is sufficient, we observe that $d_3$ is the Finsler metric for the norm on $U\times\mathbb{R}^{6g-6+2n}$ given by summing the $\ell_1$-norm on the base and the $\ell_2$-norm on the tangent fiber, hence $d_3$ is uniformly bi-Lipschitz to the Euclidean metric. Alternatively, we can choose $V$ to be a convex domain in $U$, and observe that $TV$ is convex for both $d_3$ and the Euclidean metric on $TV$, and invoke \cite[Lemma~2.1]{MT2017}.

Ultimately, we wish to show that there exists a constant $C>0$ such that
for any $\v_x(\lambda),\v_y(\mu)\in V\times A$,
\[\left|
\|\v_x(\lambda)\|_e-\|\v_y(\mu)\|_e
\right|
\leq
C
d_3(\v_x(\lambda),\v_y(\mu)).
\]
Since by the triangle inequality
\[
\left|
\|\v_x(\lambda)\|_e-\|\v_y(\mu)\|_e
\right|
\leq
\left|
\|\v_x(\lambda)\|_e-\|\v_y(\lambda)\|_e
\right|
+
\left|
\|\v_y(\lambda)\|_e-\|\v_y(\mu)\|_e
\right|
,\] 
it suffices to show that for any $\v_x(\lambda),\v_y(\mu)\in V\times A,$
\begin{gather}
\left|\|\v_x(\lambda)\|_e-\|\v_y(\lambda)\|_e\right|<C_1 d_1(x,y)\leq C_1 d_3(\v_x(\lambda),\v_y(\mu))\text{, and}\label{eq:case(1)}\\
\left|\|\v_y(\lambda)\|_e-\|\v_y(\mu)\|_e\right|
<C_2\|e(y,\lambda)-e(y,\mu)\|_2
<C_3 d_3(\v_x(\lambda),\v_y(\mu)).\label{eq:case(2)}
\end{gather}
 
 We first verify \cref{eq:case(1)}. Invoking the definition of the earthquake norm, we see that
\begin{align}
\frac{\left|
\|\v_x(\lambda)\|_e-\|\v_y(\lambda)\|_e
\right|}
{d_1(x,y)}
=
\left|
\frac{\ell_\lambda(x)-\ell_\lambda(y)}
{d_1(x,y)}
\right|.\label{eq:case1}
\end{align}
\cref{eq:case(1)} then follows from the fact that length functions are locally Lipschitz. Indeed, let $\dth(x,y)$ be the Thurston distance from $x$ to $y$. Without loss of generality, we may assume that $\ell_\lambda(x)\leq \ell_\lambda(y)$. Then
\begin{eqnarray*}
	|\ell_\lambda(x)-\ell_\lambda(y)|&=&\ell_\lambda(x)\left(\frac{\ell_\lambda(y)}{\ell_\lambda(x)}-1\right)\\
	&\leq& \ell_\lambda(x) (e^{d_{Th}(x,y)}-1)\leq C\cdot \ell_\lambda(x) d_{Th}(x,y)
\end{eqnarray*}
for some constant $C$ depending on $K$ and $A$.  Combined with the fact that $d_{Th}(x,y)$ and $d_1(x,y)$ are locally bi-Lipschitz to each other, this implies \cref{eq:case(1)}.

We next show \cref{eq:case(2)}. The triangle inequality and obtain that
\begin{align*}
\|\v_y(\lambda)\|_e
&\leq
\|\v_y(\mu)\|_e+\|\v_y(\lambda)-\v_y(\mu)\|_e
\ \text{and}\\
\|\v_y(\mu)\|_e
&\leq
\|\v_y(\lambda)\|_e+\|\v_y(\mu)-\v_y(\lambda)\|_e,
\end{align*}
and hence 
\begin{align}
\left|
\|\v_y(\lambda)\|_e
-\|\v_y(\mu)\|_e
\right|
\leq
\max\{
\|\v_y(\mu)-\v_y(\lambda)\|_e,
\|\v_y(\lambda)-\v_y(\mu)\|_e
\}.\label{eq:3.10.1}
\end{align}
By the compactness of $T^1K$, there is a constant $C_2$ such that
\begin{align}
\sup_{(y,u)\in TV}
\frac{\|(y,u)\|_e}{\|u\|_2}
\leq
\sup_{(y,u)\in TK}
\frac{\|(y,u)\|_e}{\|u\|_2}
=
\max_{(y,\hat{u})\in T^1K}
\|(y,\hat{u})\|_e
=C_2.\label{eq:3.10.2}
\end{align}
Combining \cref{eq:3.10.1} and \cref{eq:3.10.2} with $u=e(y, \lambda)-e(y,\mu)$ or $e(y,\mu)-e(y,\lambda)$, we get
\begin{align*}
\left|
\|\v_y(\lambda)\|_e
-
\|\v_y(\mu)\|_e
\right|
&\leq
\max\{\|(y, e(y, \lambda)-e(y,\mu))\|_e, \|(y, e(y,\mu)-e(y,\lambda))\|_e\}\\ &\leq
C_2\|e(y,\lambda)-e(y,\mu)\|_2,
\end{align*}
which gives the first inequality in \cref{eq:case(2)}. For the latter inequality, we use the triangle inequality and obtain that
\begin{align*}
\|e(y,\lambda)-e(y,\mu)\|_2
&\leq
\|e(x,\lambda)-e(y,\mu)\|_2
+
\|e(y,\lambda)-e(x,\lambda)\|_2\\
&\leq
d_3(\v_x(\lambda),\v_y(\mu))
+
\|e(y,\lambda)-e(x,\lambda)\|_2,
\end{align*}
where the second inequality follows from the definition of $d_3$.
\cref{prop:spatial} establishes
\begin{itemize}
\item
the differentiability of the partial derivatives of $\v_\lambda(x)$ with respect to $x\in V\subset K$, and
\item
the continuity of these derivatives with respect to $\lambda$.
\end{itemize}
Thus, these regularity properties also apply to $e(x,\lambda)$. From this we infer that there exists a constant $C_3'$ such that  for all $x,y\in V\subset K$ and for all $e(x, \lambda), e(y,\lambda) \in A$, 
\begin{align}
\|e(y,\lambda)-e(x,\lambda)\|_2
\leq 
C_3' d_1(x,y). 
\end{align}
Since $d_1(x,y)\leq d_3(\v_x(\lambda),\v_y(\mu))$ (which follows from the definition), we have \cref{eq:case(2)} with $C_3=1+C_3'$.  
\end{proof}

\newpage
\section{Magnitude minimisation}
\label{sec:magnitude}
The main goal of this section is to demonstrate the ``magnitude minimisation'' interpretation of the earthquake metric (see \cref{defn:earthquakemetric}) as follows.

\smallskip
\noindent\textbf{\cref{thm:magnitude}} (magnitude minimisation)\textbf{.} \textit{For any $x,y\in\T(S)$, the earthquake distance $d_e(x,y)$ is equal to the infimal magnitude over the collection of piecewise earthquake paths from $x$ to $y$.}

%This is a fairly technical section, and we advise first-time readers to skip ahead and return after accruing greater familiarity with the objects at hand.
%YH: I think that the previous sections are also now sufficiently technical that it doesn't make as much sense to single this one out.

\subsection{Proof strategy for \cref{thm:magnitude}}
\label{sec:proof:magnitude}

Given two arbitrary points $x,y\in \mathcal{T}(S)$, the distance $d_e(x,y)$ is realised as the infimum of the earthquake metric lengths of a sequence $(p_j:I\to\T(S))_{j\in\mathbb{N}}$ of piecewise $C^1$-paths starting at $x$ and ending at $y$. We may assume that $p_j$ is $C^1$ by approximating each of them by a $C^1$-path.

To prove \cref{thm:magnitude}, it suffices to take an arbitrary $C^1$ path $p:[0,L]\to\T(S)$, and construct a sequence of piecewise earthquake paths which have the same endpoints as $p$, and whose magnitude tends to the earthquake metric length $\int_0^L\|p'(t)\|_e\mathrm{d}t$ of the path $p(t)$. We take the following steps:
\begin{enumerate}  
\item
In \cref{lem:pathcoords}, we show that we may reduce to dealing with a path $p$ with nowhere zero derivative over a local trace coordinate chart $(U,\trace^\Gamma)$. This is not essential, but simplifies notation and avoids some bi-Lipschitz comparison arguments.

\item
We break the path into \emph{very} short segments and approximate them (see \cref{fig:pg29}) first by Euclidean segments (\cref{lem:lineapprox1}) and then by short earthquake segments (\cref{lem:lineapprox2}) sharing the same starting point and initial velocity. We choose the segments to be short enough to ensure that the endpoints of the approximating earthquake segments remain close to $p([0,L])$.

\item
We show, using \cref{lem:taylorseries}, that the endpoints of the approximating earthquake segments in the previous step can be linked up with small earthquake segments to form a path 
of earthquake ``sawtooth waves'' (see \cref{fig:pg31}), so that the total contribution of the shorter ``vertical segments'' make arbitrarily small contributions to the total magnitude as the approximation becomes arbitrarily fine.

\end{enumerate}

\begin{remark}
\label{rmk:arclength}
Since magnitude and path lengths with respect to Finsler metrics are all independent of the parametrisation of a path, we assume henceforth that $p(t)$ is parametrised by earthquake metric arclength, i.e.\  $\|p'(t)\|_e=1$. This means that the $L$ in the domain of $p$ is necessarily the earthquake metric length of the path $p(t)$. 
\end{remark}

\subsection{Step~1: reduction to working on charts}

Let $\Gamma=\{\gamma_1,\ldots,\gamma_D\}$, where $D=\dim_{\mathbb{R}}\T(S)+1$ be a collection of curves whose holonomy traces specify an embedding of $\T(S)$ as in \cref{thm:embedding}. 

\begin{lemma}
\label{lem:pathcoords}
Any $C^1$-path $p$ in $\T(S)\hookrightarrow\mathbb{R}^{\Gamma}$ with nowhere zero derivative can be expressed as a concatenation of finitely many paths $p_1,\ldots,p_m$ such that there exists a sequence of $C^1$-local coordinates $\trace^{\Gamma_j}:V_j\to\mathbb{R}^{\Gamma_j}$, where $\Gamma_j\subset\Gamma$, such that $\trace^{\Gamma_j}\circ p_j$ defines a path with nowhere zero derivative.
\end{lemma}

\begin{proof}
Define $\Gamma_i:=\Gamma-\{\gamma_i\}$. For any point in $\T(S)$, the derivatives $(\mathrm{d}\trace^{\Gamma_i})_x$ cannot all be singular, or else the normal vector of $(\mathrm{d}\trace^{\Gamma})_x$ would lie in the kernel of every projection map $\mathbb{R}^{\Gamma}\to\mathbb{R}^{\Gamma_i}$,  and hence would be equal to $0$. Thus, we can cover $\T(S)$ by  $C^1$-coordinate charts which map to $\mathbb{R}^{\Gamma_i}$, $i=1,\ldots,D$. Now cover $p([0,L])$ by these charts. Thanks to the compactness of $p([0,L])$, there are finitely many charts $(V_j,\trace^{\Gamma_{i_j}})_{j=1, \dots , m}$ and a subdivision $[t_0=0,t_1], \dots , [t_{m-1}, t_m=L]$ of $[0,L]$ such that $p([t_{j-1}, t_j])$ is contained in $V_j$. The result follows from relabelling the $\Gamma_{i_j}$ as $\Gamma_j$.
\end{proof}

\begin{remark}
\cref{lem:pathcoords} allows us to work in local coordinates in $\mathbb{R}^{\dim_{\mathbb{R}}\T(S)}$ to approximate subsegments like $\trace^{\Gamma_j}\circ p_j$ in trace function-based local coordinates.  This  simplifies notation and allows  us to use  \cref{lem:taylorseries} to understand the behaviour of short earthquake segments.
\end{remark}

\subsection{Step~2: approximating via earthquake segments}

Since we may approximate each path $\trace^{\Gamma_j}\circ p_j$ individually, \cref{lem:pathcoords} reduces the problem to approximating a single path $p(t)$ over a single trace coordinate chart $(U,\trace^{\Gamma})$. For the remainder of the proof, let 
\[
\pi:[0,L]\to\mathbb{R}^{\Gamma},\quad\pi(t):=\trace^{\Gamma}\circ p(t)
\] 
denote the corresponding image of the path $p(t)$ in trace coordinates.

\begin{lemma}
\label{lem:lineapprox1}
For each $\epsilon>0$, we can subdivide the interval $[0,L]$ into sub-intervals of the form $[s_i,s_{i+1}]$, where
\[
0=s_0<s_1<s_2<\ldots<s_{m-1}<s_m=L,
\]
so that for $i=0,\ldots,m-1$, and for all $s\in[s_i,s_{i+1}]$, we have
\begin{align}
\|\pi(s)-\pi(s_i)-(s-s_i)\pi'(s_i)\|_{\mathbb{R}^{\Gamma}}
<
\tfrac{\epsilon}{2}(s-s_i).
\label{eq:lineapprox1}
\end{align}
\end{lemma}

\begin{center}
\begin{figure}
\includegraphics[scale=0.5]{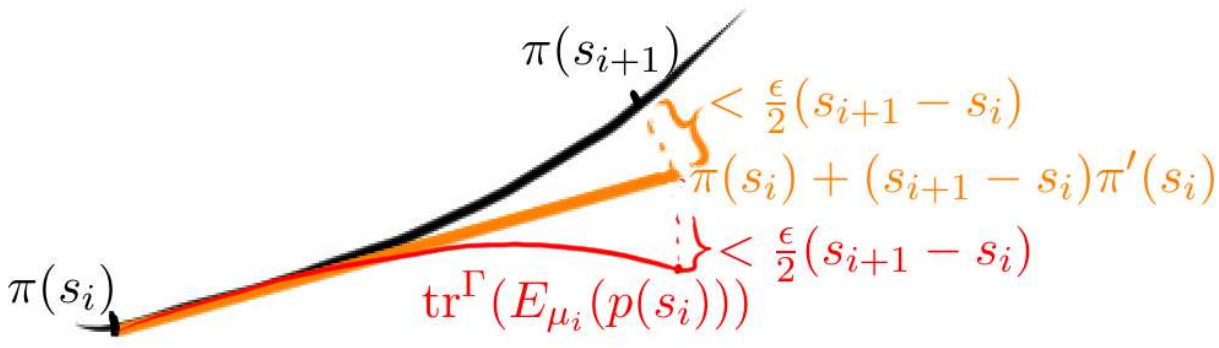}
\caption{A segment of $\pi(t)$ between $t=s_i$ and $t=s_{i+1}$, the {approximating linear segment in orange}, and the {approximating earthquake segment in red}.}
\label{fig:pg29}
\end{figure}
\end{center}

\begin{proof}
Since $\pi(t)$ is $C^1$, its derivative $\pi':[0,L]\to\mathbb{R}^{\Gamma}$ is continuous. Thus, we may cover $[0,L]$ with open intervals on which $\pi'(t)$ is nearly constant. The compactness of $[0,L]$ allows us to reduce to a (minimal) finite covering, and choosing $(s_i)_{i=0,\ldots,m}$ to lie on overlaps of successive intervals ensures that
\begin{align}
\text{for all } s\in[s_i,s_{i+1}],
\quad 
\|\pi'(s_i)-\pi'(s)\|_{\mathbb{R}^{\Gamma}}
<\tfrac{\epsilon}{2}.
\label{eq:line1condition}
\end{align}
Since $\pi(s)
=
\pi(s_i)
+
\int_{s_i}^{s}
\pi'(t)\; \mathrm{d}t$
and 
$(s-s_i)\pi'(s_i)=\int_{s_i}^{s}\pi'(s_i)\ \mathrm{d}t$, we have
\begin{align*}
\left\|
\pi(s)-\pi(s_i)-(s-s_i)\pi'(s_i)
\right\|_{\mathbb{R}^{\Gamma}}
&=
\left\|
\int_{s_i}^{s}
\pi'(t)
\; \mathrm{d}t
-
\int_{s_i}^{s}
\pi'(s_i)
\; \mathrm{d}t
\right\|_{\mathbb{R}^{\Gamma}}
\\
&\leq
\int_{s_i}^{s}
\|\pi'(t)-\pi'(s_i)\|_{\mathbb{R}^{\Gamma}}
\; \mathrm{d}t
<\tfrac{\epsilon}{2}(s-s_i).
\end{align*}
\end{proof}

\begin{lemma}
\label{lem:lineapprox2}
For each $\epsilon>0$, we can subdivide the interval $[0,L]$ into finer intervals of the form $[s_i,s_{i+1}]$, where
\[
0=s_0<s_1<s_2<\ldots<s_{m-1}<s_m=L,
\]
so that, for every $i=0,\ldots,m-1$, there is some $\mu_i\in\ml(S)$ such that
\begin{enumerate}[(i)]
\item
$\ell_{\mu_i}(p(s_i))=1$, and
%the magnitude $\ell_{\mu_i}(p(s_i))$ of the earthquake path $E_{t\mu_i}(p(s_i))$, $t\in[s_i,s_{i+1}]$, satisfies
%\[
%(s_{i+1}-s_i)\ell_{\mu_i}(p(s_i))
%=
%\int_{s_i}^{s_{i+1}}\|p'(t)\|_e\;\mathrm{d}t,
%\]

\item
$\|\trace^{\Gamma}(E_{(s_{i+1}-s_i)\mu_i}(p(s_i)))
%-\trace^{\Gamma}(p(s_{i+1}))
-\pi(s_{i+1})
\|_{\mathbb{R}^{\Gamma}}
<
\epsilon(s_{i+1}-s_i)$.

\end{enumerate}
\end{lemma}

\begin{proof}
We choose the division points $(s_i)$ to satisfy the following constraints: for all $ i=0,\ldots, m$, 
\begin{enumerate}[(a)]
\item
$s_{i+1}-s_i<\min(\frac{\epsilon}{2C_{\Gamma,K}},1)$, where $K\subset\T(S)$ is a compact set containing the image of the path $\pi(t)$, and the constant $C_{\Gamma,K}$ is defined to be $\sum_{\gamma\in\Gamma} C_{\gamma,K}$ with $C_{\gamma,K}$ constructed as in \cref{lem:keylemma};

\item
for any $s\in[s_i,s_{i+1}]$,  we have $\|\pi'(s_i)-\pi'(s)\|_{\mathbb{R}^{\Gamma}}
<\tfrac{\epsilon}{2}$ --- this is just \cref{eq:line1condition}.

\end{enumerate}
We already argued in the proof of \cref{lem:lineapprox1} that condition $(b)$ is possible to enforce, and will yield \cref{eq:lineapprox1} as a consequence. The condition~{(a)} is clearly possible. 

We now explain the choice of $\mu_i$ for each interval $[s_i,s_{i+1}]$: the infinitesimal earthquake map $\v_{(\cdot)}(p(s_i)):\ml(S)\to T_{p(s_i)}\T(S)$ is a homeomorphism (see \cref{thm:earthquakehomeo}, extracted from the proof of \cite[Proposition~2.6]{kerckhoff1985earthquakes}), and we choose $\mu_i$ to be the unique measured lamination $\mu_i\in\ml(S)$ such that $\v_{\mu_i}(p(s_i))=p'(s_i)$. Immediate consequences of this choice of $\mu_i$ are the following.
\begin{itemize}
\item
Since $\|p'(s_i)\|_e=1$ (see \cref{rmk:arclength}),  we have
\[
\|\v_{\mu_i}(p(s_i))\|_e=\ell_{\mu_i}(p(s_i))=1,
\]
which already gives us conclusion~{(i)}.
\item
Moreover,
\[
\left.\tfrac{\mathrm{d}}{\mathrm{d}t}\right|_{t=0}
\trace^{\Gamma}(E_{t\mu_i}(p(s_i)))
=
\mathrm{d}\trace^\Gamma(\v_{\mu_i}(p(s_i)))=\mathrm{d}\trace^\Gamma(p'(s_i))=\pi'(s_i),
\]
and combining this with \cref{lem:taylorseries}, we see that
\begin{align*}
\pi'(s_i)
=
\left(
\varphi^{\gamma}_1([\mu_i],p(s_i))
\right)_{\gamma\in\Gamma}\in\mathbb{R}^{\Gamma}.
\end{align*}
\end{itemize}

We now prove conclusion~{(ii)} using conditions~{(a)} and {(b)}. By \cref{eq:lineapprox1} with $s=s_{i+1}$ and the triangle inequality, it suffices to show that
\begin{align*}
\|\trace^{\Gamma}(E_{\mu_i}(p(s_i)))-\pi(s_i)-(s_{i+1}-s_i)\pi'(s_i)\|_{\mathbb{R}^{\Gamma}}
<\tfrac{\epsilon}{2}(s_{i+1}-s_i).
\end{align*}

The condition $(a)$ means that we may invoke \cref{lem:keylemma} to assert that for all $\gamma \in \Gamma$,
$$\left|
\trace_\gamma\left(E_{t\mu_i}(p(s_i))\right)
-\trace_\gamma(p(s_i))
-\varphi^\gamma_1([\mu_i],p(s_i))\; t
\right|
<
C_{\gamma,K}t^2.
$$
We finish the proof by applying the triangle inequality (see \cref{fig:pg29}) to isolate each coordinate:
\begin{align*}
&\left\|
\trace^{\Gamma}
\left(
E_{(s_{i+1}-s_i)\mu_i}(p(s_i))
\right)-\pi(s_i)-\pi'(s_i)(s_{i+1}-s_i)
\right\|_{\mathbb{R}^{\Gamma}}\\
\leq&
\sum_{\gamma\in\Gamma}
\left|
\trace_\gamma\left(E_{(s_{i+1}-s_i)\mu_i}(p(s_i))\right)
-\trace_\gamma(p(s_i))
-\varphi^\gamma_1([\mu_i],p(s_i))(s_{i+1}-s_i)
\right|\\
<&
\sum_{\gamma\in\Gamma}
C_{\gamma,K}(s_{i+1}-s_i)^2
\leq
(\sum_{\gamma\in\Gamma}C_{\gamma,K})
\cdot\tfrac{\epsilon}{2C_{\Gamma,K}}
\cdot(s_{i+1}-s_i)
=\tfrac{\epsilon}{2}(s_{i+1}-s_i).
\end{align*}
\end{proof}

\subsection{Step~3: closing up the earthquake segments}

\begin{lemma}
\label{lem:shortapprox} 
There exist  constants $\delta>0$ and  $M_\Gamma$ such that for any $s\in[0,L]$ and any  $q\in \overline{B}_{\mathbb{R}^{\Gamma}}(\pi(s),\delta)$, there exists a unique $\mu\in\ml(S)$ such that for  $\tilde{q}:=(\trace^{\Gamma})^{-1}(q)$, 
we have  
\begin{enumerate}
	\item $\trace^{\Gamma}\left(E_{\mu}(\tilde{q})\right)=\pi(s)$, i.e.\  $ E_\mu(\tilde q)=p(s)$, and 
	\item $\ell_{\mu}(\tilde q)\leq M_\Gamma \|\pi(s)-q\|_{\mathbb{R}^{\Gamma}}$.
\end{enumerate}
\end{lemma}

\begin{proof}
We first make the observation that if we use $E^{\#}_{t\lambda}(\cdot)$ to denote right earthquakes with respect to $\lambda$, then whatever statement that we wish to understand regarding left earthquaking from points $\tilde{q}$ to $p(s)=E_{t\lambda}(\tilde{q})$ can be rephrased in terms of right earthquakes with respect to $\lambda$ deforming from $p(s)$ to $\tilde{q}=E^{\#}_{t\lambda}(p(s))$. Indeed, since the initial choice of using left earthquakes was arbitrary, we could have developed everything using right earthquakes from the very beginning of this work and obtained a version of the Taylor series \cref{lem:taylorseries} for right earthquakes. 
%Another way of seeing this is to use the fact that the bending holonomy map $\eta$ used in the proof of \cref{lem:taylorseries} naturally allows for $((-t,\lambda),x)$ as input, and so we use notation of the form $\varphi_k^{\gamma}(-[\lambda],x)$ as the relevant analytic continuation of the $\varphi^\gamma_k([\lambda],x)$ coefficients in \cref{lem:taylorseries}.

We clarify our choice of $\delta$ and $M_\Gamma$. For $\lambda\in\ml(S)$ normalised so that $\ell_\lambda(p(s))=1$.
Notice that $E^\#_{t\lambda}(p(s))=E_{-t\lambda}(p(s))$ for $t>0$. 
 Hence, \begin{align*}
&\left\|
\trace^{\Gamma}(E^{\#}_{t\lambda}(p(s)))
-
\pi(s)
\right\|_{\mathbb{R}^\Gamma}=\left\|
\trace^{\Gamma}(E^{\#}_{t\lambda}(p(s)))
-
\trace^\Gamma \gamma(p(s))
\right\|_{\mathbb{R}^\Gamma}\\
=& \left\|
\trace^{\Gamma}(E_{-t\lambda}(p(s)))
-
\trace^\Gamma \gamma(p(s))
\right\|_{\mathbb{R}^\Gamma}\\
\geq & a_{|\Gamma|}\sum_{\gamma\in\Gamma}\left\|
\trace^{\gamma}(E^{\#}_{t\lambda}(p(s)))
-
\trace^\gamma \gamma(p(s))
\right\|_{\mathbb{R}^\Gamma}
\\
\geq & a_{|\Gamma|}\sum_{\gamma\in\Gamma}\left(
 \left|\varphi^\gamma_1([\lambda],p(s)) t
\right|-\left|
\trace^{\gamma}(E^{\#}_{t\lambda}(p(s)))
-
\pi(s)-\varphi^\gamma_1([\lambda],p(s)) t
\right|\right)
\\
\geq & a_{|\Gamma|}\sum_{\gamma\in\Gamma}\left(  t |\varphi^\gamma_1([\lambda],p(s)) |- t^2
 C_{\gamma,K}\right),
\end{align*}
where $a_{|\Gamma|}>0$ is a uniform constant used to bound the Euclidean norm below by the $L^1$-norm in $\mathbb{R}^{\Gamma}$ and where the last inequality follows from Lemma \ref{lem:keylemma} and Remark \ref{rmk:right:earthquake:key}. %, and {\color{red}$K$ is taken to be the image $p([0,L])$ of the path $p$.}\YH{poor exposition} 
Using compactness, we define another constant
\begin{align*}
b_\Gamma
:=
\min_{
\stackrel{[\lambda]\in\pml(S)}{s\in[0,L]}
}
\sum_{\gamma\in\Gamma}
|\varphi_1^\gamma([\lambda],p(s))|
>0,
\end{align*}
where the positivity follows from Lemma \ref{lem:pathcoords}.

We next choose $\epsilon_0>0$ subject to the following conditions:
\begin{itemize}
\item $\displaystyle\epsilon_0\leq\min(1,\frac{b_\Gamma}{2\sum_{\gamma\in\Gamma}c_{\gamma,K}})$

\item
$\epsilon_0$ is small enough so that the Euclidean $\epsilon_0$-neighbourhood $\overline{B}_{\mathbb{R}^{\Gamma}}(\pi([0,L]),\delta)$ of the image $\pi([0,L])$ of the path $\pi$ is contained within the image $\trace^{\Gamma}(U)\subset\mathbb{R}^{\Gamma}$.

\end{itemize}
The latter condition is a minor technicality needed for the use of auxiliary coordinate charts in our proof strategy. As a consequence, for every  $t\leq \epsilon_0$, 
\begin{align}
t\leq
M_\Gamma\left\|
\trace^{\Gamma}(E^{\#}_{\varepsilon\lambda}(p(s)))
-
\pi(s)
\right\|_{\mathbb{R}^\Gamma}
,\text{ where }M_\Gamma=\frac{2}{a_{|\Gamma|}b_\Gamma}.
\label{eq:upperbound}
\end{align}

We are now equipped for the proof of \cref{lem:shortapprox}. We know from \cite[Appendix Theorem~2]{Ker} that for every point $p(s)$ along the given path $p$, the right earthquake map
\[
E^{\#}_{(\cdot)}(p(s)):\ml(S)\to\T(S),\quad\lambda\mapsto E^{\#}_{\lambda}(p(s))
\]
is a homeomorphism. 
We specify that we choose $\delta>0$ to be any positive constant  less than or equal to $\epsilon_0/M_\Gamma$.
By setting $t=M_\Gamma \delta$ in \cref{eq:upperbound}, we see that the boundary sphere of the ball,
\[
\ml_{\leq \delta}(p(s))
:=
\left\{
\mu\in\ml(S)
\mid
\ell_\mu(p(s))
\leq M_{\Gamma}\delta
\right\},
\]
maps to a sphere lying outside the Euclidean open ball ${B}_{\mathbb{R}^{\Gamma}}(\pi(s),\delta)$, and hence the image of $\ml_{\leq \delta}(p(s))$ under $\trace^{\Gamma}\circ E^{\#}_{(\cdot)}(p(s))$ contains $\overline{B}_{\mathbb{R}^{\Gamma}}(\pi(s),\delta)$. This means that for any $s\in[0,L]$ and any $q\in \overline{B}_{\mathbb{R}^{\Gamma}}(\pi(s),\delta)$, there exists some $\mu\in\ml(S)$ such that 
$
\trace^{\Gamma}(E^{\#}_{\mu}(p(s)))=q.$
Moreover, by \eqref{eq:upperbound}, where $\lambda$ is assumed to be of unit length, we see that
$$0\leq\ell_\mu(p(s))\leq M_\Gamma\|\pi(s)-q\|_{\mathbb{R}^{\Gamma}}.$$
Equivalently, we have, for $\widetilde{q}=(\trace^{\Gamma})^{-1}(q)$,  
\[
\trace^{\Gamma}(E_{\mu}(\widetilde{q}))=\pi (s)\text{ and }\ell_\mu(\widetilde{q})\leq M_\Gamma\|\pi(s)-q\|_{\mathbb{R}^{\Gamma}}.
\] 
The uniqueness of $\mu$ is due to $E_{(\cdot)}(\widetilde{q})$ being a homeomorphism.
\end{proof}

\subsection{Proving the magnitude minimisation interpretation}

\begin{proof}[Proof of \cref{thm:magnitude}]
Given $\epsilon>0$, we approximate $\pi:[0,L]\to\mathbb{R}^{\Gamma}$ by forming a sawtooth-like piecewise earthquake path as follows:
\begin{itemize}
\item
break up the interval $[0,L]$ into intervals $[s_i,s_{i+1}]_{i=0,\ldots,m-1}$ subject to \cref{lem:lineapprox2} with the given $\epsilon$;

\item
replace each segment $\pi([s_i,s_{i+1}])$ by two earthquake segments --- the first segment is given by (see \cref{lem:lineapprox2})
\[
\trace^{\Gamma}(E_{t(s_{i+1}-s_i)\mu_i}(p(s_i))),\quad t\in[0,1],
\]
and the second segment joins the endpoint $\trace^{\Gamma}(E_{(s_{i+1}-s_i)\mu_i}(p(s_i)))$ to $\pi(s_{i+1})$ using \cref{lem:shortapprox}.
\end{itemize}

\begin{center}
\begin{figure}[ht]
\includegraphics[scale=0.4]{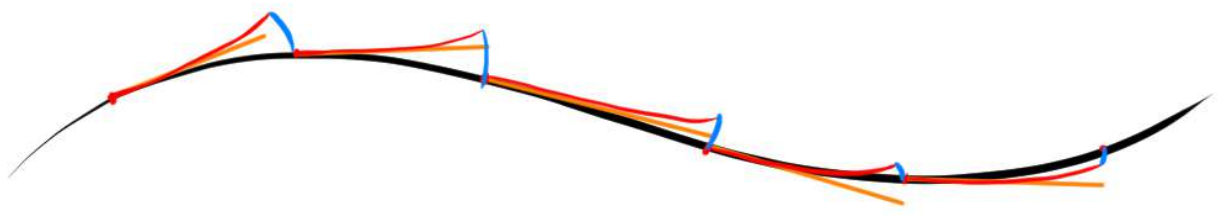}
\caption{Sawtooth waves (red and blue) which approximate the desired $C^1$-path $\pi(t)$.}
\label{fig:pg31}
\end{figure}
\end{center}

The earthquake metric length of the first segment is $s_{i+1}-s_i$, which is precisely the length of the segment $p([s_i,s_{i+1}])$. Thanks to \cref{lem:lineapprox2} (ii) and \cref{lem:shortapprox} (ii), the length of the second segment is less than ${\epsilon M_\Gamma(s_{i+1}-s_i)}$. Therefore, the earthquake metric length of the approximating piecewise earthquake path is 
$\displaystyle
\sum_{i=1}^{m-1}(s_{i+1}-s_i)=L$
 up to an error smaller than 
$\displaystyle\sum_{i=1}^{m-1}\epsilon M_\Gamma(s_{i+1}-s_i)=\epsilon M_\Gamma L$.
As $\epsilon\to0$, the length of the approximation tends to $L$, as desired.
\end{proof}

\newpage

\section{Mapping class group rigidity and earthquake metric asymmetry}
\label{sec:infinitesimal}

We establish the infinitesimal rigidity of the earthquake metric in two different ways. The first proof relies on \cref{duality} and the infinitesimal rigidity of Thurston's asymmetric metric established in \cite{HOP,Pan}. The second proof makes use of a rigidity result obtained in \cite{OP} concerning homeomorphisms of measured laminations space which preserve zero intersection number, which establishes ``topological rigidity'' (\cref{topological rigidity}). The remainder of the proof follows the third author's argument in \cite[Theorem 3.1]{Pan}. We then employ infinitesimal rigidity of the earthquake metric to show that $d_e$ is rigid with respect to the mapping class group and that it is asymmetric as a metric.

\subsection{Flats in the unit sphere of the earthquake norm}

We give a characterisation of pairs of measured laminations with zero intersection number, and show that this also characterises pairs of measured laminations in the unit norm ball of the earthquake Finsler metric which lie on the same flat.

\begin{theorem}[flatness characterisation]
\label{segment}
Let $\lambda$ and $\mu$ be two measured  laminations on $S$.
Then  the following three conditions are equivalent.
\begin{enumerate}[(1)]
\item
The equality $\| t \v_x(\lambda)+(1-t) \v_x(\mu)\|_e=t\| \v_x(\lambda)\|_e+(1-t)\| \v_x(\mu)\|_e$ holds for every $t \in (0,1)$.
\item
The same equality holds for some $t \in (0,1)$.
\item
 $i(\lambda, \mu)=0$.
 \end{enumerate}
 \end{theorem}
\begin{proof}
The implication from \textit{(1)} to \textit{(2)} is logically trivial.

We show that \textit{(3)} implies \textit{(1)}.
Suppose that $i(\lambda, \mu)=0$. For every $t\in (0,1)$,  the two measured laminations $t\lambda$ and $(1-t)\mu$ do not intersect transversely, hence $t\lambda+(1-t)\mu$ is a well-defined measured lamination. Indeed, $t\v_x(\lambda)+(1-t)\v_x(\mu)=\v_x(t\lambda+(1-t)\mu)$ because \cref{Kerckhoff cosine} tells us that the two sides of this identity evaluated on $\ell_\gamma$ agree for every simple closed curve $\gamma$.
It follows that 
\begin{align}
\| \v_x(t\lambda+(1-t)\mu) \|_e&=\len_x(t\lambda+(1-t)\mu)\notag\\
&=t\len_x(\lambda)+(1-t)\len_x(\mu)=t\| \v_x(\lambda)\|_e +(1-t)\| \v_x(\mu)\|_e,
\end{align}
and we are done.

We next show that \textit{(2)} implies \textit{(3)} by proving its contraposition. Suppose that $i(\lambda, \mu)>0$. Let  $t$ be any number in $(0,1)$. We prove that $\| t \v_x(\lambda)+(1-t) \v_x(\mu)\|_e<t\| \v_x(\lambda)\|_e+(1-t)\| \v_x(\mu)\|_e$ under the assumption that $\| \v_x(\lambda) \|_e=\| \v_x(\mu)\|_e=1$. By \cref{Th:Thurston-Ball}, there is a measured lamination $\nu$ such that $\v_x(\nu)=t\v_x(\lambda)+(1-t)\v_x(\mu)$. It remains to show that $\| \v_x(\nu) \|_e <1$.

Suppose, seeking a contradiction, that $\| \v_x(\nu) \|_e =1$.
By Wolpert's duality, we have
\begin{equation}
\label{differential}
t (d \ell_\lambda)_x+(1-t) (d\ell_\mu)_x=(d\ell_\nu)_x.
\end{equation}
We extend the support $|\nu|$ of $\nu$ to a maximal geodesic lamination $\tilde \nu$.
Since $i(\lambda, \mu)>0$, either $\lambda$ or $\mu$ intersects $\tilde \nu$ transversely.
Without loss of  generality, we can assume that $\mu$ does.
Let $w_{\tilde \nu}$ be the unit stretch vector along $\tilde \nu$.
Then, by \cite[Lemma 4.2]{HOP}, we have $(d\ell_\nu)_x(w_{\tilde \nu})=1$, whereas $(d\ell_\lambda)_x(w_{\tilde \nu})\leq 1$ and $(d\ell_\mu)_x(w_{\tilde \nu})<1$.
This shows that the right hand side of \cref{differential} applied to $w_{\tilde \nu}$
is $1$ whereas the left hand side is less than $1$.
This is a contradiction, and we have completed the proof.
\end{proof}

\begin{corollary}
\label{cor:strictconvexity}
The unit norm spheres for the earthquake Finsler norm on the tangent spaces of the Teichm\"uller space $\T(S)$ are strictly convex if and only if $S=S_{0,4}\text{ or }S_{1,1}$. 
\end{corollary}

\begin{proof}
If $S=\ S_{0,4}$ or $S_{1,1}$,  then $S$ cannot support any pair of distinct measured laminations with null intersection. The result then follows from \cref{segment}.
\end{proof}

 \subsection{Rigidity of the earthquake metric}
\cref{segment} implies the following  topological infinitesimal rigidity property of the earthquake norm.  For any $x\in\T(S)$, let $\e_x \colon \pml(S) \to T^1_x \T(S)$ be, as before, the homeomorphism from $\pml(S)$ to the unit tangent sphere $T^1_x\T(S)$ which takes $[\lambda] \in \pml(S)$ to the unit tangent vector $\frac{1}{\len_x(\lambda)}\v_x(\lambda)$ with respect to the earthquake norm. 

\begin{corollary}[topological infinitesimal rigidity]
\label{topological rigidity}
Let $x\in\T(S)$ and $y\in\T(S')$. For any linear isometry $h \colon (T_x \T(S),{\|\cdot\|}_e) \to (T_y\T(S'),{\|\cdot\|}_e)$  between the tangent spaces with respect to the earthquake norm,    there is a homeomorphism $g \colon S \to S'$ such that $h(v)=\e_y\circ g\circ (\e_x)^{-1}(v)$ for any $v\in T_x\T(S)$.
\end{corollary}
\begin{proof} 
This follows from \cref{duality} and the topological rigidity of the Thurston norm (\cite[Theorem 1.6]{Pan} or \cite[Theorem 1.8]{HOP}).

If $S\neq S_{0,4},S_{0,5},S_{0,6},S_{1,1},S_{1,2},S_{2,0}$, we provide an alternative proof.
Since $h$ is an isometry, it sends flats in the unit sphere of $T_x\T(S)$ to flats in the unit sphere of $T_y\T(S')$.  Hence, by  \cref{segment}, we see that the composition map $\e_y^{-1} \circ h \circ \e_x:\pml(S)\to\pml(S')$ is a homeomorphism which preserves disjointness.  Combined with \cite[Theorem 2]{OP}, this implies that $\e_y^{-1} \circ h \circ \e_x$ is induced by some homeomorphism $g:S\to S'$, that is, $\e_y^{-1} \circ h \circ \e_x([\lambda])=g([\lambda])$ for any $[\lambda]\in\pml(S)$.   Hence, $h(v)=\e_y\circ g\circ (\e_x)^{-1}(v)$ for any $v\in T_x\T(S)$.

\end{proof}

We now give the statement for the infinitesimal rigidity for the earthquake norm. Let $\v_x:\ml(S)\to T_x\T(S)$ be the homeomorphism which sends  $\lambda$ to the earthquake vector $\v_x(\lambda)$.

\begin{theorem}[infinitesimal rigidity]
\label{geometrically rigid}
Any linear isometry  $h \colon (T_x \T(S),{\|\cdot\|}_e) \to (T_y\T(S),{\|\cdot\|}_e)$ between the tangent spaces at $x$ and $y$ in $\T(S)$  is induced by an extended mapping class, that is, there is a homeomorphism  $g:S\to S$ with $g^*(y)=x$ such that 
$h(v)=\v_y\circ g \circ(\v_x)^{-1}(v)$ for any $v\in \T(S)$.\end{theorem}

We provide two different proofs of \cref{geometrically rigid}, one proof which works for an arbitrary surface, and the second one which works for $S\neq S_{1,1},S_{0,4}$.

\begin{proof}[First proof of \cref{geometrically rigid}.]
By \cref{duality}, $h$ induces a linear isometry $h^* \colon T_y^* \to T_x^*$ with respect to $\|\cdot \|_{\mathrm{Th}}$. The results of
\cite[Theorem 1.5]{Pan} or \cite[Theorem 1.1 and Theorem 1.15]{HOP} show that $h^*$ is induced by an extended mapping class. Taking the dual again, we obtain our theorem.
\end{proof}
\begin{proof}[Second proof for $S\neq S_{1,1},S_{0,4}$.]
Let $g:S\to S$ be the homeomorphism obtained from \cref{topological rigidity}. Then for any $\lambda\in\ml(S)$, we have
\begin{align*}
	h\left(\frac{\v_x(\lambda)}{\ell_x(\lambda)}\right)=\frac{\v_y(g(\lambda))}{\ell_y(g(\lambda))}.
\end{align*} 
We show that there exists a constant $K$ such that for any measured lamination $\lambda$, we have $\len_x(\lambda)=K\len_y(g(\lambda))$.
This implies the conclusion since the length function embeds $\T(S)$ into $P\reals_+^{\ml(S)}$ as was shown by Thurston, see \cite{FLPtranslation,FLP}.

We first note that for any measured lamination $\lambda$ on $S$, we have
\begin{align*}
h(\v_x(\lambda))
=\len_x(\lambda) h\left(\v_x\left(\tfrac{\lambda}{\len_x(\lambda)}\right)\right)
=\len_x(\lambda)\v_y\left(\tfrac{g(\lambda)}{\len_y(g(\lambda))}\right)
=\tfrac{\len_x(\lambda)}{\len_y(g(\lambda))}\v_y(g(\lambda)).
\end{align*}
Suppose that $\lambda$ and $\mu$ are two disjoint measured laminations. Then we have three equalities:
\begin{align*}
h(\v_x(\lambda))&=\tfrac{\len_x(\lambda)}{\len_y(g(\lambda))}\v_y(g(\lambda)),\\
h(\v_x(\mu))&=\tfrac{\len_x(\mu)}{\len_y(g(\mu))}\v_y(g(\mu)),   \text{ and}\\
h(\v_x(\lambda \sqcup \mu))
&=\tfrac{\len_x(\lambda)+\len_x(\mu)}{\len_y(g(\lambda))+\len_y(g(\mu))}\v_y(g(\lambda \sqcup \mu)),
\end{align*}
where by $\lambda \sqcup \mu$ we mean the measured lamination obtained from summing two measured laminations $\lambda$ and $\mu$ with disjoint support.

We note that $\v_x(\lambda \sqcup \mu)=\v_x(\lambda)+\v_x(\mu),$   and $\v_y(g(\lambda\sqcup \mu))=\v_y(g(\lambda))+\v_y(g(\mu))$ .
By \cref{Th:Thurston-Ball}, $\v_y(g(\lambda))$ and $\v_y(g(\mu))$ are linearly independent.
Therefore, we have
\[
\frac{\len_x(\lambda)}{\len_y(g(\lambda))}
=\frac{\len_x(\lambda)+\len_x(\mu)}{\len_y(g(\lambda)+\len_y(g(\mu))}
=\frac{\len_x(\mu)}{\len_y(g(\mu))}.
\]

This equality holds in particular when $\lambda$ and $\mu$ are disjoint weighted simple closed curves.
Recall that two disjoint simple closed curves correspond to the two vertices of an edge in the curve complex.
Since the curve complex is connected (for $S\neq S_{1,1},S_{0,4}$), every pair of simple closed curves can be jointed by a sequence of simple closed curves in which each pair of adjacent simple closed curves are disjoint.
Therefore, the above equality holds for every pair of weighted simple closed curves even when they are not disjoint.
By the density of weighted simple closed curves in $\ml(S)$, the above equality holds for every pair of measured laminations $\lambda$ and $\mu$, without the assumption that they are disjoint.
Thus we have shown that there is a constant $K$ such that $\len_x(g(\lambda))=K\len_y(\lambda)$ for any $\lambda \in \ml(S)$, and we are done.
\end{proof}

As a corollary, we have the following:\medskip

\noindent\textbf{\cref{thm:mcg rigid}} (mapping class group rigidity)\textbf{.} 
The isometry group of $(\T(S), d_e)$ is the extended mapping class group, with the usual action. This statement holds except when $S=S_{0,4},S_{0,5},S_{1,1}$ or $S_{2,0}$, for which the isometry group is the extended mapping class group modulo the hyperelliptic involution.

%Every isometry of $(\T(S), d_e)$ is induced by an extended mapping class of $S$.
%\textit{The extended mapping class group of $S$, {\color{blue}modulo the hyperelliptic involution when $S=S_{1,1},S_{2,0},S_{0,4},S_{0,5}$}, is precisely the isometry group of $(\T(S), d_e)$.}

\begin{proof}
With \cref{thm:lipregular} at hand, the proof is essentially the same as that of \cite[Theorem 1.5]{DLRT} (see also \cite[Theorem 1.7 and Theorem 1.8]{Pan} or  \cite[Corollary 1.17 and Corollary 1.18]{HOP}), and here is a sketch: let $f:\T(S)\to\T(S)$ be an isometry with respect to the earthquake metric. By \cref{thm:lipregular}, the earthquake norm is locally Lipschitz. Combined with \cite[Theorem A and Theorem B]{MT2017}, this implies that the isometry $f$ is of class $C^1$ and norm-preserving.  Combining with  \cref{geometrically rigid}, we see that at every point $x\in\T(S)$, there exists  an extended mapping class $g_x$ of $S$ such that
$f(x)=g_x(x)$. The discreteness of the action of the extended mapping class group on $\T(S)$ then enables us to choose a common extended mapping class $g$ such that $f(x)=g(x)$ for all  $x\in\T(S)$, and this tells us that every isometry can be realised by the action of some extended mapping class. To finish off the proof, we note that extended mapping classes acts isometrically on $(\T(S),d_e)$, and only hyperelliptic involutions for $S=S_{1,1},S_{2,0},S_{0,4},S_{0,5}$ act trivially. Hence, the isometry group of $(\T(S),d_e)$ is the extended mapping class group, modulo the hyperelliptic involutions for the listed exceptional surfaces.
\end{proof}

\subsection{Asymmetry of the earthquake metric}
\label{subsec:asymmetry}
In this subsection, we use the infinitesimal  rigidity  to prove the asymmetry of the earthquake norm.

\smallskip

\noindent\textbf{\cref{prop:asymmetric}}
The earthquake norm is asymmetric at every point of $\T(S)$. That is,  for any $x\in\T(S)$, there exists a tangent vector $v\in T_x\T(S)$ such that $\|v\|_e\neq\|-v\|_e$.

\begin{proof}[Proof of \cref{prop:asymmetric}]
Suppose on the contrary that the earthquake norm \emph{is} symmetric at some point $x\in\T(S)$, which is to say: $\|v\|_e=\|-v\|_e$ for every ${v\in T_x\T(S)}$. By \cref{geometrically rigid}, there is a mapping class $g$ which fixes $x$ such that 
\begin{align}\label{eq:tangent:x}
\v_x\circ g\circ (\v_x)^{-1} (v)=-v, \text{ for all } v\in T_x\T(S).
\end{align}
Consider the map $g_*:\T(S)\to\T(S)$  induced by $g$, and its induced map on the tangent space, $dg_*:T_x\T(S)\to T_{g_*(x)}\T(S)=T_x\T(S)$.  For any measured lamination $\lambda$, the map $g$ sends the earthquake path $E_{t\lambda}(x)$ to the earthquake path $E_{tg(\lambda)}(g(x))$. Therefore, 
\[
dg_*(\v_x(\lambda))=\v_{g_*(x)}(g(\lambda))=\v_x(g(\lambda)),
\] 
where the second equality follows from the fact that $g$ fixes $x$. Hence, for any tangent vector $v\in T_x\T(S)$, we have 
	\begin{equation*}
		dg_* (v)= \v_x\circ g\circ (\v_x)^{-1}(v).
	\end{equation*}
	Combined with \cref{eq:tangent:x}, this implies that $dg_*(v)=-v$ for every $v\in T_x\T(S)$.
	
On the other hand, deforming the quotient of $(S,x)$ under  the cyclic group generated by $g$ yields a smooth path in $\T(S)$ which passes through $x$ and which is pointwise fixed by $g$. Hence, the tangent vector $v\neq0$ to this path at $x$ is fixed by $dg_*$, contradicting the identity $dg_*(v)=-v$ for all $v\in T_x\T(S)$. This completes the proof.
\end{proof}

\begin{corollary}
The Thurston norm is asymmetric at every point of $\T(S)$.
\end{corollary}

\begin{proof}
Combine \cref{prop:asymmetric} with the infinitesimal duality between the earthquake and Thurston metrics (\cref{duality}). Alternatively, apply the same proof strategy as for \cref{prop:asymmetric}. 
\end{proof}

\noindent\textbf{\cref{cor:metricasym}.}
\textit{The earthquake metric is not symmetric, that is, there exist $x$ and $y$ in $\T(S)$ such that $d_e(x,y)\neq d_e(y,x)$.}

\begin{proof}
We assume, for contradiction, that $d_e$ is symmetric: for all $ x,y\in\T(S)$,
\[
d_e(x,y)=d_e(y,x)=d_e^\#(x,y).
\]
This means that the identity map $\mathrm{Id}:\T(S)\to\T(S)$ on Teichm\"uller space induces an isometry between $(\T(S),d_e)$ and $(\T(S),d_e^\#)$. By \cite[Theorem~B]{MT2017} and \cref{thm:lipregular}, the respective Binet--Legendre metrics for the left and right earthquake metrics must be locally Lipschitz continuous, hence \cite[Theorem~B]{MT2017} ensures that $\|(x,v)\|_e=\mathrm{Id}^*\|(x,v)\|_e^\#=\|(x,v)\|_e^\#$ for all $(x,v)\in T\T(S)$. This contradicts \cref{prop:asymmetric}, hence there must be points $x,y\in\T(S)$ such that $d_e(x,y)\neq d_e(y,x)$.
\end{proof}

\begin{remark}
The proof of \cref{cor:metricasym} does not actually need \cite[Theorem~B]{MT2017} and \cref{thm:lipregular}, as those results are used to ensure the regularity of isometries of Finsler metrics with sufficient regularity. Indeed, one needs only apply the arguments outlined on the first page of \cite[\S3]{MT2017}. We present the above proof for simpler referencing.

\end{remark}

\newpage
\section{Comparisons with other norms}
\label{sec:comparison}

The goal of this section is to compare the following Finsler norms on Teichm\"uller space: 
\begin{itemize}
\item
the earthquake norm ${\|\cdot\|}_e$, 
\item
the Weil--Petersson norm ${\|\cdot\|}_{\mathrm{WP}}$, 
\item
the Thurston norm ${\|\cdot\|}_{\mathrm{Th}}$, and 
\item
 the Teichm\"uller norm ${\|\cdot\|}_T$.
\end{itemize}
We prove following  \cref{thm:e:wp:th:m} mentioned in the introduction.

\smallskip

\noindent{\bf \cref{thm:e:wp:th:m}.}
{\em	There are   positive constants $C_0,C_1,C_2$, depending only on the topology of $S$, such that 
	for any $x\in\T(S)$ and any $v\in T_x\T(S)$, 
\begin{equation*}
C_0 \ell_{\mathrm{sys}}(x)\Log\frac{1}{\ell_{\mathrm{sys}}(x)}\|v\|_{\mathrm{Th}}\leq 	\|v\|_e{\leq} C_1\|v\|_{\mathrm{WP}}{\leq}  C_2 \|v\|_{\mathrm{Th}},
\end{equation*}
where $\Log(t)=\max\{1,\log t\}$.}

We establish \cref{thm:e:wp:th:m} by proving each inequality through a series of pairwise norm comparisons: \cref{lem:e:th:1}, \cref{thm:e:wp}, and \cref{thm:wp:th}. We strengthen \cref{thm:e:wp:th:m} with Wolpert \cite[Lemma 3.1]{Wolpert1979} and Burns--Masur--Wilkinson \cite[Lemma 5.4]{BMW2012} to get \cref{thm:various}. This allows us to pairwise compare each of the four norms under consideration, and we state these pairwise comparisons as a corollary for convenient referencing.

\begin{corollary}\label{cor:comparison:pairwise}
There are positive constants $C$ and $C'$, depending only on the topology of $S$, such that for any $x\in\T(S)$ and any $v\in T_x\T(S)$, 
\begin{align}
C \ell_{\mathrm{sys}}(x)\Log\frac{1}{\ell_{\mathrm{sys}}(x)}\|v\|_{\mathrm{Th}}
\leq &\|v\|_e
\leq C' \|v\|_{\mathrm{Th}}\\
C \ell_{\mathrm{sys}}(x)\Log\frac{1}{\ell_{\mathrm{sys}}(x)}\|v\|_{\mathrm{WP}}
\leq &\|v\|_e
\leq C'\|v\|_{\mathrm{WP}}\\
C \ell^2_{\mathrm{sys}}(x)\Log\frac{1}{\ell_{\mathrm{sys}}(x)}\|v\|_{T}
\leq &\|v\|_e
\leq C' \|v\|_{T}\\
C \ell_{\mathrm{sys}}(x)\Log\frac{1}{\ell_{\mathrm{sys}}(x)}\|v\|_{\mathrm{Th}}
\leq &\|v\|_{\mathrm{WP}}
\leq C' \|v\|_{\mathrm{Th}}\\
C \ell^2_{\mathrm{sys}}(x)\Log\frac{1}{\ell_{\mathrm{sys}}(x)}\|v\|_{T}
\leq &\|v\|_{\mathrm{WP}}
\leq C' \|v\|_{T}\\
C\ell_{\mathrm{sys}}(x)\|v\|_{T}
\leq &\|v\|_{\mathrm{Th}}
\leq C' \|v\|_{T}.\label{eg:teichvsth}
\end{align}
Moreover, we also see that
\begin{align}
C  \ell_{\mathrm{sys}}(x) {\|v\|_{\mathrm{Th}}}\leq &{\|-v\|_{\mathrm{Th}}}\leq \frac{C'}{ \ell_{\mathrm{sys}}(x)}\|v\|_{\mathrm{Th}}
\\
C  \ell^2_{\mathrm{sys}}(x)\Log\frac{1}{\ell_{\mathrm{sys}}(x)} {\|v\|_{e}}\leq &{\|-v\|_{e}}\leq \frac{C'}{ \ell_{\mathrm{sys}}(x)\Log\frac{1}{\ell_{\mathrm{sys}}(x)}}\|v\|_{e}.
\end{align}
\end{corollary}

\begin{remark}
\cref{eg:teichvsth} tells us that the ratio of the Thurston norm over the Teichm\"uller norm is bounded below by a constant times the (length of the) systole. This resolves \cite[Problem 5.3]{PapadopoulosSu2015}. 
\end{remark}

\subsection{Earthquake vs. Thurston}
We start with the comparison between the earthquake norm and the Thurston norm.

\begin{lemma}\label{lem:e:th:1}
	There exists a positive constant $C$ depending only on the topology of $S$ such that for any hyperbolic surface $x\in\T(S)$ and any tangent vector $v\in T_x\T(S)$, we have
\begin{align}\label{eq:e:th:1}
C{\ell_{\mathrm{sys}}(x)\Log(\tfrac{1}{\ell_{\mathrm{sys}}(x)})}\|v\|_{\mathrm{Th}}\leq\|v\|_{e},
\end{align}
	where, as before, $\Log(t)=\max\{1,\log t\}$.
\end{lemma}

\begin{proof}
Let $\alpha$ be a simple closed curve on $S$. For $x\in\mathcal{T}(S)$, we have $||\v_x(\alpha)||_e=\ell_\alpha(x)$, and
\begin{align*}
||\v_x(\alpha)||_{\mathrm{Th}}
=
\sup_\beta \frac{\v_x(\alpha)(\ell_\beta)}{\ell_\beta(x)}
=
\sup_\beta \frac{\Cos_x(\alpha,\beta)}{\ell_\beta(x)}
\leq \sup_\beta \frac{i(\alpha,\beta)}{\ell_\beta(x)},
\end{align*}
where $\beta$ ranges over all simple closed curves on $x$. Therefore,
\begin{align*}
\frac{||\v_x(\alpha)||_{\mathrm{Th}}}{||\v_x(\alpha)||_e}
=
\frac{||\v_x(\alpha)||_{\mathrm{Th}}}{\ell_\alpha(x)}
%=\sup_{\beta} \frac{\v_x(\alpha)(\ell_\beta)}{\ell_\alpha(x)\ell_\beta(x)}
\leq \sup_\beta \frac{i(\alpha,\beta)}{\ell_\alpha(x)\ell_\beta(x)}.
\end{align*}
Torkaman \cite{Torkaman2023} discovered that as one approaches the thin part of Teichm\"uller space, \ie  as  $\ell_{\mathrm{sys}}(x)\to 0$, we have
\begin{align}
\label{eq:intersection:length:asym}
\sup_{\alpha,\beta} \frac{i(\alpha,\beta)}{\ell_\alpha(x)\ell_\beta(X)}
\sim
\frac{1}{2\ell_{\mathrm{sys}}(x)\log\frac{1}{\ell_{\mathrm{sys}}(x)}}.
\end{align}
In particular, there exists a constant $C$ depending only on the topology of $S$ such that 
\begin{align} \label{eq:inters}
\sup_{\alpha,\beta}\frac{i(\alpha,\beta)}{\ell_\alpha(x)\ell_\beta(x)}
\leq\frac{C}{\ell_{\mathrm{sys}}(x)\Log\frac{1}{\ell_{\mathrm{sys}}(x)}}.
\end{align}
 Therefore,
     	\begin{align*}
		\frac{||\v_x(\alpha)||_{\mathrm{Th}}}{||\v_x(\alpha)||_e}\leq \sup_\beta \frac{i(\alpha,\beta)}{\ell_\alpha(x)\ell_\beta(x)}
		 \leq \frac{C}{\ell_{\mathrm{sys}}(x)\Log \frac{1}{\ell_{\mathrm{sys}}(x)}}.
		\end{align*}
\end{proof}

Note that the estimate in \cref{eq:inters} is weaker than the one in \cref{eq:intersection:length:asym}, and that the former estimate can also be proved using the  thick-thin decomposition estimate in \cite[Proposition 3.1]{LRT2012}.

\subsection{Earthquake vs. Weil--Petersson}

We next compare the earthquake norm and the Weil--Petersson norm. We make use of the following estimate for the Weil--Petersson gradient of length functions by Bonsante--Seppi--Tamburelli:

\begin{theorem}[{\cite[Theorem H]{BST2017}}]\label{lem:wp:lbd}
There exists a universal positive constant $K$ such that for any simple closed curve $\alpha$ and any hyperbolic surface $x\in\T(S)$, we have
\begin{align}\label{eq:wp:lbd}
\|\nabla \ell_\alpha\|_{\mathrm{WP}}\geq \frac{K}{|\chi(S)|} \ell_\alpha,
\end{align}
where $\nabla\ell_\alpha$ is the gradient of the length function $\ell_\alpha:\T(S)\to\R$.
\end{theorem}

\begin{remark}
Even though the surface $S$ is closed in \cite[Theorem H]{BST2017}, the same argument holds for punctured surfaces.
\end{remark}

\begin{proposition}\label{thm:e:wp} There exists a positive constant $C$ depending only on the topology of $S$ such that for any hyperbolic surface $x\in\T(S)$ and any tangent vector $v\in T_x\T(S)$, we have
\begin{align}\label{eq:e:wp}
\|v\|_{e}\leq C \|v\|_{\mathrm{WP}},
\end{align}
where $C=\frac{2|\chi(S)|}{K}$ depends only on the topological type of $S$, and $K>0$ is a universal constant (independent of $S$) given by \cref{lem:wp:lbd}.
\end{proposition}

\begin{proof}[Proof of \cref{thm:e:wp}]
By \cref{Th:Thurston-Ball}, the tangent vector $v$ can be represented as the earthquake vector $\v_x(\alpha)$ for some measured lamination $\alpha$. Since the set of weighted simple closed curves is dense in the space of measured laminations, we have only to consider the case when $\alpha$ is a simple closed curve to get \cref{eq:wp:lbd}.
In this case, by \cref{lem:wp:nabla}, we have
\begin{align*}
\|\v_x(\alpha)\|_{\mathrm{WP}}=\tfrac{1}{2}\|\nabla\ell_\alpha\|_{\mathrm{WP}}.
\end{align*}	
Combined with \cref{lem:wp:lbd}, this completes the proof.
\end{proof}

Given \cref{thm:e:wp}, one naturally wonders whether the earthquake norm and the \wp norm might be bi-Lipschitz equivalent. We show that this is not the case.

\begin{proposition}[not bi-Lipschitz]
\label{prop:notbilipschitz}
For any simple closed curve $\alpha$, the earthquake norm and the \wp norm for the vectors $\v_x(\alpha)\ (x \in \T(S))$ are not bi-Lipschitz equivalent. Specifically, there is some uniform constant $c$ such that
\[
\sqrt{\frac{\left\|\v_x(\alpha)\right\|_{e}
}{2\pi}}
\leq 
\left\|\v_x(\alpha)\right\|_{\mathrm{WP}}
\leq 
\sqrt{2c
\left\|\v_x(\alpha)\right\|_{e}},
\text{ as }\ell_\alpha(x)\to 0.
\]
\end{proposition}

\begin{proof}
Given an arbitrary $\alpha$, Riera's formula \cite[Theorem 2, equation (b)]{Riera2005} implies that $\left\|\nabla\ell_\alpha\right\|_{\mathrm{WP}}
\geq \sqrt{ 2\ell_\alpha(x)/\pi}$. On the other hand, Wolpert \cite[Lemma 3.16]{Wolpert2008} showed that 
\begin{align*}
\left\|\nabla \ell_\alpha\right\|_{\mathrm{WP}}
\leq\sqrt{ c(\ell_\alpha(x)+\ell^2_\alpha(x)e^{\ell_\alpha(x)/2})}
\end{align*}  
for some universal constant $c$. Thus, for small $\ell_\alpha(x)$, we have
\[
\sqrt{\frac{\left\|\v_x(\alpha)\right\|_{e}
}{2\pi}}
=
\sqrt{\frac{\ell_\alpha(x)}{2\pi}}
\leq
\tfrac{1}{2}\left\|\nabla\ell_\alpha\right\|_{\mathrm{WP}}
\leq 
\sqrt{2c\;\ell_\alpha(x)}
=
\sqrt{2c
\left\|\v_x(\alpha)\right\|_{e}}.
\]
Using \cref{lem:wp:nabla} to replace $\tfrac{1}{2}\|\nabla\ell_\alpha\|_{\mathrm{WP}}$ with $\|\v_x(\alpha)\|_{\mathrm{WP}}$ yields the result.
\end{proof}

We now employ \cref{thm:e:wp} to show:
\begin{theorem}\label{thm:wp:th}
	There exists a constant $C$ depending only on the topology of $S$, such that for any $x\in \T(S)$ and any tangent vector $v\in T_x\T(S)$,  
	\begin{align}
	\label{eq:e:wp:th}
		\frac{\|v\|_{\mathrm{WP}}}{\|v\|_{\mathrm{Th}}}
		\leq 
		\sup_{v\in T_x \T(S)}
 		\frac{\|v\|_{e}}{2\|v\|_{\mathrm{WP}}}.
	\end{align}
	In particular, 
	\begin{align}
	\label{eq:wp:th}
		\|v\|_{\mathrm{WP}}\leq C \|v\|_{\mathrm{Th}}.
	\end{align}
	Moreover, 
	\begin{equation}\label{eq:e:th}
		\|v\|_{e}\leq C \|v\|_{\mathrm{Th}}.
	\end{equation}
\end{theorem}
\begin{proof}
By the density of simple closed curves in $\pml(S)$, it suffices to verify \cref{eq:e:wp:th} for $v=\v_x(\alpha)$, where $\alpha$ is an arbitrary simple closed curve.   For simplicity, we omit the $x$ in  $\v_\alpha(x)=v_x(\alpha)$ and write $\v_\alpha$. Recall from \cref{lem:wp:nabla} that $\|\v_\alpha\|_{\mathrm{WP}}=\frac{1}{2}\|\nabla\ell_\alpha\|_{\mathrm{WP}}$, hence
\begin{align*}
\frac{\|\v_\alpha\|_{\mathrm{WP}}}{\|\v_\alpha\|_{\mathrm{Th}}}
= \frac{\|\nabla \ell_\alpha\|_{\mathrm{WP}}}{2\|\v_\alpha\|_{\mathrm{Th}}}
&=\frac{1}{2\|\v_\alpha\|_{\mathrm{Th}}} \sup _{\beta}  \frac{|\mathrm{d}\ell_\alpha(\v_\beta)|}{\|\v_\beta\|_{\mathrm{WP}}}\\
&=\frac{1}{2\|\v_\alpha\|_{\mathrm{Th}}} \sup _{\beta}  \frac{-\mathrm{d}\ell_\alpha(\v_\beta)}{\|-\v_\beta\|_{\mathrm{WP}}}
= \frac{1}{2\|\v_\alpha\|_{\mathrm{Th}}} \sup _{\beta}  \frac{\mathrm{d}\ell_\beta(\v_\alpha)}{\|\v_\beta\|_{\mathrm{WP}}},
\end{align*}
where the last equality makes use of \cref{reciprocal}. Hence,
\begin{align*}
\frac{\|\v_\alpha\|_{\mathrm{WP}}}{\|\v_\alpha\|_{\mathrm{Th}}}
= \frac{1}{2\|\v_\alpha\|_{\mathrm{Th}}} \sup _{\beta}  \frac{[\v_\alpha(\frac{\mathrm{d}\ell_\beta}{\ell_\beta})]\ell_\beta}{\|\v_\beta\|_{\mathrm{WP}}}
&\leq \frac{1}{2\|\v_\alpha\|_{\mathrm{Th}}} \sup _{\beta}  \frac{\|\v_\alpha\|_{\mathrm{Th}}\ell_\beta}{\|\v_\beta\|_{\mathrm{WP}}}\\
&=\frac{1}{2}\sup _{\beta}  \frac{\ell_\beta}{\|\v_\beta\|_{\mathrm{WP}}}
=\sup _{\beta}  \frac{\|\v_\beta\|_{e}}{2\|\v_\beta\|_{\mathrm{WP}}}.
\end{align*}
This proves \cref{eq:e:wp:th}, which in turn,  combined with \cref{thm:e:wp}, gives  \cref{eq:wp:th} and \cref{eq:e:th}.
\end{proof}

%\begin{theorem}\label{thm:T:th}
%There exists a  constant $C$ depending on the topology of $S$, such that 
%	for any $x\in\T(S)$ and any $v\in T_x\T(S)$, 
%	\begin{equation*}
	%	\|v\|_{\mathrm{Th}}\geq C \cdot\ell_{\mathrm{sys}}(x) \|v\|_T
	%\end{equation*}
	%and
	%\begin{equation*}
	%	C\cdot  \ell_{\mathrm{sys}}(x) \|v\|_{\mathrm{Th}}\leq \|-v\|_{\mathrm{Th}}\leq \frac{C}{\ell_{\mathrm{sys}}(x)}\|v\|_{\mathrm{Th}},
%	\end{equation*}
%	where  ${\|\cdot\|}_T$ represents the Teichm\"uller norm.
%\end{theorem}
%\begin{proof}
%	By \cite[Lemma 5.4]{BMW2012},  there exists a constant $C$ depending on the topology of $S$ such that for any $x\in\T(S)$ and any $v\in T_x\T(S)$, 
	%\begin{equation*}
%		\|v\|_{\mathrm{WP}}\geq C \ell_{\mathrm{sys}}(x) \|v\|_T.
%	\end{equation*}
%	Combined with \cref{thm:wp:th}, this proves the first inequality.  The second inequality follows from the first inequality and the fact that  $\|v\|_T\geq \|v\|_{\mathrm{Th}}$ (see \cite[Lemma 3.1]{Wolpert1979}):
%	\begin{equation*}
%		\|-v\|_{\mathrm{Th}}\leq \|-v\|_T=\|v\|_T\leq \frac{C} {\ell_{\mathrm{sys}}(x)} \|v\|_{\mathrm{Th}}.
	%\end{equation*}
	%Interchanging $-v$ and $v$ yields the other direction.
%\end{proof}

\newpage

\section{Long earthquake paths are non-geodesic}

\label{s:geodesy}

The main goal of this section is to show the following.
\smallskip

\noindent\textbf{\cref{thm:longquake:nongeo}.} \textit{Sufficiently long left earthquake paths in $\T(S)$ cannot be geodesics with respect to the left earthquake metric.}

\subsection{Dehn twists case}
\label{DT}

We first give an elementary proof for the special case where the earthquake in question is a Fenchel--Nielsen twist, as it tells us something general about the geometry of moduli space.
(This kind of argument was originally suggested by Thurston in  \cite{ThM}.)

\begin{proof}[Proof of \cref{thm:longquake:nongeo} for Fenchel-Nielsen twists]
Let $\gamma$ be an arbitrary simple closed curve on $S$, consider an arbitrary point $x\in\T(S)$ and set $x_m:=E_{m\gamma}(x)\in\T(S)$ as its $m$-th right Dehn twist around $\gamma$  (see \cref{fig:earthquake:path}). Choose an arbitrary $y\in\T(S)$ such that $\ell_{\gamma}(y)<\ell_{\gamma}(x)$, and similarly define $y_m:=E_{m\gamma}(y)$. Since $d_e$ is mapping class group invariant, for all $m\in\mathbb{Z}$,
\[
d_e(y_m,x_m)=d_e(E_{m\gamma}(y),E_{m\gamma}(x))=d_e(y,x).
\]
Hence, the triangle inequality tells us that
\begin{align*}
d_e(x,x_m)\leq d_e(x,y)+d_e(y,y_m)+d_e(y,x).
\end{align*}
Since  $d_e(y,y_m)$ is bounded above by the length of the earthquake path from $y$ to $y_m$, we obtain the following.
\begin{align*}
d_e(x,x_m)\leq d_e(x,y)+m\ell_\gamma(y)^2+d_e(y,x).
\end{align*}
This inequality ensures that for $m\in\mathbb{N}$ sufficiently large, $d_e(x,x_m)$ cannot be realised by the earthquake path from $x$ to $x_m$, which has length
\[
m\ell_{\gamma}(x)^2
\gg
d_e(x,y)+m\ell_\gamma(y)^2+d_e(y,x).
\]
Since subsegments of geodesics are geodesics, no earthquake path containing $E_{t\gamma}(x)$, $t\in[0,m]$, can be a geodesic.
\end{proof}

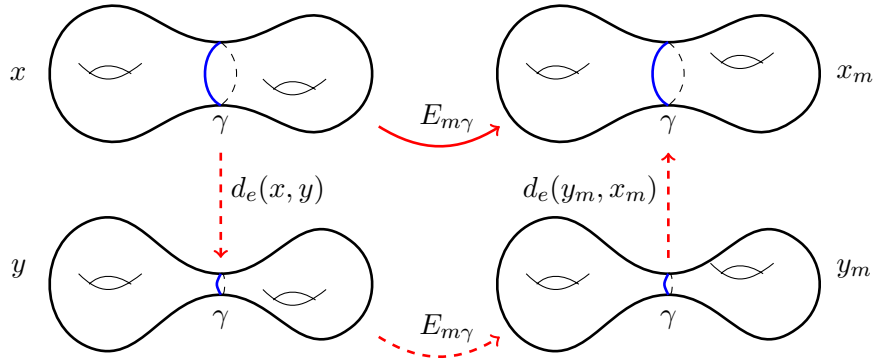
\begin{figure}[h!]
	\begin{tikzpicture}[scale=0.7]
		\draw[line width=1pt] (0-1,1.2)..controls (1-1,1.6) and (1.4-1,0.6)..(2.5-1,0.6)
		 (0-1,1.2)..controls (-1-1,0.8) and (-1-1,-.8)..(0-1,-1.2)
		 (0-1,-1.2)..controls (1-1,-1.6) and (1.4-1,-0.6)..(2.5-1,-0.6) 
		 (0+5-1,0.8)..controls (4-1,1.6) and (3.6-1,0.6)..(2.5-1,0.6)
		 (5-1,0.8)..controls (5.5-1,0.4) and (5.5-1,-0.4)..(5-1,-0.8)
		 (5-1,-0.8)..controls (4-1,-1.6) and (3.6-1,-0.6)..(2.5-1,-0.6);
		 \draw (-0.2-1,0.2).. controls (0.2-1,-0.2) and (0.5-1,-0.2)..(1-1,0.2) 
		 (0-1,0).. controls (0.2-1,0.2) and (0.5-1,0.2)..(0.8-1,0); 
		 \draw (-0.2+3.5-1,0.2-0.3).. controls (0.2+3.5-1,-0.2-0.3) and (0.5+3.5-1,-0.2-0.3)..(1+3.5-1,0.2-0.3) 
		 (0+3.5-1,0-0.3).. controls (0.2+3.5-1,0.2-0.3) and (0.5+3.5-1,0.2-0.3)..(0.8+3.5-1,0-0.3);
		 \draw[line width=1pt, color=blue] (2.5-1,0.6)..controls (2.1-1,0.4) and (2.1-1,-0.4)..(2.5-1,-0.6);
		 \draw[dashed](2.5-1,0.6)..controls (2.9-1,0.3) and (2.9-1,-0.3)..(2.5-1,-0.6);
		 \draw(2.5-1,-0.6)node[below]{{$\gamma$}}
		 (2.5-1-3.5,0)node[left]{{}{$x$}};
		 
		 	\draw[line width=1pt] (0+6.5+1,1.2)..controls (1+6.5+1,1.6) and (1.4+6.5+1,0.6)..(2.5+6.5+1,0.6)
		 (0+6.5+1,1.2)..controls (-1+6.5+1,0.8) and (-1+6.5+1,-.8)..(0+6.5+1,-1.2)
		 (0+6.5+1,-1.2)..controls (1+6.5+1,-1.6) and (1.4+6.5+1,-0.6)..(2.5+6.5+1,-0.6) 
		 (0+5+6.5+1,0.8)..controls (4+6.5+1,1.6) and (3.6+6.5+1,0.6)..(2.5+6.5+1,0.6)
		 (5+6.5+1,0.8)..controls (5.5+6.5+1,0.4) and (5.5+6.5+1,-0.4)..(5+6.5+1,-0.8)
		 (5+6.5+1,-0.8)..controls (4+6.5+1,-1.6) and (3.6+6.5+1,-0.6)..(2.5+6.5+1,-0.6);
		 \draw (-0.2+6.5+1,0.2).. controls (0.2+6.5+1,-0.2) and (0.5+6.5+1,-0.2)..(1+6.5+1,0.2) 
		 (0+6.5+1,0).. controls (0.2+6.5+1,0.2) and (0.5+6.5+1,0.2)..(0.8+6.5+1,0); 
		 \draw (-0.2+6.5+3.5+1,0.2+0.2).. controls (0.2+3.5+6.5+1,-0.2+0.2) and (0.5+3.5+6.5+1,-0.2+0.2)..(1+3.5+6.5+1,0.2+0.2) 
		 (0+3.5+6.5+1,0+0.2).. controls (0.2+3.5+6.5+1,0.2+0.2) and (0.5+3.5+6.5+1,0.2+0.2)..(0.8+3.5+6.5+1,0+0.2);
		 \draw[line width=1pt, color=blue] (2.5+6.5+1,0.6)..controls (2.1+6.5+1,0.4) and (2.1+6.5+1,-0.4)..(2.5+6.5+1,-0.6);
		 \draw[dashed](2.5+6.5+1,0.6)..controls (2.9+6.5+1,0.3) and (2.9+6.5+1,-0.3)..(2.5+6.5+1,-0.6);
			\draw(2.5+6.5+1,-0.6)node[below]{{$\gamma$}} (2.5-1+11.5,0)node[right]{{}{$x_m$}};
			
			%===second row====
			\draw[line width=1pt] (0-1,1.2-4)..controls (1-1,1.6-4) and (1.4-1,0.6-4-0.4)..(2.5-1,0.6-4-0.4)
		 (0-1,1.2-4)..controls (-1-1,0.8-4) and (-1-1,-.8-4)..(0-1,-1.2-4)
		 (0-1,-1.2-4)..controls (1-1,-1.6-4) and (1.4-1,-0.6-4+0.4)..(2.5-1,-0.6-4+0.4) 
		 (0+5-1,0.8-4)..controls (4-1,1.6-4) and (3.6-1,0.6-4-0.4)..(2.5-1,0.6-4-0.4)
		 (5-1,0.8-4)..controls (5.5-1,0.4-4) and (5.5-1,-0.4-4)..(5-1,-0.8-4)
		 (5-1,-0.8-4)..controls (4-1,-1.6-4) and (3.6-1,-0.6-4+0.4)..(2.5-1,-0.6-4+0.4);
		 \draw (-0.2-1,0.2-4).. controls (0.2-1,-0.2-4) and (0.5-1,-0.2-4)..(1-1,0.2-4) 
		 (0-1,0-4).. controls (0.2-1,0.2-4) and (0.5-1,0.2-4)..(0.8-1,0-4); 
		 \draw (-0.2+3.5-1,0.2-0.3-4).. controls (0.2+3.5-1,-0.2-0.3-4) and (0.5+3.5-1,-0.2-0.3-4)..(1+3.5-1,0.2-0.3-4) 
		 (0+3.5-1,0-0.3-4).. controls (0.2+3.5-1,0.2-0.3-4) and (0.5+3.5-1,0.2-0.3-4)..(0.8+3.5-1,0-0.3-4);
		 \draw[line width=1pt, color=blue] (2.5-1,0.6-4-0.4)..controls (2.1-1+0.3,0.4-4-0.4) and (2.1-1+0.3,-0.4-4+0.4)..(2.5-1,-0.6-4+0.4);
		 \draw[dashed](2.5-1,0.6-4-0.4)..controls (2.9-1-0.3,0.3-4) and (2.9-1-0.3,-0.3-4)..(2.5-1,-0.6-4+0.4);
		 \draw(2.5-1,-0.6-3.7)node[below]{{$\gamma$}} (2.5-1-3.5,-3.7)node[left]{{}{$y$}};
		 
		 	\draw[line width=1pt] (0-1+8.5,1.2-4)..controls (1-1+8.5,1.6-4) and (1.4-1+8.5,0.6-4-0.4)..(2.5-1+8.5,0.6-4-0.4)
		 (0-1+8.5,1.2-4)..controls (-1-1+8.5,0.8-4) and (-1-1+8.5,-.8-4)..(0-1+8.5,-1.2-4)
		 (0-1+8.5,-1.2-4)..controls (1-1+8.5,-1.6-4) and (1.4-1+8.5,-0.6-4+0.4)..(2.5-1+8.5,-0.6-4+0.4) 
		 (0+5-1+8.5,0.8-4)..controls (4-1+8.5,1.6-4) and (3.6-1+8.5,0.6-4-0.4)..(2.5-1+8.5,0.6-4-0.4)
		 (5-1+8.5,0.8-4)..controls (5.5-1+8.5,0.4-4) and (5.5-1+8.5,-0.4-4)..(5-1+8.5,-0.8-4)
		 (5-1+8.5,-0.8-4)..controls (4-1+8.5,-1.6-4) and (3.6-1+8.5,-0.6-4+0.4)..(2.5-1+8.5,-0.6-4+0.4);
		 \draw (-0.2-1+8.5,0.2-4).. controls (0.2-1+8.5,-0.2-4) and (0.5-1+8.5,-0.2-4)..(1-1+8.5,0.2-4) 
		 (0-1+8.5,0-4).. controls (0.2-1+8.5,0.2-4) and (0.5-1+8.5,0.2-4)..(0.8-1+8.5,0-4); 
		 \draw (-0.2+3.5-1+8.5,0.2-0.3-4+0.5).. controls (0.2+3.5-1+8.5,-0.2-0.3-4+0.5) and (0.5+3.5-1+8.5,-0.2-0.3-4+0.5)..(1+3.5-1+8.5,0.2-0.3-4+0.5) 
		 (0+3.5-1+8.5,0-0.3-4+0.5).. controls (0.2+3.5-1+8.5,0.2-0.3-4+0.5) and (0.5+3.5-1+8.5,0.2-0.3-4+0.5)..(0.8+3.5-1+8.5,0-0.3-4+0.5);
		 \draw[line width=1pt, color=blue] (2.5-1+8.5,0.6-4-0.4)..controls (2.1-1+0.3+8.5,0.4-4-0.4) and (2.1-1+0.3+8.5,-0.4-4+0.4)..(2.5-1+8.5,-0.6-4+0.4);
		 \draw[dashed](2.5-1+8.5,0.6-4-0.4)..controls (2.9-1-0.3+8.5,0.3-4) and (2.9-1-0.3+8.5,-0.3-4)..(2.5-1+8.5,-0.6-4+0.4);
	\draw(2.5-1+8.5,-0.6-3.7)node[below]{{$\gamma$}}
	(2.5-1+11.5,-3.7)node[right]{{}{$y_m$}};
 \draw[->, color=red, line width=1pt] (4.5,-1)..controls (5.3,-1.5) and (6,-1.5)..(6.8,-1);
 \draw (5.8,-1.3)node[above]{{$E_{m\gamma}$}};
	\draw[->, color=red, line width=1pt, dashed] (4.5,-1-4)..controls (5.3,-1.5-4) and (6,-1.5-4)..(6.8,-1-4);
 \draw[->, color=red, line width=1pt, dashed] (1.5,-1.5)--(1.5,-3.5);
 \draw[->, color=red, line width=1pt, dashed] (1.5+8.5,-3.5)--(1.5+8.5,-1.5);
 \draw (5.8,-1.3-4)node[above]{{$E_{m\gamma}$}}
 (1.5,-2.2)node[right]{{$d_e(x,y)$}} (1.5+8.5,-2.2)node[left]{{$d_e(y_m,x_m)$}};	
			\end{tikzpicture}
	\caption{Long earthquakes cannot be geodesic.}
	\label{fig:earthquake:path}
\end{figure}

This proof did not make use of the earthquake metric in a \emph{fundamental} way.
The same argument can be used to show that for any mapping class group invariant path metric on Teichm\"uller space where arbitrarily long earthquakes continue to be geodesics, the distance of Fenchel-Nielsen twists around a fixed simple closed curve with respect to the metric have to be uniform:

\begin{lemma}
\label{lem:dtstrange}
Let $d$ be any mapping class group invariant asymmetric metric $d$ on $\T(S)$ where all left earthquake paths on $\T(S)$ are geodesics.
 Then for any essential simple closed curve $\gamma$ on $S$ the following holds.
\[
\text{For any } x,y\in\T(S) \text{ and for any } m\in\mathbb{N},\quad d(x,E_{m\gamma}(x))=d(y,E_{m\gamma}(y)).
\]
\end{lemma}

\begin{proof}
Since earthquakes for simple closed curves are  Fenchel-Nielsen twists, the unique (bi-infinite) left earthquake joining $x$ and $E_{\gamma}(x)$ contains $E_{m\gamma}(x)$ for all $m\in\mathbb{N}$. In particular,
\begin{align}
d(x,E_{m\gamma}(x))=md(x,E_{\gamma}(x)).\label{eq:s8_1}
\end{align}
By the triangle inequality and the mapping class group invariance of $d$,
\begin{align}
d(x,E_{m\gamma}(x))
&\leq d(x,y)+d(y,E_{m\gamma}(y))+d(E_{m\gamma}(y),E_{m\gamma}(x))\notag\\
&= d(x,y)+md(y,E_{\gamma}(y)) + d(y,x).\label{eq:s8_2}
\end{align}
Combining \cref{eq:s8_1} and \cref{eq:s8_2}, we obtain
\[
d(x,E_{\gamma}(x))
\leq\tfrac{1}{m}(d(x,y)+d(y,x))+d(y,E_{\gamma}(y)),
\]
and letting $m \to \infty$, we have \[d(x,E_{\gamma}(x))\leq d(y,E_{\gamma}(y)).
\]
By exchanging the roles of $x$ and $y$, we also have $d(y,E_{\gamma}(y))\leq d(x,E_{\gamma}(x))$, and hence $d(y,E_{\gamma}(y))= d(x,E_{\gamma}(x))$. Again appealing to \cref{eq:s8_1}, we obtain $d(y,E_{m\gamma}(y))= d(x,E_{m\gamma}(x))$ as required.
\end{proof}

We suspect that the conclusions for \cref{lem:dtstrange} may be too peculiar to hold for general asymmetric metrics on $\T(S)$. Verifying this for the Teichm\"uller space of flat tori yields \cref{prop:torus}. To this end, we  recall that 
$\T(S_{1,0})=\mathbb{H}^2$, that in this case the mapping class group action is that of $\mathrm{PSL}_2(\mathbb Z)$, and that the earthquake paths in this case are the horocycles of the hyperbolic plane, see \cite{Ker-ICM}.

\begin{proposition}\label{prop:torus}
There is no $\mathrm{PSL}_2(\mathbb Z)$-invariant finite asymmetric metric on $\mathbb{H}^2$ where all anti-clockwise horocycles are geodesics.
\end{proposition}

\begin{proof}
In $\T(S_{1,0})=\mathbb{H}^2$, assume that there is an asymmetric metric $d$ with respect to which all earthquake paths, that is, anti-clockwise horocycles, are geodesics. Then, by \cref{lem:dtstrange}, for any $z\in\mathbb{H}^2$ the horocyclic segment
\[
\begin{bsmallmatrix}
1&t\\
0&1\end{bsmallmatrix}
\cdot z \quad (t\in[0,1]) \]
is a geodesic (for $d$) of the same length --- we normalise this length to be $1$. The oriented segments depicted in \cref{fig:45}, with lengths $a,b, c_1, c_2, d_1$ and $d_2$ are all anti-clockwise horocycles,  hence they are geodesics by assumption. By the triangle inequality, the lengths of the depicted oriented segments (see \cref{fig:45}) satisfy $a\leq c_1+d_2$ and $b\leq c_2+d_1$, hence $a+b\leq (c_1+c_2)+(d_1+d_2)$. We note the following.
\begin{itemize}
\item
The segment of length $a$ is the  image of a  path of length $1$ under $\begin{bsmallmatrix}0&1\\-1&0\end{bsmallmatrix}$:
\[
\begin{bsmallmatrix}0&1\\-1&0\end{bsmallmatrix}
\begin{bsmallmatrix}
1&t\\
0&1\end{bsmallmatrix}
\begin{bsmallmatrix}0&1\\-1&0\end{bsmallmatrix}^{-1}
\cdot (1+i) \quad \text{for }t\in[0,1],
\]
and hence $a=1$.

\item
The segment of length $b$ is the concatenation of two paths of length $1$:
\[
\begin{bsmallmatrix}
1&t\\
0&1\end{bsmallmatrix}
\cdot (-1+i)\quad \text{for }t\in[0,1]\quad
\text{ and }\quad
\begin{bsmallmatrix}
1&t\\
0&1\end{bsmallmatrix}
\cdot i
\quad \text{for }t\in[0,1],
\]
and hence $b=2$.

\item
the segment of length $c_1+c_2$ is the image of a subsegment of a length $1$ path under $\begin{bsmallmatrix}1&0\\-1&1\end{bsmallmatrix}$:
\[
\begin{bsmallmatrix}1&0\\-1&1\end{bsmallmatrix}
\begin{bsmallmatrix}
1&t\\
0&1\end{bsmallmatrix}
\begin{bsmallmatrix}1&0\\-1&1\end{bsmallmatrix}^{-1}
\cdot (2-\sqrt{3})i \quad  \text{for }t\in[0,1],
\]
and hence $c_1+c_2<1$;

\item
the segment of length $d_1+d_2$ is the image of a subsegment of a length $1$ path under $\begin{bsmallmatrix}1&0\\1&1\end{bsmallmatrix}$:
\[
\begin{bsmallmatrix}1&0\\1&1\end{bsmallmatrix}
\begin{bsmallmatrix}
1&t\\
0&1\end{bsmallmatrix}
\begin{bsmallmatrix}1&0\\1&1\end{bsmallmatrix}^{-1}
\cdot (2+\sqrt{3})i \quad \text{for }t\in[0,1],
\]
and hence $d_1+d_2<1$.
\end{itemize}
This is impossible, for it would cause $3=a+b\leq c_1+c_2+d_1+d_2<2$.
\begin{center}
\begin{figure}[ht]
\includegraphics[scale=0.3]{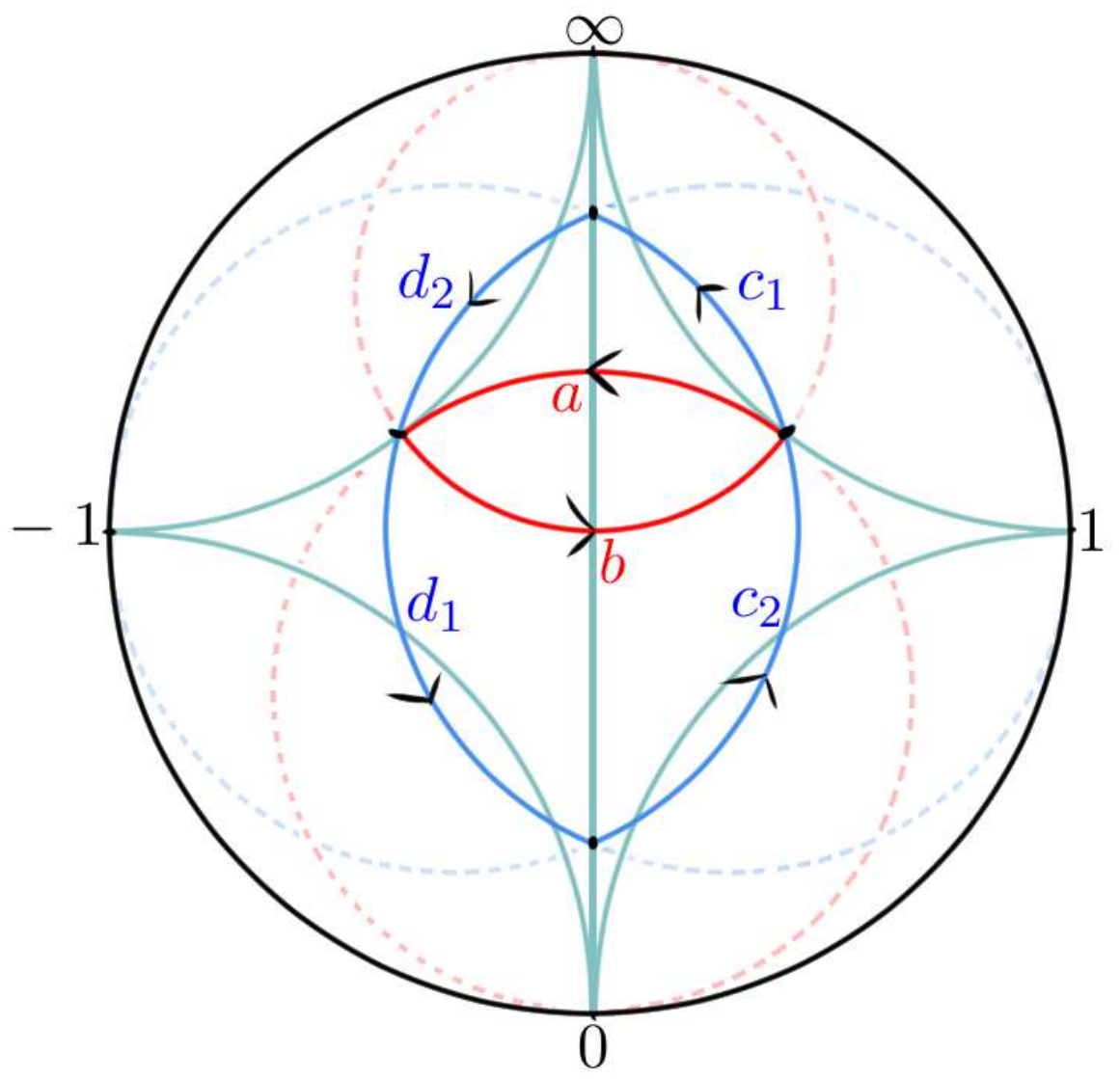}
\caption{Violation of triangle inequality.}
\label{fig:45}
\end{figure}
\end{center}
\end{proof}

\begin{remark} 
We finish off this subsection by mentioning the 1932 article  \cite{Busemann1932} by Herbert Busemann (with an English translation in \cite{Busemann1932-translation}). The setting is that of a metric space to which one adds some very general axioms such as existence of midpoints for any pair of points and extension of geodesics. Busemann calls these spaces ``line spaces". This is the setting in which he placed most of his later works (with different terminology).  In the article mentioned, Busemann studies spaces satisfying an additional axiom, namely  he considers \emph{geometries where circles of infinite radius are the shortest lines} (this is the title of the paper). The expression ``circle of infinite radius" designates a horosphere, that is, the limit of a sequence of geometric circles in a plane passing through a
given point with centres going to infinity along a ray starting at this point.  Busemann shows that the metrics satisfying the property in the title are generalisations of the Euclidean metric, namely, they are associated with a norm on a vector space. Now the relation between earthquake paths and horocycles which we mentioned in \cref{s:horo} makes a connection between Busemann's intuition and the study of a metric on Teichm\"uller space in which earthquakes are geodesics. Such a metric, if it exists would have Euclidean behaviour over the whole Teichmüller space.
This, together with  \cref{prop:torus} leads us to the following 
\end{remark}
\begin{conjecture}
There is no asymmetric metric on Teichmüller space which is mapping class group invariant and for which left earthquakes are geodesics.  
\end{conjecture}

\subsection{General case}

\begin{proof}[Proof of \cref{thm:longquake:nongeo} for general earthquakes]
	Let $\lambda$ be an arbitrary measured lamination and $x\in\T(S)$ an arbitrary point in  Teichm\"uller space.  Consider the earthquake ray $E_{t\lambda}(x)\quad(t\geq0)$.  By \cref{Kerckhoff cosine}, for any measured lamination $\alpha$, we have
\begin{align*}
\ell_\alpha(E_{t\lambda}(x))\leq \ell_\alpha(x)+t\cdot i(\alpha,\lambda). 
\end{align*}
By \cref{eq:thurston:ratio}, the Thurston distance $\dth(x,E_{t\lambda}(x))$ satisfies 
\begin{align*}
\dth(x,E_{t\lambda}(x))
&= \log \sup_{\alpha\in \ml(S)} \frac{\ell_\alpha(E_{t\lambda}(x))}{\ell_\alpha(x)}\\
&\leq\log \sup_{\alpha\in \ml(S)} \frac{\ell_\alpha(x)+t\cdot i(\alpha,\lambda)}{\ell_\alpha(x)}\\
&=\log \sup_{\alpha\in \ml} \left(1+t\frac{ i(\alpha,\lambda)}{\ell_\alpha(x)}  \right)\\
&=\log \left(1+ t\cdot \max_{\mu \in \ml(S) (\ell_\alpha(x)=1) } i(\mu,\lambda)\right).
\end{align*}
Combined with \cref{cor:e<wp}, this implies that
\begin{align}\label{eq:dist:earthquakepath}
		d_e(x,E_{t\lambda}(x))
		\leq C_2 \dth(x,E_{t\lambda}(x))
		\leq C_2 \log  \left(1+ t\cdot \max_{\mu\in\ml(S) (\ell_\alpha(x)=1)} i(\mu,\lambda)\right).
\end{align}
We remark that $\max_{\mu \in \ml(S) (\ell_\alpha(x)=1)}i(\mu,\lambda)$ is a constant depending only on $x$.
On the other hand, the earthquake metric length of the earthquake path $\{E_{s\lambda}(x)\}_{s\in[0,t]}$ is $t\ell_\alpha(x)$, which, for large $t$, is bigger than $d_e(x,E_{t\lambda}(x))$, is bounded above by \cref{eq:dist:earthquakepath}. Thus, the earthquake path $\{E_{s\lambda}(x)\}_{s\in[0,t]}$ fails to be an earthquake metric geodesic when $t$ is big enough.
\end{proof}

\begin{remark}
It is possible to push the proof strategy for the Fenchel--Nielsen case to earthquakes with respect to general laminations.
Instead of pinching closed geodesics, we consider harmonic stretch lines defined by Pan--Wolf  \cite{PanWolf2022}.
As is shown in \cite[Proposition 13.10]{PanWolf2022}, harmonic stretch lines and earthquake paths along the same measured lamination \lq\lq commute'' with each other, and we can also control earthquake distances in terms of Thurston distances by \cref{eq:e:th}. 
Therefore we can repeat the same argument as \cref{DT}.
\end{remark}

\begin{remark}
\cref{thm:longquake:nongeo} (at least) for the case of closed surfaces can also be obtained by combining results in  \cite{BST2017} and our inequality in \cref{cor:e<wp}.
For $x, y \in \T(S)$, let $\lambda$ be the measured lamination such that $E_\lambda(x)=y$.  By Theorem C and Theorem F in \cite{BST2017}, it follows that
\begin{equation*}
\exp\left(\tfrac{a}{|\chi(S)|}\dwp(x,y)-b|\chi(S)|\right)
\leq 
\tfrac{1}{4}\ell_{\lambda}(x)
+
\tfrac{\pi^2}{2}|\chi(S)|.
\end{equation*}
Combined with \cref{cor:e<wp}, this implies that if $d_e(x,y)$ is sufficiently large, then
\begin{eqnarray*}
	\ell_\lambda(x)&\geq& \tfrac{1}{4}\ell_{\lambda}(x)+\tfrac{\pi^2}{2}|\chi(S)|\\
	&\geq & 	\exp\left(\tfrac{a}{|\chi(S)|}\dwp(x,y)-b|\chi(S)|\right)\\
	&\geq & \exp\left(\tfrac{a}{|C_1\chi(S)|}d_e(x,y)-b|\chi(S)|\right)
	\\
	&\geq& \exp (C d_e(x,y))
\end{eqnarray*}
for some topological constant $C$. 
Hence, $d_e(x,y)\leq C^{-1}\log \ell_\lambda(x)$.
\end{remark}

\subsection{Regularity}

We end this section with a \emph{positive} result about the geodesics with respect to the earthquake metric for some special surfaces.

\begin{proposition}
Geodesics of the earthquake metric on $\T(S_{1,1})$ and $\T(S_{0,4})$ are $C^1$.
\end{proposition}

\begin{proof}
Theorem~d of Guillaume Buro's doctoral dissertation \cite{burothesis} asserts that on a two-dimensional differentiable manifold with a Finsler metric which is strictly convex and (locally) Lipschitz, its metric geodesics are $C^1$. We previously showed the strict convexity in \cref{cor:strictconvexity} and the (local) Lipschitzness in \cref{thm:lipregular}.
\end{proof}

%%==New section===
\newpage
\section{Distance to and from the boundary}
\label{sec:asyp:dist}

The aim of this section is to show that the earthquake metric is incomplete (see \cref{defn:FDcomplete} for a definition of completeness in the setting of asymmetric metrics), and to quantitatively study this incompleteness. Mumford's compactness theorem \cite{mumford} and \cref{prop:topology:finsler} tells us that incompleteness can only happen as the systoles of a family of surfaces geometrically converge  to a cusp, and we use the Weil--Petersson completion as a convenient language for discussing sequences which penetrate into arbitrarily thin parts of moduli space.

\begin{theorem}\label{thm:dist:boundary:wp}
Let $\overline{\T(S)}^\mathrm{WP}$ denote the completion of Teichm\"uller space with respect to  the Weil--Petersson metric. Then there exists $C>1$ such that  for any $x\in\T(S)$,
\begin{equation*}
	2\ell_{\mathrm{sys}}(x)\mathrm{Log} \frac{1}{\ell_{\mathrm{sys}}(x)}\leq d_e(x,\partial \overline{\T(S)}^{\mathrm{WP}})\leq 2C \ell_{\mathrm{sys}}(x)\mathrm{Log} \frac{1}{\ell_{\mathrm{sys}}(x)}. 
\end{equation*}
Moreover,  as $\ell_{\mathrm{sys}}(x)\to 0$ we have
\begin{equation*}
\frac{d_e(x,\partial \overline{\T(S)}^{\mathrm{WP}})}{2\ell_{\mathrm{sys}}(x)~\mathrm{Log} \frac{1}{\ell_{\mathrm{sys}}(x)}}\to 1.
\end{equation*}
The same hold for $d_e(\partial \overline{\T(S)}^{\mathrm{WP}},x)$.
\end{theorem} 

\begin{remark}
	As we shall see in the proof, the conclusion also holds if we replace Teichm\"uller space by the moduli space.
\end{remark}

\begin{remark}[width function]
Many of the bounds  in this section are obtained by using the fact that geodesic segments must spend some definite time every time they pass through the collar neighbourhood of some simple closed geodesic. We use the following notation:
\begin{itemize}
\item
let $w(t)=\mathrm{arcsinh}(\mathrm{cosech}\frac{t}{2})$ be the width of a ``standard'' collar neighbourhood of a simple closed curve with length $t>0$, and

\item
we  use $\mathrm{Log}(\frac{1}{t}):=\max\{1,\log(\frac{1}{t})\}$ as an approximation for $w(t)$.
\end{itemize}
We note that both functions are non-increasing in $t$.
\end{remark}

%+=========
\subsection{Lower bound for the earthquake metric}

\begin{lemma}
\label{lem:lowerbound:emetric}
For any $x,y\in\T(S)$ and  any simple closed curve $\alpha$,  we have
\begin{align*}
d_e(x,y)
\geq 
2\int^{\max\{\ell_\alpha(x),\ell_\alpha(y)\}}_{\min \{\ell_\alpha(x),\ell_\alpha(y)\}} w(\ell)\,\mathrm{d}\ell.
\end{align*}
\end{lemma}
 
\begin{proof}
Let $\sigma:[0,1]\to \T(S)$ be a $C^1$-path connecting $x$ to $y$. We denote its length with respect to $d_e$ by $L(\sigma)$. By \cref{prop:spatial},  there is a continuous family $\beta_t \ (t\in[0,1])$ of measured laminations such that the tangent vector $\dot \sigma(t)$ is equal to the earthquake vector $\v_{\sigma(t)}(\beta_t)$ for every $t \in [0,1]$. 
It follows from \cref{Kerckhoff cosine} that
\begin{align*}
\text{for any } t\in[0,1],\quad
\ell_\alpha(\sigma(t'))-\ell_\alpha(\sigma(t))=\Cos_{\sigma(t)}(\beta_t,\alpha) (t'-t) +o(t'-t)
\end{align*}
as $t'\to t$. In particular, for every $\epsilon>0$, there exists a neighbourhood of $t$ such that 
\begin{align*}
|\ell_\alpha(\sigma(t'))-\ell_\alpha(\sigma(t''))|\leq i(\alpha,\beta_t)|t'-t''|+\epsilon|t'-t''|
\end{align*}
for any $t'$ and $t''$ in this neighbourhood. 
We subdivide the interval $[0,1]$ into finer intervals of the form $[t_i,t_{i+1}]$, where $0=t_0<t_1<\cdots <t_n=1$, so that for every $i=0,\ldots,n-1$,
\begin{equation}
\label{l-int-ineq}
|\ell_{\alpha}(\sigma (t_{i+1}))-\ell_{\alpha}(\sigma({t_i}))|
\leq 
i(\alpha,\beta_{t_i})(t_{i+1}-t_i)+\epsilon(t_{i+1}-t_i).
\end{equation}
We next make use of the fact that geodesic segments gain some extra length at least equal to $2w(\ell_\alpha(\sigma))$ each time they pass through the collar neighbourhood of $\alpha$, to obtain a lower bound of $(t_{i+1}-t_i) \ell_{\beta_{t_i}}(\sigma(t_i))$, which is a constituent of a Riemann sum approximating the earthquake metric length integral for $\sigma$. 
Set $\hat\ell_\alpha:=\min\{\ell_\alpha(\sigma(t)) \mid t\in[0,1]\}$. The fact that the function $w$ is monotone non-increasing, combined with \cref{l-int-ineq}, implies  that
\begin{align*}
(t_{i+1}-t_i) \ell_{\beta_{t_i}}(\sigma(t_i))
&\geq2(t_{i+1}-t_i)i(\alpha,\beta_{t_i}) w(\ell_\alpha(\sigma(t_i)))\\
&\geq2|\ell_{\alpha}(\sigma(t_{i+1}))-\ell_{\alpha}(\sigma(t_i))|w(\ell_\alpha(\sigma(t_i)))
-2\epsilon(t_{i+1}-t_i)w (\hat\ell_\alpha).
\end{align*}
Hence, a Riemann sum approximating the length $\mathrm{L}(\sigma)$  satisfies
\begin{align*}
\sum_{i=0}^{n-1}(t_{i+1}-t_i)\ell_{\beta_{t_i}}(\sigma(t_i))
&\geq2\sum_i \left|	\ell_{\alpha}(\sigma(t_{i+1}))-\ell_{\alpha}(\sigma(t_i))\right| w(\ell_\alpha(\sigma(t_i)))
-2\epsilon \cdot w (\hat\ell_\alpha).
\end{align*}
Letting $n\to\infty$ and $\max|t_{i+1}-t_i|\to0$, we see that :\begin{align*}
\mathrm{L}(\sigma)
&=\lim_{n\to\infty} \sum_{i=0}^{n-1}(t_{i+1}-t_i) \ell_{\beta_{t_i}}\\
&\geq\lim_{i\to\infty}2\sum_i \left|\ell_{\alpha}(\sigma(t_{i+1}))-\ell_{\alpha}(\sigma(t_i))\right| w(\ell_\alpha(\sigma(t_i)))
-2\epsilon \cdot w (\hat\ell_\alpha) \\
&=2\int_\sigma  w(\ell_\alpha(\sigma(t)))\,|\mathrm{d}\ell_\alpha(\sigma(t))|-2\epsilon \cdot w(\hat\ell_\alpha)\\
&\geq 2\int^{\max\{\ell_\alpha(x),\ell_\alpha(y)\}}_{\min \{\ell_\alpha(x),\ell_\alpha(y)\}} w(\ell)\,\mathrm{d}\ell
-2\epsilon \cdot w (\hat\ell_\alpha).
\end{align*}
Taking  $\epsilon\to0$ yields 
\begin{align*}
\mathrm{L}(\sigma)
\geq
2\int^{\max\{\ell_\alpha(x),\ell_\alpha(y)\}}_{\min \{\ell_\alpha(x),\ell_\alpha(y)\}} w(\ell)\,\mathrm{d}\ell.
\end{align*}
Since $d_e(x,y)$ is defined as the infimal length $\mathrm{L}(\sigma)$ over all $C^1$ paths $\sigma$ joining $x$ and $y$, this gives the requisite lower bound for $d_e(x,y)$.
\end{proof}

\subsection{Upper bound of distance to and from the boundary}

We next establish the following upper bound for distances to and from the boundary.

\begin{theorem}[upper bound]
\label{prop:asymp:dist:upper}
For all $\epsilon>0$, there exists $\delta>0$ such that for any $x\in\T(S)$ with $\ell_{\mathrm{sys}}(x)<\delta$, we have
\begin{align}
d_e(x,\partial \overline{\T(S)}^{\mathrm{WP}})&\leq 2(1+\epsilon) \ell_{\mathrm{sys}}(x)\log \tfrac{1}{\ell_{\mathrm{sys}}(x)},\text{ and }\label{eq:distanceto}\\
d_e(\partial \overline{\T(S)}^{\mathrm{WP}},x)&\leq 2(1+\epsilon) \ell_{\mathrm{sys}}(x)\log \tfrac{1}{\ell_{\mathrm{sys}}(x)}\label{eq:distancefrom}.
\end{align}
\end{theorem}

We establish the upper bound constructively. We shall consider  Fenchel-Nielsen twists around simple closed curves $\beta$ which wind around some short curve $\alpha$ sufficiently many times to make the angles of intersection between $\alpha$ and $\beta$ close to $0$, and we show that this process causes the length of $\alpha$ to converge to $0$ in finite time.  For the remainder of this subsection, we adopt the following notation.
Given an arbitrary $x\in\T(S)$,
\begin{itemize}
\item
let $\Gamma$ be a short pants decomposition for $x$, i.e.\  $\Gamma$ consists of curves shorter than the Bers constant  (see, for instance, \cite[Chapter 5]{Buser2010});
\item
let $\alpha\in\Gamma$ denote a systole of $x$, whose length is shorter than some $\delta$ (which we choose so that $\alpha$ is shorter than the width of its collar neighbourhood);
\item
let $S_\alpha$ be the connected component of $S\setminus(\Gamma\setminus\{\alpha\})$ which is not a pair of pants.

\end{itemize}

\begin{center}
\begin{figure}[ht]
\includegraphics[scale=0.5]{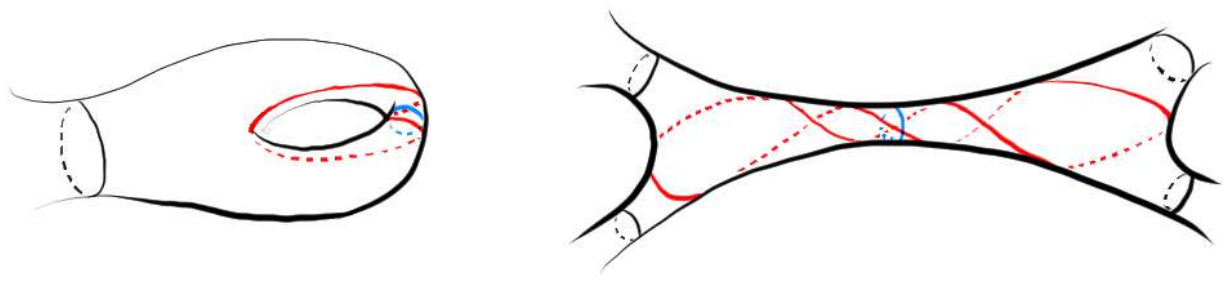}
\caption{The only two possible configurations for {\color{blue}$\alpha$} (short curve) and {\color{red}$\beta$} (long curve) on $S_\alpha$, up to homeomorphism.}
\label{fig:pg48}
\end{figure}
\end{center}
\begin{lemma}\label{lem:angles}
Let $\alpha$ and $\beta$  be two simple closed geodesics on $x\in \T(S)$ intersecting at more than one point. For any two intersection points $p_1,p_2\in \alpha\cap\beta$, their respective intersection angles $\theta_1,\theta_2\in(0,\pi)$ measured counterclockwise from $\alpha$ to $\beta$ satisfy
\begin{align*}
e^{-\ell_\alpha(x)/2}<\frac{\sin \theta_1}{\sin \theta_2}<e^{\ell_\alpha(x)/2}.
\end{align*}
\end{lemma}

\begin{proof}
Position a lift of $\alpha$ in the upper half plane $\mathbb{H}^2$ to the geodesic $\overline{0\infty}$ joining $0$ and $\infty$. Further set a lift of $\beta$ to be the geodesic $\overline{-1r}$ joining $-1$ and $r>0$, so that $\overline{0\infty}\cap\overline{-1r}$ is a lift of $p_1$ and hence the angle from $\overline{0\infty}$ to $\overline{-1r}$ is $\theta_1$. By a standard formula relating cross-ratios and intersection angles (see, e.g., \cite[Equation~(1)]{Ker}), we have
\begin{align}
\cos\theta_1=2(0,r;-1,\infty)-1=\frac{1-r}{1+r}.\label{eq:theta1}
\end{align}
One nearest translate, with respect to $\alpha$, of $\overline{-1r}$ is the geodesic $\overline{(-e^{\ell_\alpha(x)})(re^{\ell_\alpha(x)})}$ joining $-e^{\ell_\alpha(x)}$ and $re^{\ell_\alpha(x)}$. Since the segment on $\overline{0\infty}$ between its intersections with $\overline{-1r}$ and $\overline{(-e^{\ell_\alpha(x)})(re^{\ell_\alpha(x)})}$ is a fundamental domain, there must be a lift $\overline{uv}$ of $\beta$, with end points $u\in(-e^{\ell_\alpha(x)},-1)$ and $v\in(r,re^{\ell_\alpha(x)})$, such that $\overline{0\infty}\cap\overline{uv}$ is a lift of $p_2$. Note that the angle from $\overline{0\infty}$ to $\overline{uv}$ is $\theta_2$, and hence
\begin{align}
\cos\theta_2=2(0,v;u,\infty)-1=\frac{1+\frac{v}{u}}{1-\frac{v}{u}}=\frac{1-\frac{v}{|u|}}{1+\frac{v}{|u|}}.\label{eq:theta2}
\end{align}
Since $\theta_i\in(0,\pi)$, the quantities $\sin\theta_i$ are positive and \cref{eq:theta1} and \cref{eq:theta2} mean that
\begin{align}
\sin\theta_1=\frac{2}{\sqrt{r}+\sqrt{r^{-1}}},
\quad\text{and}\quad
\sin\theta_2=\frac{2}{\sqrt{v/|u|}+\sqrt{|u|/v}}.
\end{align}
The latter quantity $\sin\theta_2$ is infimised by supremising the denominator, which is given by the following choices: 
\begin{itemize}
\item
$u=-1$ and $v=re^{\ell_\alpha}$ when $r\geq1$, or
\item
$u=-e^{\ell_\alpha}$ and $v=r$ when $r\leq1$.
\end{itemize}
In either case, 
\[
\sin\theta_2
>
\frac{2}{\sqrt{re^{\ell_\alpha}}+\sqrt{r^{-1}e^{\ell_\alpha}}}
=
\frac{\sin\theta_1}{e^{\ell_\alpha/2}},
\]
and hence $\frac{\sin\theta_1}{\sin\theta_2}<e^{\ell_\alpha/2}$. By symmetry, $\frac{\sin\theta_2}{\sin\theta_1}<e^{\ell_\alpha/2}$, and the result follows.
\end{proof}

\begin{definition}[minimal angle of intersection]\label{def:nimimal:angle}
For any two intersecting simple closed geodesics $\gamma_1,\gamma_2$ on $(S,x)$, we define $\theta_x(\gamma_1,\gamma_2)\in[0,\pi)$ as the angle, measured anti-clockwise from $\gamma_1$ to $\gamma_2$, which is minimal among all the intersection points in $\gamma_1\cap\gamma_2$.
\end{definition}

\begin{lemma}\label{lem:length:angle}
Let $x\in\T(S)$ be a hyperbolic surface with a shortest pants decomposition $\Gamma$. Suppose that $\beta$ is a simple closed curve which intersects $\alpha\in\Gamma$ minimally and which is disjoint from other curves in $\Gamma$.  Then there exists a  constant $C$ which only depends on the topology of $S$ such that 
\begin{align}
\frac{\ell_{\beta}(x)}{2i(\alpha,\beta)}\leq\log \frac{1}{\ell_{\alpha}(x)}+\log\frac{1}{\sin\theta}+C
\end{align}
where $\theta$ is the intersection angle from $\alpha$ to $\beta$ at any point of $\alpha\cap \beta$.
\end{lemma}

%
%\begin{remark}
%The angle condition can be dropped, and one can also derive a lower bound of the form $\log \frac{1}{\ell_{\mathrm{alpha}}(x)}+\log\frac{1}{\sin\theta_x(\alpha,\beta)}-C$ whenever $\alpha$ is sufficiently short, but we retain the present form of \cref{lem:length:angle} to keep things simple.
%\end{remark}

\begin{proof}
The connected component $S_\alpha$ is either a one-holed torus or a four-holed sphere (\cref{fig:pg48}), possibly with cusps in the latter case. We prove the claim for the four-holed sphere; the proof for the one-holed torus case can be similarly derived, or one can also quotient a one-holed torus by its hyperelliptic involution and then lift  the resulting orbifold to a four-holed sphere to derive the result for the one-holed torus via the result for the four-holed sphere.

Consider a pair of pants $P$ constituting a component of $S_\alpha\setminus \alpha$. Let 
\begin{itemize}
\item $b$ denote the length of a component $\rho$ of  $\beta\cap P$, and

\item $d$ denote the length of the unique orthogeodesic $\varrho$ on $P$ with both endpoints on $\alpha$.
 In particular, $\rho$ and $\varrho$ are freely homotopic, with endpoints allowed to move on $\alpha$.
\end{itemize}

Take  lifts $\tilde{\varrho},\tilde\rho$ of $\varrho, \rho$ which intersect on the universal cover $\hyperbolic^2$ (see \cref{fig:pg49}).
They stretch between two lifts $\tilde{\alpha}_1$, $\tilde{\alpha}_2$ of $\alpha$. 
For $i=1,2$, let $h_i$ denote the length, along $\tilde{\alpha}_i$, of the segments joining the endpoints of the lifts of $\tilde \varrho$ and $\tilde{\beta}_0$, and let $\phi_i$ denote the anti-clockwise angle from $\tilde{\alpha}_i$ to $\tilde\rho$. 
Depending on whether  $\tilde\rho$ and $\tilde\varrho$ intersect or not,  there are two cases which we should consider.

\begin{center}
\begin{figure}[ht]
\includegraphics[scale=0.3]{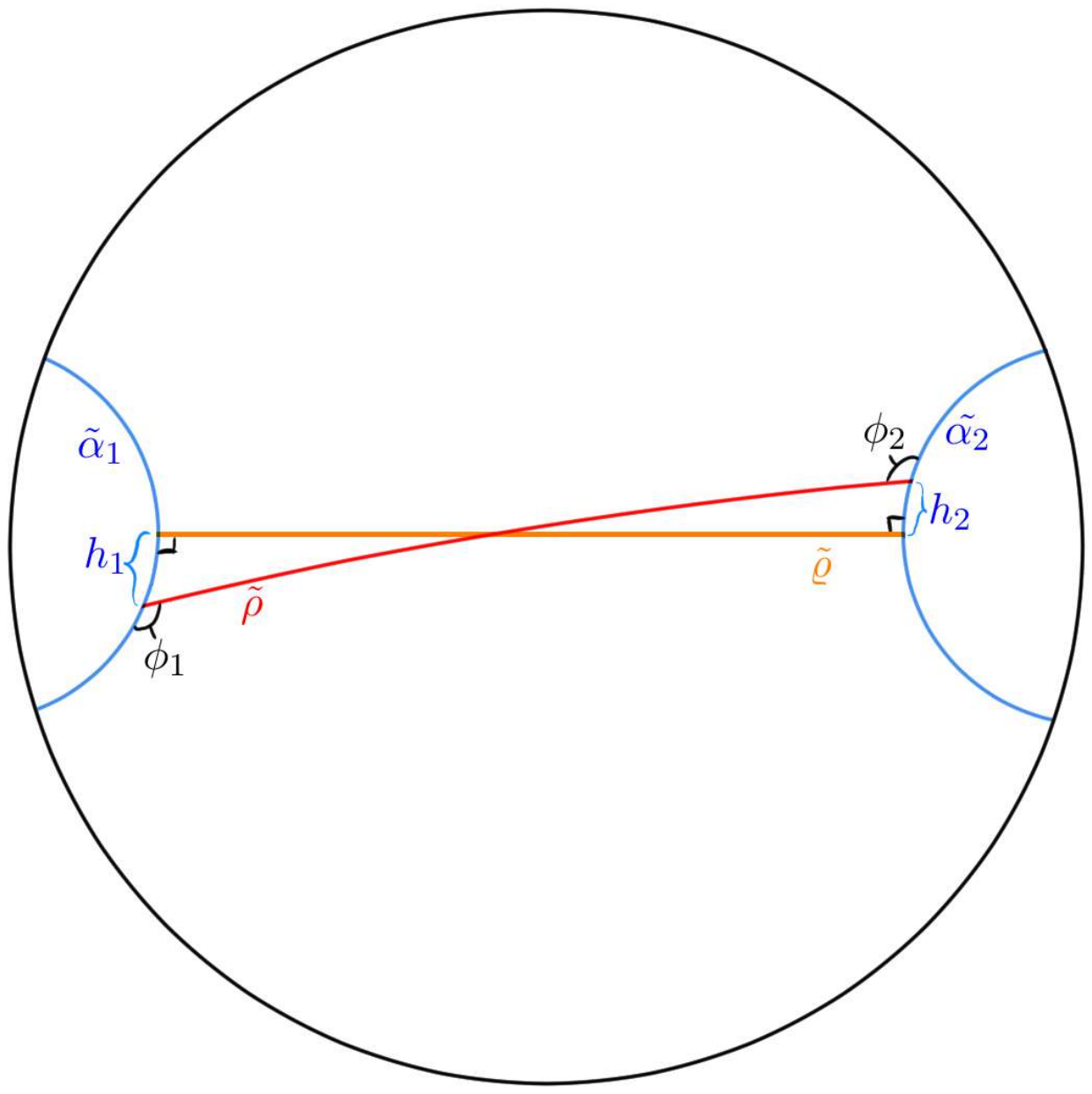}
\caption{Depiction of how $\tilde{\varrho}$, $\tilde\rho$, $\tilde{\alpha}_1$, and $\tilde{\alpha}_2$ are configured.}
\label{fig:pg49}
\end{figure}
\end{center}

\noindent
Case~(1): when $\tilde\rho$ and $\tilde\varrho$ intersect each other.

\noindent
By \cite[Eq.~(2.3.2) in Ch.~2]{Buser2010},
\begin{equation}
\begin{split}
\label{cosh}
\cosh b=\cosh h_1\cosh h_2\cosh d+\sinh h_1\sinh h_2 \\
\leq \cosh h_1\cosh h_2 (\cosh d+1).
\end{split}
\end{equation}
By \cite[Formula glossary, 2.2.2 (v)]{Buser2010}, we have
\[
\cosh h_i<\frac{1}{\sin\phi_i}, \quad i=1,2.
\]
Let $\theta$ be the  angle from $\alpha$ to $\beta$ at any point of $\alpha \cap \beta$. 
By Lemma \ref{lem:angles}, we see that 
\begin{equation*}
e^{-B/2}\leq 	\frac{\sin\phi_i}{\sin\theta}\leq e^{B/2}
\end{equation*}
 where $B$ is the Bers constant, which depends only on the topological type of $S$. Hence,
 \[
\cosh h_i<\frac{1}{\sin\phi_i}\leq \frac{e^{B/2}}{\sin\theta}.
\]
 Combined with \eqref{cosh}, this implies 
\begin{align}\label{eq:anglebounde} 
\cosh b<\frac{e^B(\cosh d +1)}{\sin^2 \theta}.
\end{align}

\noindent
Case~(2): when $\tilde\rho$ and $\tilde\varrho$ are disjoint. 

\noindent
By \cite[Eq.~(2.3.2) in Ch.~2]{Buser2010}, we have
\begin{eqnarray}
	\cosh b&=&\cosh h_1\cosh h_2\cosh d -\sinh h_1\sinh h_2\nonumber \\
	&\leq& \cosh h_1\cosh h_2\cosh d. \label{eq:coshb}
\end{eqnarray} 
Let $\tilde\eta$ be the geodesic segment perpendicular to both $\tilde\rho$ and $\tilde\varrho$. This segment cuts $\tilde\rho$ (resp. $\tilde\varrho$) into two subsegments of length $b_1$ and $b_2$ (resp. $d_1$ and $d_2$) respectively. In particular, $b=b_1+b_2$ and $d=d_1+d_2$. By \cite[Formula Glossary, 2.3.1 (iii)]{Buser2010}, 
\begin{eqnarray*}
	\cosh d_1&=&\cosh b_1\sin\phi_1\\
	\cosh d_2&=&\cosh b_2\sin\phi_2. \label{eq:coshb}
\end{eqnarray*}
Hence, 
\begin{eqnarray*}
	\cosh b\leq 2 \cosh b_1 \cosh b_2=\frac{2\cosh d_1\cosh d_2}{\sin\phi_1\sin\phi_2}\leq \frac{2\cosh d}{\sin\phi_1\sin\phi_2}.
\end{eqnarray*}
Combined with Lemma \ref{lem:angles}, this implies that
\begin{equation}\label{eq:anglebounde2}
	\cosh b\leq \frac{2e^B\cosh d}{\sin^2\theta},
\end{equation}
where $B$ is the Bers constant.\medskip 

Combining the conclusions \cref{eq:anglebounde} and \cref{eq:anglebounde2} from the two cases:
\begin{equation}\label{eq:anglebounde3}
	\cosh b\leq \frac{2e^B\cosh d}{\sin^2\theta}.
\end{equation}

Next, we estimate $d$ from above, which in turn gives the desired upper bound for $b$.
Denote the lengths of the non-$\alpha$ boundaries of $P$ by $L_1,L_2$, and let $a_i$ be the length of the subsegment of $\alpha$ between the orthogeodesic $\varrho$ and the orthogeodesic joining $\alpha$ to the length $L_i$ side of $P$ (see \cref{fig:pg50}). 

\begin{center}
\begin{figure}[ht]
\includegraphics[scale=0.3]{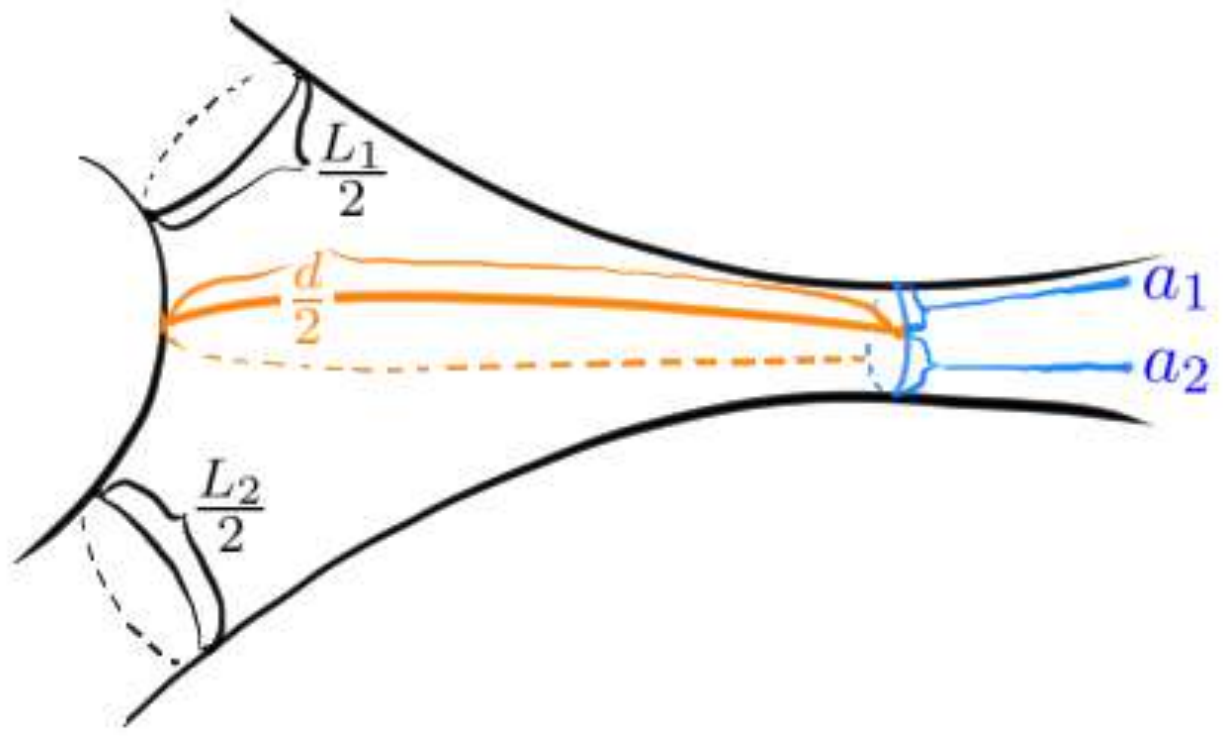}
\caption{Geodesic segments on $P$.}
\label{fig:pg50}
\end{figure}
\end{center}

Since $a_1+a_2=\frac{\ell_\alpha(x)}{2}$, one of these two segments has length at least $\tfrac{\ell_\alpha(x)}{4}$. Combine this with \cite[Formula glossary, 2.3.4 (i)]{Buser2010}, and we get
\begin{align}
\sinh(\tfrac{d}{2})
=\frac{\cosh \frac{L_i}{2}}{\sinh a_i}\leq\frac{\cosh B}{\sinh \tfrac{\ell_\alpha(x)}{4} },\label{eq:orthogeobound}
\end{align}
where $B$ is the Bers constant for $S$. Algebraically combining \cref{eq:anglebounde} and \cref{eq:orthogeobound},  we have 
\begin{align}\label{eq:length:b}
b\leq2\left(\log  \frac{1}{\ell_{\alpha}(x)}+\log\frac{1}{\sin\theta}+C\right)
\end{align}
for a constant $C$ which depends only on the topological type of $S$. Note that $\Gamma$ cuts $\beta$ into $i(\alpha,\beta)$ segments, each of which satisfies \eqref{eq:length:b}. The lemma follows.
\end{proof}

\begin{proof}[Proof of \cref{prop:asymp:dist:upper}]
We only need to establish \cref{eq:distanceto}, as \cref{eq:distancefrom} follows by doing the same proof by for the right earthquake metric. We also note that if \cref{eq:distanceto} is true for all $\epsilon>0$ smaller than some constant, then \cref{prop:asymp:dist:upper} holds true.

We prove \cref{eq:distanceto} by explicitly constructing a path $x_t$ of controlled earthquake metric length starting at an arbitrary $x_0=x$ with small systole $\alpha$, so that $\ell_\alpha(x_t) \to 0$ along the path. 
Fix an angle $\vartheta\in(\frac{\pi}{2}, \pi)$. 
Let $\beta_0$ be a simple closed geodesic on $S_\alpha$ that minimally intersects $\alpha$ and has $\theta_x(\alpha,\beta)>\vartheta$ (see \cref{def:nimimal:angle} for the definition of $\theta_x(\alpha,\beta)$). We produce a piecewise earthquake path $x_t$ in $\T(S)$ as follows.
\begin{itemize}
\item
We proceed along $x_s:=E_{s\beta_0}(x_0)$. By \cite[Proposition~{3.5}]{Ker}, the minimal intersecting angle $\theta_0(s):=\theta_{x_s}(\alpha,\beta_0)$ monotonically decreases, and we stop when the angle reaches $\vartheta$. Denote the time this occurs by $t_1$, and let $\beta_1$ be the left Dehn twist\footnote{We remind the reader that the action of left Dehn twists on the Teichm\"uller space corresponds to right Fenchel-Nielsen twists.} of $\beta_0$ with respect to $\alpha$, \cite[Proposition~{3.5}]{Ker} then ensures that $\theta_1(t_1):=\theta_{x_{t_1}}(\alpha,\beta_1)$ is greater than $\vartheta$. 

\item
We continue along $x_{s+t_1}:=E_{s\beta_1}(x_{t_1})$ until $\theta_{x_{s+t_1}}(\alpha,\beta_1)=\vartheta$. Denote the stopping time by $t_2$, and let $\beta_2$ be the left Dehn twist of $\beta_1$ with respect to $\alpha$. Again, $\theta_2(t_2):=\theta_{x_{t_2}}(\alpha,\beta_2)$ is greater than $\vartheta$.

\item
We iterate the above process indefinitely.
\end{itemize}

This procedure yields a path where we are always earthquaking with respect to some $\beta_i$ that intersects $\alpha$ at angle
\begin{align}
\theta_i(s)\in (\vartheta, \pi)
\text{ for $s\in[t_i,t_{i+1})$, and }
\theta_i(t_{i+1})=\vartheta.
\label{eq:exactvartheta}
\end{align}
In particular, $\theta_{x_{t_i}}(\alpha, \beta_{i-1})=\theta_{i-1}(t_i)=\vartheta$. 
Hence, by \cref{lem:length:angle} we obtain, 
\begin{align*}
\ell_{\beta_{i-1}}(x_{t_{i}})
\leq
2i(\alpha,\beta_{i-1})\left(\log\frac{1}{\ell_{\alpha}(x_{t_{i}})}+\log\frac{1}{\sin\vartheta}+C\right).
\end{align*}
Recall that $\beta_i$ is a left Dehn twist of $\beta_{i-1}$ and that $\ell_\alpha(x_{t_i})$ is decreasing. Therefore, for any $s\in[t_i,t_{i+1})$, 
\begin{align*}
\ \ell_{\beta_{i}}(x_{s})=\ell_{\beta_{i}}(x_{t_{i}}) 
&\leq \ell_{\beta_{i-1}}(x_{t_{i}}) +i(\alpha, \beta_{i-1} )\ell_{\alpha}(x_{t_i})\\
\leq&\
2i(\alpha,\beta_{i-1})\left(\log\frac{1}{\ell_{\alpha}(x_{t_{i}})}+\log\frac{1}{\sin\vartheta}+C+\ell_\alpha(x_{t_i})\right)\\
\leq&\ 2i(\alpha,\beta_{i-1})\left(\log\frac{1}{\ell_{\alpha}(x_{s})}+\log\frac{1}{\sin\vartheta}+C+\ell_\alpha(x_{0})\right)\\
\leq &\ 2i(\alpha,\beta_{i})\left(\log\frac{1}{\ell_{\alpha}(x_{s})}+\log\frac{1}{\sin\vartheta}+C+\ell_\alpha(x_{0})\right).
\end{align*}

By Wolpert's cosine formula \cite[Corollary 2.12]{Wolpert1982}, 
at each interval $[t_i,t_{i+1})$, the length of $\alpha$ changes at the rate of $\frac{\mathrm{d}\ell_\alpha}{\mathrm{d}s}<i(\alpha,\beta_i)\cos(\vartheta)<0$ for $s \in [t_i, t_{i+1})$, hence is monotone decreasing, and we have moreover
\begin{equation*}
	\left|\frac{\mathrm d \ell_\alpha}{\mathrm ds}\right|\geq i(\alpha,\beta_i)|\cos(\vartheta)|>0.
\end{equation*}

We bound the length of the piecewise earthquake path $(x_s)$ as follows:
\begin{align*}
&\sum_{i=0}^\infty
\int_{t_i}^{t_{i+1}}\|\dot{x}_s\|_e\ \mathrm{d}s
=
\sum_{i=0}^\infty
\int_{t_i}^{t_{i+1}}\|\v_{\beta_i}(x_s)\|_e\ \mathrm{d}s
=
\sum_{i=0}^\infty
\int_{t_i}^{t_{i+1}}\ell_{\beta_i}(x_s)\ \mathrm{d}s
\\ =& 
\sum_{i=0}^\infty
\int_{\ell_\alpha(t_i)}^{\ell_\alpha(t_{i+1})}\ell_{\beta_i}(x_s)\ \left(\frac{\mathrm{d}\ell_{\alpha}}{\mathrm{d}s}\right)^{-1} \mathrm{d}\ell_\alpha
= \sum_{i=0}^\infty
\int_{\ell_\alpha(t_{i+1})}^{\ell_\alpha(t_{i})}\ell_{\beta_i}(x_s)\ \left|\left(\frac{\mathrm{d}\ell_{\alpha}}{\mathrm{d}s}\right)^{-1}\right| \mathrm{d}\ell_\alpha
\\
\leq&
\sum_{i=0}^\infty
\int_{\ell_\alpha(x_{t_{i+1}})}^{\ell_{\alpha}(x_{t_{i}})} \frac{2i(\alpha,\beta_i)}{i(\alpha,\beta_i)|\cos\vartheta|}\left(\log\frac{1}{\ell_{\alpha}(x_{t_{s})}}+\log\frac{1}{\sin\vartheta}+C+\ell_\alpha(x_0)\right)\ \mathrm{d}\ell_\alpha
\\=&
\sum_{i=0}^\infty
\int_{\ell_\alpha(x_{t_{i+1}})}^{\ell_{\alpha}(x_{t_{i}})}\frac{2}{|\cos\vartheta|}\left(\log\frac{1}{\ell_{\alpha}}+\log\frac{1}{\sin\vartheta}+C+\ell_\alpha(x_0)\right)\ \mathrm{d}\ell_\alpha
\\ \leq&
\int_{0}^{\ell_{\alpha}(x_0)}\frac{2}{|\cos\vartheta|}\left(\log\frac{1}{\ell_{\alpha}}+\log\frac{1}{\sin\vartheta}+C+\ell_\alpha(x_0)\right)\ \mathrm{d}\ell_\alpha
\\
\leq&
\frac{2}{|\cos\vartheta|}
\ell_\alpha(x_0)\log\frac{1}{\ell_\alpha(x_0)}
+
\frac{2}{|\cos\vartheta|}\left(\log\frac{1}{\sin\vartheta}+1+C+\ell_\alpha(x_0)\right)\ell_\alpha(x_0).
\end{align*}

By choosing $\vartheta$ close to $\pi$ and $\delta>\ell_\alpha(x_0)$ sufficiently small ($\log\frac{1}{\delta}\gg\log\frac{1}{\sin\vartheta}$), we can bring the length of the path $x_t$ to the form of the desired expression, as per \cref{eq:distanceto}. 

All that remains is to show that the path $x_t$ tends to the boundary locus of Teichm\"uller space. To this end, showing $\ell_\alpha(x_t)\to0$ suffices. Assume otherwise, that $\ell_\alpha(x_t)$ monotonically decreases to some constant $A>0$. Since only the geometry of $S_\alpha$ changes along the path $x_t$, the path itself lies within a $2$-dimensional subspace of $\T(S)$ parameterised by  the length and twist parameters $\ell_\alpha$ and $\tau_\alpha$. In particular, since $\ell_\alpha(x_t)$ is decreasing and is bounded below by $A$,  the path $x_t$ must converge in $\T(S)$ unless $\tau_\alpha(x_t)$ is  unbounded. The latter is impossible, as the length of $x_t$ would be then comparable to the infinite length path formed by joining arbitrarily many $\alpha^j$ left Dehn twists of some point in this region. Thus, the sequence $(x_{t_i})$ has an accumulation point $\bar{x}\in\T(S)$. We derive a contradiction as follows.
Recall that $\beta_j$ is the $j$-times iterated Dehn twist of $\beta$ with respect to $\alpha$ and that $\theta_j(x):=\theta_x(\alpha,\beta_j)$. By \cite[Proposition~{3.5}]{Ker},  we see that $\theta_i(x)<\theta_{i+1}(x)$ for every $x\in\T(S)$. By the continuity of the angle function,
\[
\text{for all } j\in\mathbb{N},\quad
\theta_j(\bar{x})
=
\lim_{i\to\infty} \theta_j(x_{t_{i+1}})
\leq
\lim_{i\to\infty}\theta_{i}(x_{t_{i+1}})
=
\vartheta,
\]
where the last equality is due to \cref{eq:exactvartheta}. However, $(\beta_j)_{j\in\mathbb{N}}$ is a sequence of simple closed curves which is obtained by performing the left Dehn twist (i.e.\  right earthquake) arbitrarily many times with respect to $\alpha$, and hence $\lim_{j\to \infty}\theta_j(\bar x) = \pi$, and hence $\vartheta=\pi$. This contradicts the initial choice of $\vartheta\in(\frac{\pi}{2},\pi)$, and thus $\ell_\alpha$ tends to $0$ along $x_t$.
\end{proof}

\begin{remark}
It is possible to obtain the above result more directly by concretely specifying how long one should twist along a given curve $\beta_i$ before switching to $\beta_{i+1}$. 
\end{remark}

\subsection{Proof of \cref{thm:dist:boundary:wp}}\label{sec:proof:asym:dist:wp}

The upper bound in the statement of  \cref{thm:dist:boundary:wp} follows from the first inequality in \cref{prop:asymp:dist:upper}. For the lower bound, 
  let $\sigma:[0,1)\to \T(S)$ be an arbitrary piecewise $C^1$-path connecting $x$ to the boundary $\partial \overline{\T(S)}^{\mathrm{WP}}$. We subdivide $[0,1)$ into segments $[t_i,t_{i+1}]$ with $0=t_0<t_1<\cdots <t_n=0$ such that within each segment, the surfaces $\sigma(t)$, $t\in[t_i,t_{i+1}]$, share a common systole, which we refer to as $\alpha_i$. By \cref{lem:lowerbound:emetric}, the length of $\sigma$ with respect to the earthquake path is at least
\begin{align*}
&2\sum_{i} \int^{\max\{\ell_{\alpha_i}(\sigma(t_{i})),\ell_{\alpha_i}(\sigma(t_{i}))\}}_{\min \{\ell_{\alpha_i}(\sigma(t_{i})),\ell_{\alpha_i}(\sigma(t_{i}))\}}w(\ell)\ d\ell\\
=&2\sum_{i} \int^{\max\{\ell_{\mathrm{sys}}(\sigma(t_{i})),\ell_{\mathrm{sys}}(\sigma(t_{i}))\}}_{\min \{\ell_{\mathrm{sys}}(\sigma(t_{i})),\ell_{\mathrm{sys}}(\sigma(t_{i}))\}}w(\ell)\ d\ell \\
\geq& 2\int_0^{\ell_{\mathrm{sys}}(x)}\log \tfrac{1}{\ell}\ d\ell\\
=&2\ell_{\mathrm{sys}}(x)\log\frac{1}{\ell_{\mathrm{sys}}(x)}+2\ell_{\mathrm{sys}}(x)\\
\geq&2 \ell_{\mathrm{sys}}(x)\log\frac{1}{\ell_{\mathrm{sys}}(x)}.
\end{align*}
Finally, by \cref{prop:asymp:dist:upper}, we see that
\begin{eqnarray*}
\lim_{\ell_{\mathrm{sys}}(x)\to0}
\frac{d_e(x,\partial\overline{\T(S)})}{2\ell_{\mathrm{sys}}(x)\log\frac{1}{\ell_{\mathrm{sys}}(x)}}\leq  1.
\end{eqnarray*}
When combined with the lower bound established above, this yields the desired asymptotic distance, and this completes the proof.\qed

%\begin{remark}
%  Notice that for $m,y\in\T (S)$,  the left earthquake distance from $m$ to $y$ is exactly the right earthquake distance from $y$ to $m$.  Considering the right  earthquake instead of the left earthquake (assume that $d_e$ represents the left earthquake distance), we see that $d_e(\partial\overline{\T(S)}^E,m)$ is also comparable to $ \ell_{\mathrm{sys}}(x)\log\frac{1}{\ell_{\mathrm{sys}}(x)}.$
% \end{remark}

%=============
%=============
\newpage
\section{Earthquake metric completion}
\label{s:completion}

The incompleteness of the earthquake metric naturally leads one to draw conceptual comparisons with the Weil--Petersson metric. 
We recall that the completion of the Weil--Petersson metric of Teichm\"uller space has a rich associated theory:
 Teichm\"uller space completes to the augmented Teichm\"uller space \cite{abikoff1977degenerating}, and correspondingly,
the moduli space completes to its Deligne--Mumford compactification \cite{harvey1974chabauty,masur1976extension}.
What happens when one ``completes'' the earthquake metric? To rigorously answer this, we first need to explain what it means to complete a nonproper asymmetric metric. We give two approaches:
\begin{enumerate}
\item
we define a wholly novel notion of the completion of asymmetric metric spaces based on sequences which have finite ``\emph{distance-series}'';

\item
we take the completion with respect to a symmetrisation of the earthquake metric.
\end{enumerate}

The first approach is more intrinsic to the metric, but requires the development of a sufficiently geometrically motivated theory of completions for asymmetric metric spaces. The second approach is readily accessible using tools and language at our disposal, but is at the cost of the completion so-produced being more inherently tied to the symmetrised metric rather than the earthquake metric itself, and a priori it may suffer from the non-uniqueness of the choice of symmetrisation. We show that both methods yield the \emph{same} earthquake metric completion of Teichm\"uller space, and indeed that our two approaches produce completions which (topologically) coincide with the Weil-Petersson completion (\cref{thm:universalcompletion}).

\subsection{FD-completions of asymmetric metric spaces}
\label{sec:fdcompletion}

The standard theory of metric completion is founded on Cauchy sequences.
We observe that any Cauchy sequence $(x_n)$ in a (symmetric) metric space $(X,d)$ necessarily contains a subsequence $(x_{n_i})$ such that the series $\sum d(x_{n_i},x_{n_{i+1}})$, which we call a distance-series, is finite. We develop a novel theory of metric completion based on equivalence classes of sequences with finite distance-series, instead of  Cauchy sequences.

\begin{definition}[distance-series]
Given a sequence $(x_n)$ in an asymmetric metric space $(X,d)$, we call
\[
\sum_{i=1}^\infty
d(x_i,x_{i+1})
\]
the \emph{forward distance-series} for $(x_n)$. We likewise refer to $\sum_{i=1}^\infty d(x_{i+1},x_i)$ as the backward distance-series for $(x_n)$.
\end{definition}

\begin{definition}[FD-sequence] 
A sequence in $(X,d)$ is a \emph{forward FD-sequence} if its forward distance-series is finite. Define backward FD-sequences analogously.
\end{definition}

\begin{remark}
Unless otherwise specified, any FD-sequence mentioned henceforth without a specifying adjective should be taken as a forward FD-sequence.
\end{remark}

\begin{definition}[FD-completeness for asymmetric metrics]
\label{defn:FDcomplete}
We say that an asymmetric metric space $(X,d)$ is \emph{forward FD-complete} if for every forward FD-sequence $(x_n)$, there exists some $x\in X$ such that
\[
\lim_{n\to\infty} d(x_n,x)=0.
\]

\end{definition}

We next specify an equivalence relation on FD-sequences.

 \begin{definition}[interlacing]
Given two sequences $(x_i)$ and $(x'_j)$, we call a sequence $(\hat x_k)$ an \emph{interlacing}  of $(x_i)$ and $(x'_j)$ if
\begin{itemize}
%\item
%each $\hat x_k$ is taken from $(x_i)$ or $(x'_j)$;
\item
the subsequence of $(\hat x_k)$ comprised of elements from $(x_i)$ is infinite and has strictly increasing indices (as elements of $(x_i)$); and 
\item
the subsequence of $(\hat x_k)$ comprised of elements from $(x'_j)$ is also infinite and  has strictly increasing indices (as elements of $(x'_j)$).
\end{itemize}
\end{definition}

\begin{remark}
We can equivalently define interlacings to only admit subsequences which alternate between elements in $(x_i)$ and in $(x_j')$. This might be a simpler definition, but at the cost of occasional increase in the technicality of proofs (see, e.g., the proof of \cref{lem:infimalrepsequence}).
\end{remark}

\begin{definition}[FD-equivalence]
Two forward FD-sequences $(x_i)$ and $(x'_j)$ are \emph{forward FD-equivalent} if $(x_i)$ and $(x'_j)$ have an interlacing $(\hat x_k)$ which is a forward FD-sequence.  We similarly define backward FD-equivalence.
\end{definition}

\begin{remark}\label{rmk:FD:subsequence}
As a small but oft-used observation, any FD-sequence $(x_n)$ is FD-equivalent to all of its subsequences via the interlacing $(\hat{x}_n)=(x_n)$.
\end{remark}

\begin{lemma}\label{lem:equivrel}
The forward FD-equivalence is an equivalence relation on the set of forward FD-sequences. Likewise, the backward FD-equivalence is an equivalence relation on the set of backward FD-sequences.
\end{lemma}

\begin{proof}
Reflexivity and symmetry are clear. For transitivity, let $(x_i)$, $(x'_i)$, and $(x''_i)$ be arbitrary forward FD-sequences such that 
  $(x_i)$ and $(x'_i)$ are forward FD-equivalent, and
$(x'_i)$ and $(x''_i)$ are forward FD-equivalent.
Therefore, there exist subindices ${i_k}$ and ${j_k}$ satisfying
 \begin{align*}
 		i_1<i_3<i_5<\cdots;~ i_2<i_4<i_6<\cdots;~ j_1<j_3<j_5<\cdots;~ j_2<j_4<j_6<\cdots.
 \end{align*}
 	 such that both 
\begin{align*}
x_{i_1}, x_{i_2}',x_{i_3},x_{i_4}',\cdots
\quad\text{ and }\quad
x'_{j_1}, x''_{j_2},x'_{j_3},x_{j_4}'',\cdots
\end{align*}
 	are forward FD-sequences. We construct a forward  FD-interlacing of $(x_i)$ and $(x''_i)$ as follows.
	Start with $\hat x_1=x_{i_1}$. Take an odd index $k_1$ such that $j_{k_1}>i_2$ and set $\hat x_2:=x''_{j_{k_1+1}}$. Then 
 	\begin{equation*}
 		d(\hat x_1,\hat x_2)
 		\leq d(x_{i_1},x'_{i_2})+d(x'_{i_2},x'_{j_{k_1}})+d(x'_{j_{k_1}},x''_{j_{k_1+1}}).
 	\end{equation*}
 	Next, let $k_2$ be an even index such that $i_{k_2}>j_{k_1+2}$, and set $\hat x_3:=x_{i_{k_2+1}}$.
 	Then  
 	\begin{equation*}
 		d(\hat x_2,\hat x_3)
 		\leq d_e(x''_{j_{k_1+1}},x'_{j_{k_1+2}})+d(x'_{j_{k_1+2}},x'_{i_{k_2}})+d(x'_{i_{k_2}},x_{i_{k_2+1}}).
 	\end{equation*}
 	Inductively, we construct an interlacing  $\hat x_i$ of $(x_i)$ and $(x''_i)$ such that
 	\begin{align*}
 		\sum_{i\geq1} d(\hat x_i,\hat x_{i+1})
 		\leq 
		%\sum_{i\geq1} d( x_i, x_{i+1})+
		&\sum_{i\geq1} d( x'_i, x'_{i+1})
		%+\sum_{i\geq1} d( x''_i, x''_{i+1})\\ 
 		+\sum_{k\geq 1} [d( x_{i_{2k-1}}, x'_{i_{2k}})+d(x'_{i_{2k}},x_{i_{2k+1}})]\\
 		& +\sum_{k\geq 1} [d( x'_{j_{2k-1}}, x''_{j_{2k}})+d(x''_{j_{2k}},x'_{j_{2k+1}})] <\infty,
 	\end{align*} 
	which means that it is forward FD.
 	Hence $(x_i)$ and $(x''_i)$ are forward FD-equivalent. This shows that forward FD-equivalence is an equivalence relation. The proof for backward FD-equivalence being an equivalence relation is essentially the same.
 \end{proof}

\begin{definition}[FD-completion of asymmetric metric spaces]
The \emph{forward FD-completion} of an asymmetric metric space $(X,d)$ is defined as the set of forward FD-equivalence classes of forward FD-sequences. We denote this set by $\overline{X}$, and write each FD-equivalence class of an FD-sequence $(x_n)$ in the form $[x_n]$. We define the backward FD-completion analogously, and denote it by $\overline{X}^\#$.
\end{definition}

\begin{remark}
At present, $\overline{X}$ and $\overline{X}^{\#}$ are merely sets. We (asymmetrically) metrise them in Appendix~B, \cref{thm:FDcomplete}.
\end{remark}

\begin{proposition}[natural inclusion]
\label{lem:natinclude}
The map $\iota: X\to\overline{X}$ sending  $x\in X$ to the FD-equivalence class represented by the constant FD-sequence $(x)_{n\in\mathbb{N}}$ is an injection. In other words, $\overline{X}$ contains a copy of $X$ as the set of FD-equivalence classes represented by constant FD-sequences. This is also true for $\overline{X}^\#$, and we denote the inclusion by $\iota^\#$.
\end{proposition}

\begin{proof}
The map $\iota$ is well defined because constant sequences are FD-sequences. For the injectivity, consider two constant FD-sequences $(x)$ and $(y)$ which are FD-equivalent. The existence of a interlacing means that
\[
0=\lim_{n\to\infty}d(x,y)=d(x,y)
\quad\text{and}\quad
0=\lim_{n\to\infty}d(y,x)=d(y,x).
\]
By the axioms of asymmetric metric spaces, we see that $x=y$, which is the required injectivity. 
\end{proof}

We use \cref{lem:natinclude} to regard $X$ as a subset of $\overline{X}$, and this allows us to speak of extensions of functions and metrics.

\begin{definition}[extensions of functions and metrics]
Consider a function $f:(X,d_X)\to (Y,d_Y)$, and let $\iota_X:X\to\overline{X}$ and $\iota_Y:Y\to\overline{Y}$ denote the respective natural inclusion maps for $X$ and $Y$. We say that $\phi$ is an \emph{extension} of $f$ if $\iota_Y\circ f = \phi\circ\iota_X$, that is, if it satisfies the following commutative diagram:
\begin{center}$\begin{CD}
X 			@>f>> 	Y\\
@VV\iota_X V			@VV\iota_Y V\\
\overline{X}	@>\phi>>	\overline{Y}
\end{CD}$
\end{center}
Correspondingly, an asymmetric metric $D$ on $\overline{X}$ or $\overline{X}^{\#}$ is a \emph{metric extension} of $d_X$ if $D\circ (\iota_X\times\iota_X)=d_X$. 

\end{definition}

Having only looked at the FD-completion from a purely set-theoretic point of view, there are indubitably many questions as to why this construction is ``natural'' and if it merits being dubbed a ``completion''. In response, we show in \cref{appendix:FD} that the FD-completions $\overline{X}$ and $\overline{X}^\#$ of an asymmetric metric space $(X,d)$ satisfy the following:

\begin{enumerate}
\item
$\overline{X}$ is naturally metrised via an asymmetric metric $\bar{d}$ which extends the asymmetric metric $d$ (\cref{thm:asymmetricmetric} and \cref{prop:natisoinclude}). Similarly, $\overline{X}^\#$ is naturally metrised via an asymmetric metric $\bar{d}^\#$ which extends the reverse metric $\bar{d}^\#$ on $X$ (\cref{rmk:backwardmetric});

\item
the natural inclusion map $\iota:(X,d)\to(\overline{X},\bar{d})$ is an isometric embedding with forward-dense image (\cref{prop:natisoinclude});

\item
forward FD-completions are forward FD-complete (\cref{thm:FDcomplete});

\item
FD-completions generalise the notion of Cauchy completion when $d$ is symmetric (\cref{thm:Fd=cauchy});

\item
the process of taking the forward FD-completion of an asymmetric metric space can be naturally promoted to an endofunctor on the category of asymmetric metric spaces with Lipschitz morphisms (see \cref{defn:aml}).
\end{enumerate}

\subsection{FD-sequences in the earthquake metric}

Having defined FD-completions, we move on to establish the properties of the forward and backward FD-completions of $(\T(S),d_e)$. We employ the following notation throughout:
\begin{itemize}
\item
$\overline{\T(S)}^\mathrm{WP}$ denotes the Weil--Petersson completion of $\T(S)$;

\item 
$\overline{\T(S)}$ denotes the forward FD-completion of $\T(S)$ with respect to the left earthquake metric;

\item
$\overline{\T(S)}^{\#}$ denotes the backward FD-completion of $\T(S)$ with respect to the left earthquake metric.

\end{itemize}

\begin{remark}
Thanks to \cref{rmk:rightearthquake}, $\overline{\T(S)}^{\#}$ (regarded purely as a set) is identical to the forward FD-completion of $\T(S)$ with respect to the right earthquake metric. In fact, (see \cref{rmk:backwardmetric}) this identity applies also at the level of the metric extension: the backward metric extension of the left earthquake metric is definitionally equal to the forward metric extension of the right earthquake metric.
\end{remark}

Many of the arguments in the remainder of this section depend on understanding the limiting behaviour of FD-sequences in $(\T(S),d_e)$. Crucially, the length spectrum for an FD-sequence with respect to $d_e$ stabilises, and this can be characterised in terms of convergence in the Weil--Petersson metric.

\begin{lemma}
\label{lem:fdiswp}
Any (forward or backward) FD-sequence $(x_n)$ with respect to the earthquake metric satisfies the following:
\begin{enumerate}[(i)]
\item
for every simple closed curve $\alpha$, $(\ell_\alpha(x_n))_{n\in\mathbb{N}}$ converges in $[0,\infty]$; 
\item
$(x_n)$ converges in $\overline{\T(S)}^{\mathrm{WP}}$, and hence is a Cauchy sequence with respect to the Weil--Petersson metric.
\end{enumerate}
\end{lemma}

\begin{proof}
We first note that it suffices to only deal with forward FD-sequences.

(i) We assume that $(x_n)$ is a forward FD-sequence, and verify the convergence of $\ell_\alpha(x_n)$. 
Consider two subsequences
\begin{itemize}
\item
$(x_{n_i})$ where $\ell_{\alpha}(x_{n_i})\to \delta=\liminf_{n}\ell_{\alpha}(x_n)\in[0,\infty]$, and
\item
$(x_{n_j'})$ where $\ell_{\alpha}(x_{n_j'})\to \delta'=\limsup_{n}\ell_{\alpha}(x_n)\in[0,\infty]$.
\end{itemize}
If the sequence $(\ell_{\alpha}(x_n))$ diverges in $[0,\infty]$, then $\delta$ is strictly smaller than $\delta'$,  hence is finite. Let $[a,b]$ be any nontrivial closed interval contained in $(\delta,\delta')$. By \cref{lem:lowerbound:emetric}, there exists $M$ such that for any $i,j>M$,
\begin{align}\label{eq:fd:sequence}
d_e({x_{n_i}},x_{n_j'})\geq 2\int^{b}_{a} w(\ell)\ \mathrm{d}\ell>0,
\end{align}
	 	
On the other hand, the assumption that $(x_n)$ is an FD-sequence implies that for any $m>n$,
\begin{align*}
d_e(x_n,x_m)
\leq
\sum_{k=n}^\infty d_e(x_k,x_{k+1})\to 0,\quad\text{as }n\to\infty.
\end{align*}
This contradicts \eqref{eq:fd:sequence}, and hence $\ell_\alpha(x_n)$ converges for every simple closed curve $\alpha$.

(ii) We now show that $(x_n)$ is a Cauchy sequence with respect to the Weil--Petersson metric. Let $B$ be the Bers constant, which depends only on the topological type of $S$, and let $N_1$ be a natural number such that  
\begin{align}\label{eq:N:s'}
\sum_{k\geq  N_1} d_e(x_k,x_{k+1})<\int_{ B}^{2 B}w(\ell)\;\mathrm{d}\ell.
\end{align}
Consider the surface $x_{N_1}$, and let $\Gamma$ be a pants decomposition of $x_{N_1}$ comprised of geodesics of length at most $B$.  For any $\gamma\in \Gamma$, we can see that 
\begin{align}\label{eq:N:2B'}
\ell_\gamma(x_n)< 2 B, \quad \text{for all }n>N_1.
\end{align}
Otherwise, by \cref{lem:lowerbound:emetric}, 
\begin{align*}
d_e(x_{N_1}, x_n)\geq 2 \int_{\ell_\gamma(x_{N_1})}^{\ell_\gamma(x_{n})}w(\ell)\;\mathrm{d}\ell \geq 2 \int_{B}^{2 B}w(\ell)\;\mathrm{d}\ell,
\end{align*}  
which contradicts \cref{eq:N:s'}. Combining \eqref{eq:N:2B'} with part (i) of this lemma, we see that $\lim\limits_{n\to\infty}\ell_\gamma(x_n)\in [0, 2 B)$ exists for every $\gamma\in \Gamma$. In particular, $\Gamma$ must contain all the \emph{pinching curves} of $(x_n)$, whose collection we denote by $P$.
If $\Gamma=P$, the sequence $(x_n)$ converges to a  single point in the augmented Teichm\"{u}ller space with respect to the Weil--Petersson metric and we are done.

Suppose otherwise that $\Gamma$ properly contains $P$.
Define $\mathbf{s}:=\inf\{\ell_\gamma(x_i)\mid \gamma\in\Gamma\setminus P,\  i\in\mathbb{N}\}$.  
We see that $\mathbf{s}>0$ because there are only finitely many curves in $\Gamma\setminus P$ and none of them are pinched to length $0$ as $i\to\infty$. 
Any two shortest curves on $x_i$ which minimally intersect $\gamma$ without intersecting any other curve in $\Gamma$ are necessarily related by a Dehn twist or a half-Dehn twist with respect to $\gamma$. An upper bound for the lengths of such pairs of curves for all $\gamma\in \Gamma\setminus P$ on all $x_i$, $i\in\mathbb{N}$, exists since $\mathbf{s}>0$ and \cref{eq:N:2B'}.
%Hence, for every $\gamma\in \Gamma$, there is an upper bound in terms of $\mathrm{s}$ and $B$ on the length of simple orthogeodesic segments on each component of $x_i\setminus P$ which start and end on $\gamma$. This in turn asserts an upper bound $L=L(\mathbf{s},B)$  on the length of two shortest curves on $x_i$ which (minimally) intersect $\gamma$ without intersecting any other curve in $\Gamma$ (see, e.g.: \cite[Theorem~3.12]{IT}). The homotopy classes and the lengths of any two such curves determine the twist parameter for $\gamma$, and we show that both the homotopy classes and the lengths of such curves stabilise for $(x_i)$ as $i\to\infty$. 
Let $N_2$ be a constant such that
\begin{align*}
\sum_{k\geq  N_2} d_e(x_k,x_{k+1}) < \int_{L}^{2L}w(\ell)\;\mathrm{d}\ell. 
\end{align*}
Given an arbitrary $\gamma\in \Gamma\setminus P$, let $\alpha$ and $\beta$ denote two shortest curves on $x_{N_2}$ which minimally intersect $\gamma$ and are disjoint from other curves in $\Gamma$. We know that $\ell_\alpha(x_{N_2}),\ell_\beta(x_{N_2})\leq L$ by our definition of $L$. Applying the same argument as for the proof of \cref{eq:N:2B'}, we see that
\begin{align*}
\ell_{\alpha}(x_i),\ell_{\beta}(x_i)\leq  2L, \quad\text{for all }i> N_2.
\end{align*}
Combined with part (i) of this lemma, the lengths of $\alpha$ and $\beta$ converge in $(0,2L)$.
This also shows that Fenchel--Nielsen twist parameter for $\gamma$ also converges as can be seen by considering the Fricke--Klein embedding of Teichm\"{u}ller space (see, e.g., \cite[Theorem~3.12]{IT}). 
Since this holds for all $\gamma\in\Gamma\setminus P$, the criterion of convergence in the Weil--Petersson completion (see, e.g., the characterisation immediately prior to Remark~5.1 in \cite{mondello_convergence}) tells us that $(x_n)$ is a convergent sequence in $\overline{\T(S)}^{\mathrm{WP}}$, and hence is Cauchy.
\end{proof}

\begin{corollary}
\label{cauchy eq}
Let $(x_n)$ and $(x_n')$ be two FD-equivalent FD-sequences in $(\T(S), d_e)$.
Then they are Cauchy equivalent with respect to the Weil--Petersson metric, and moreover we have $\underset{n\to\infty}{\lim}d_e(x_n, x'_n)\to 0$.
\end{corollary}
\begin{proof}
By \cref{lem:fdiswp}, both $(x_n)$ and $(x_n')$ converge in $\overline{\T(S)}^{\mathrm{WP}}$, whose limits we denote by $x$ and $x'$ respectively.
Since $(x_n)$ and $(x_n')$ are FD-equivalent, there is an FD-interlacing $(y_n)$ between them.
Since $(y_n)$ itself is also an FD-sequence, it converges to a point in $y$ in $\overline{\T(S)}^{\mathrm{WP}}$.
Since $(y_n)$ contains infinitely many terms in both $(x_n)$ and $(x_n')$, we have $x=y=x'$.
Therefore, $(x_n)$ and $(x_n')$ are Cauchy equivalent, and $\underset{n\to \infty}{\lim}d_\mathrm{WP}(x_n, x_n')=0$.
By \cref{cor:e<wp}, this implies also $\underset{n\to \infty}{\lim}d_e(x_n, x_n')=0$
\end{proof}

\begin{lemma}
\label{lem:fd-equi:wp-equi}
Any two Weil--Petersson metric Cauchy sequences $(x_n)$ and $(y_n)$ satisfying $\underset{n\to\infty}{\lim}d_e(x_n,y_n)=0$ must be Cauchy equivalent with respect to the Weil--Petersson metric.
\end{lemma}

\begin{proof}
Since $(x_n)$ and $(y_n)$ are Cauchy sequences with respect to the Weil-Petersson metric, they respectively converge to points $x,y\in\overline{\T(S)}^{\mathrm{WP}}$. Assume $x\neq y$. Then, there must be some simple closed curve $\gamma$ on $S$ such that $|\ell_{\gamma}(x)-\ell_{\gamma}(y)|>\delta$ for some $\delta\in(0,\infty]$, where $\delta=\infty$ means that exactly one of $\ell_\gamma(x)$ and $\ell_\gamma(y)$ is infinity. 
Hence, by Lemma \ref{lem:lowerbound:emetric}, 
\[
\limsup_{n\to\infty}
d_e(x_n,y_n)
\geq
\int_{\min\{\ell_{\gamma}(x),\ell_{\gamma}(y)\}}^{\min\{\ell_{\gamma}(x),\ell_{\gamma}(y)\}+\delta} w(\ell)\ \mathrm{d}\ell>0,
\]
which contradicts the assumption that $\lim\limits_{n\to\infty} d_e(x_n,y_n)=0$. Thus, the limits $x$ and $y$ must be equal, therefore $(x_n)$ and $(y_n)$ must be Cauchy equivalent.
\end{proof}

%\begin{lemma}\label{lem:lr:equiv}
%Any forward/backward FD-sequence $(x_n)$ in $(\T(S),d_e)$ contains a subsequence which is both a forward and backward FD-sequence.
%\end{lemma}

%\begin{proof}
%Suppose that $(x_n)$ is a forward FD-sequence, then by \cref{lem:fdiswp}, it is Cauchy with respect to the Weil--Petersson metric.  Therefore, there is a subsequence $(x_{n_k})_{k\in\mathbb{N}}$ which is an FD-sequence in $(\T(S),d_{\mathrm{WP}})$. By \cref{thm:e:wp}, it follows
%\begin{align*}
%\sum_{n=1}^\infty 
%d_e(x_{n_{k+1}},x_{n_k})
%\leq
%\sum_{n=1}^\infty
%C_1d_{\mathrm{WP}}(x_{n_{k+1}},x_{n_k})
%<
%\infty.
%\end{align*}
%Thus $(x_{n_k})$ is a backward FD-sequence. It is also a forward FD-sequence since it is a subsequence of $(x_n)$.
%\end{proof}
%I commented this out, because we don't use it any more.

\subsection{The metric geometry of $\overline{\T(S)}$.}

The goal of this subsection is to show that $\overline{\T(S)}$ inherits a unique continuous extension of $d_e$ which is also a Busemannian metric (see \cref{thm:!extension}). 

\begin{theorem}
\label{thm:metric:extension}
The function $\bar{d_e}:\overline{\T(S)}\times\overline{\T(S)}\to[0,\infty)$ given by
 \begin{align*}
\bar{d_e}([x_n],[y_n])
:=
\lim_{n\to\infty} d_e(x_n,y_n),
\end{align*}
where $(x_n)$ and $(y_n)$ are arbitrary FD-sequences respectively representing $[x_n]$ and $[y_n]$, is well defined and defines an asymmetric metric on $\overline{\T(S)}$.
\end{theorem}

\begin{proof}
We first show that $\bar{d_e}$ is well defined. Let $(x_n), (x'_n)$ be FD-sequences representing $[x_n]$ and let $(y_n), (y'_n)$ be FD-sequences representing $[y_n]$. The triangle inequality tells us that
\[
d_e(x_n,y_n)
\leq d_e(x_n,x'_n)+d_e(x'_n,y'_n)+d_e(y'_n,y_n).
\]
\cref{cauchy eq} tells us that
$
d_e(x_n,x'_n),\ d_e(y'_n,y_n)\to 0,
$
and  hence $\bar{d_e}([x_n],[y_n])$ is independent of the FD-sequences representatives chosen to define the limit. Moreover, the fact that $(x_n)$ and $(y_n)$ are Cauchy sequences in $d_{\mathrm{WP}}$ proved in \cref{lem:fd-equi:wp-equi} combined with \cref{cor:e<wp} means that $d_e(x_n,y_n)\leq C_1 d_{\mathrm{WP}}(x_n,y_n)$ remains bounded, hence $\bar{d_e}:\overline{\T(S)}\times\overline{\T(S)}\to[0,\infty)$ is well defined. Note that the expression given for $\bar{d_e}$ agrees with the expression for the standard metric extension of $d_e$ to its completion, as given in \cref{thm:asymmetricmetric}, hence $\bar{d_e}$ is an asymmetric metric.
\end{proof}

\begin{theorem}
\label{thm:busemannianext}
The  function $\bar{d_e}$ defined in \cref{thm:metric:extension} is a Busemannian metric on the forward FD-completion $\overline{\T(S)}$.
\end{theorem}

\begin{proof}
We first note that the triangle inequality holds, and $\bar{d_e}(\xi,\xi)=0$ is clear. Thus, we need only verify that
\begin{enumerate}
\item for any $\xi,\eta\in \overline{\T(S)}$, $\bar{d_e}(\xi,\eta)=0\Rightarrow \xi=\eta$; and
\item for any $\eta\in \overline{\T(S)}$ and any $(\xi_k)$ in $\overline{\T(S)}$, we have $\bar{d_e}(\xi_k,\eta)\to 0 \Leftrightarrow\bar{d_e}(\eta,\xi_k)\to 0$.
\end{enumerate}
For {(1)}, suppose that $\bar{d_e}(\xi,\eta)=\underset{n\to\infty}{\lim} d_e(x_n,y_n)=0$. Let $(x_n)$ and $(y_n)$ be FD-sequences respectively representing $\xi$ and $\eta$. Then \cref{lem:fd-equi:wp-equi} asserts that $(x_n)$ and $(y_n)$ are equivalent Cauchy sequences with respect to $d_{\mathrm{WP}}$. Since $d_{\mathrm{WP}}$ is symmetric, this means that
\begin{align*}
0=\lim_{n\to\infty} C_1d_{\mathrm{WP}}(y_n,x_n)\geq \lim_{n\to\infty} d_e(y_n,x_n)=\bar{d_e}(\eta,\xi).
\end{align*}
Since we showed above that $\bar{d_e}$ is an asymmetric metric, the fact that $\bar{d_e}(\xi,\eta)=\bar{d_e}(\eta,\xi)=0$ implies $\xi=\eta$.

For {(2)}, we shall only show that $\bar{d_e}(\xi_k,\eta)\to 0$ implies $\bar{d_e}(\eta,\xi_k)\to 0$; the proof of the converse is essentially identical. For each $k$, let $(x_{n,k})_{n\in\mathbb{N}}$ denote an arbitrary FD-sequence representative for the FD-equivalence class $\xi_k$. Likewise let $(y_n)$ be an FD-sequence with respect to $d_e$ which represents the FD-equivalence class $\eta\in\overline{\T(S)}$. By \cref{lem:fdiswp}, for every $k$, the sequence $(x_{n,k})_{n\in\mathbb{N}}$ is a Cauchy sequence in $(\T(S),d_{\mathrm{WP}})$, as is $(y_n)$. We choose a sequence of increasing indices $(m_k)_{k\in\mathbb{N}}$ so that  for each $k\in\mathbb{N}$, \begin{equation}\label{eq:wp:mk:infty}
	\dwp (x_{m_k,k},x_{n,k})<2^{-k}, \text{ for every }n>m_k,
\end{equation}
and $$|
\bar{d_e}(\xi_k,\eta)
-
d_e(x_{m_k,k},y_{m_k})
|<2^{-k}.$$
Hence
\begin{align}
\lim_{k\to0}
d_e(x_{m_k,k},y_{m_k})
=0.\label{eq:limitsto0}
\end{align}
Since $(y_{m_k})_{k\in\mathbb{N}}$ is a  Cauchy sequence in $(\T(S),d_{\mathrm{WP}})$ by  \cref{lem:fd-equi:wp-equi}  and  \cref{eq:limitsto0}), we conclude that  $(x_{m_k,k})_{k\in\mathbb{N}}$ is also a Cauchy sequence in $(\T(S),d_{\mathrm{WP}})$ and that $(x_{m_k,k})_{k\in\mathbb{N}}$ and $(y_{m_k})_{k\in\mathbb{N}}$ are Cauchy equivalent in $(\T(S),d_{\mathrm{WP}})$.
Therefore,
\begin{align}
0\leq
\lim_{k\to\infty} d_e(y_{m_k},x_{m_k,k})
\leq 
\lim_{k\to\infty} C_1 d_{\mathrm{WP}}(y_{m_k},x_{m_k,k})=0,\label{eq:reverseto0}
\end{align}
where the inequality and the constant $C_1$ both come from \cref{cor:e<wp}.

Finally, we make use of \cref{eq:wp:mk:infty}, \cref{eq:reverseto0} and the Cauchy convergence of $(y_k)_{k\in\mathbb{N}}$ in $(\T(S),d_{\mathrm{WP}})$ to show:
\begin{align*}
&\lim_{k\to\infty}
\bar{d_e}(\eta,\xi_k)\\
=
&\lim_{k\to\infty}\lim_{n\to\infty}
d_e(y_n,x_{n,k})\\
\leq
&\lim_{k\to\infty}\lim_{n\to\infty}
\left(d_e(y_n,y_{m_k})+d_e(y_{m_k},x_{m_k,k})+d_e(x_{m_k,k},x_{n,k})
\right)\\
\leq
&\lim_{k\to\infty}\lim_{n\to\infty}
C_1\left(d_{\mathrm{WP}}(y_n,y_{m_k})+d_{\mathrm{WP}}(y_{m_k},x_{m_k,k})+d_{\mathrm{WP}}(x_{m_k,k},x_{n,k})
\right)
=0.
\end{align*}
We have thus verified that $\bar{d_e}$ satisfies all of the axioms for being a Busemannian metric on $(\overline{\T(S)},\bar{d_e})$.
 \end{proof}
 %=========

\begin{remark}
The analogous statements (and proofs) of \cref{thm:metric:extension} and \cref{thm:busemannianext} also hold for the backward FD-completion $(\overline{\T(S)}^{\#},\bar{d_e}^{\#})$. 
\end{remark}

\subsection{Homeomorphic extension maps}

One important consequence of $\bar{d_e}$ being Busemannian (\cref{thm:busemannianext}) is that the forward, backward and symmetric topologies on $(\overline{\T(S)},\bar{d_e})$ all agree, and we set it as the standard topology on $\overline{\T(S)}$. Likewise, $\bar{d}_e^\#$ is also Busemannian and defines a standard topology on $\overline{\T(S)^\#}$. This allows us to speak of continuous maps between these spaces.
 
\begin{theorem}
\label{thm:lr:wp:completion} 
The identity map of Teichm\"uller space $\T(S)$ admits a unique homeomorphic extension between any pair of the following three:
\begin{itemize}
\item
the forward FD-completion $(\overline{\T(S)},\bar{d_e})$, 
\item
the forward FD-completion $(\overline{\T(S)}^{\#},\bar{d}_e^\#)$,
\item
and the Weil--Petersson completion $(\overline{\T(S)}^\mathrm{WP},\bar{d}_{\mathrm{WP}})$.
\end{itemize}
\end{theorem}

\begin{proof}
We only show that $\mathrm{id}:(\T(S),d_e)\to(\T(S),\mathrm{d}_{\mathrm{WP}})$ has a unique homeomorphic extension. The proof for the backward FD-completion is identical, and the uniqueness of these two homeomorphisms then asserts the uniqueness of the requisite extension homeomorphism between $(\overline{\T(S)},\bar{d_e})$ and $(\overline{\T(S)}^\#,\bar{d_e}^\#)$.  The proof will be divided into the following steps: we
\begin{itemize}
 	\item define $\Pi:\overline{\T(S)}\to \overline{\T(S)}^\mathrm{WP}$;
 	\item show that the map $\Pi$ is surjective;
 	\item show that the map $\Pi$ is injective;
 	\item show that the map $\Pi$ is a homeomorphism; and
 	\item show that $\Pi$ is the unique homeomorphic extension.
 \end{itemize} 

\noindent\textsc{Step~1: constructing the map $\Pi:\overline{\T(S)}\to \overline{\T(S)}^\mathrm{WP}$.} By \cref{lem:fdiswp} any FD-sequence $(x_n)$ in $(\T(S),d_e)$ is a Cauchy sequence in $(\T(S),d_{\mathrm{WP}})$, and hence we define $\Pi$ by sending an FD-equivalence class represented by an FD-sequence $(x_n)$ with respect to $d_e$ to the Cauchy equivalence class represented by $(x_n)$ with respect to $d_\mathrm{WP}$. This is well defined thanks to the last sentence of \cref{lem:fd-equi:wp-equi}.\vspace{1em}

\noindent\textsc{Step~2: establishing surjectivity.} Consider an arbitrary element of $\overline{\T(S)}^{\mathrm{WP}}$ and let it be represented by some Cauchy sequence $(x_n)$ with respect to $d_\mathrm{WP}$. Then $(x_n)$ contains a subsequence $(x_{n_k})_{k\in\mathbb{N}}$ which is an FD-sequence with respect to $d_\mathrm{WP}$. 
By \cref{thm:e:wp}, there exists a constant $C_1$ such that
\begin{align*}
\sum_{k=1}^\infty
d_e(x_{n_k},x_{n_{k+1}})
\leq
\sum_{k=1}^\infty
C_1d_{\mathrm{WP}}(x_{n_k},x_{n_{k+1}})
<\infty.
\end{align*}
Therefore, $(x_{n_k})_{k\in\mathbb{N}}$ is also an FD-sequence with respect to $d_e$, and its FD-equivalence class maps to the same Cauchy equivalence class as $(x_n)_{n\in\mathbb{N}}$ since $(x_{n_k})_{k\in\mathbb{N}}$ is a subsequence of $(x_n)$.\vspace{1em}

\noindent\textsc{Step~3: establishing injectivity.} Consider two (forward) FD-sequences $(x_n)$ and $(y_n)$ in $(\T(S),d_e)$ which respectively represent two arbitrary points $[x_n],[y_n]\in\overline{\T(S)}$.  By \cref{lem:fdiswp}, these two sequences are Cauchy sequences with respect to $d_\mathrm{WP}$. Now suppose that they are Cauchy equivalent in $(\T(S),d_{\mathrm{WP}})$. Then, there are subsequences $(x_{n_k})_{k\in\mathbb{N}}$ and $(y_{n_k})_{k\in\mathbb{N}}$ which are FD-sequences with respect to $d_\mathrm{WP}$. 
The Cauchy equivalence ensures that $(x_{n_k})_{k\in\mathbb{N}}$ and $(y_{n_k})_{k\in\mathbb{N}}$ are FD-equivalent (\cref{lem:congruence}) with respect to $d_\mathrm{WP}$, and hence there is an interlacing with respect to $d_\mathrm{WP}$ between them. Thanks to \cref{cor:e<wp}, such an interlacing  is also an interlacing for $d_e$, and hence $(x_{n_k})_{k\in\mathbb{N}}$ and $(y_{n_k})_{k\in\mathbb{N}}$ are FD-equivalent with respect to $d_e$. Therefore, $(x_n)$, $(x_{n_k})_{k\in\mathbb{N}}$, $(y_{n_k})_{k\in\mathbb{N}}$ and $(y_n)$ are all FD-equivalent, and  $[x_n]=[y_n]$. This gives the requisite injectivity.\vspace{1em}

\noindent\textsc{Step~4: establishing homeomorphicity.} %\begin{lemma}\label{lem:lr:wp:homeo}
% 	 The map $\Pi:\overline{\T(S)}\to \overline{\T(S)}^\mathrm{WP}$ is a homeomorphism.
% \end{lemma}
%\begin{proof}
Steps 2 and 3 show that $\Pi$ is a bijection. By \cref{cor:e<wp}, its inverse map is Lipschitz   hence is also continuous (since both metrics are Busemannian). We only need to show that $\Pi$ is continuous. Fortunately, as both metrics are Busemannian, continuity is the same as the convergence of limits (\cref{rmk:continuityissues}). Specifically, we need to verify the following: given an arbitrary sequence $(\xi_k)$ in $(\overline{\T(S)},\bar{d_e})$ that converges to some point $\xi\in\overline{\T(S)}$, i.e.\  $\bar{d_e}(\xi_k,\xi)\to0$, then 
\[
\lim_{k\to\infty}\bar{d}_{\mathrm{WP}}(\Pi(\xi_k),\Pi(\xi))=0.
\]

For each $k$, let $(x_{n,k})_{n\in\mathbb{N}}$ be an arbitrary FD-sequence in $(\T(S),d_e)$ representing $\xi_k$. By \cref{lem:fdiswp}, it is also a Cauchy sequence with respect to $d_\mathrm{WP}$, and we note that $\Pi(\xi_k)$ is represented by $(\mathrm{id}(x_{n,k}))_{n\in\mathbb{N}}=(x_{n,k})_{n\in\mathbb{N}}$.
Let $(x_n)$ be an arbitrary FD-sequence in $(\T(S),d_e)$ representing $\xi$; it is likewise a Cauchy sequence with respect to $d_\mathrm{WP}$, and we note that $\Pi(\xi)$ is represented by $(x_n)$.

Since $(x_{n,k})_{n\in\mathbb{N}}$ is a Cauchy sequence for every $k$, we may choose an increasing sequence $(m_k)_{k\in\mathbb{N}}$ of indices so that
$
\text{for all }n\geq m_k,\quad
d_{\mathrm{WP}}(x_{m_k,k},x_{n,k})<2^{-k}.
$
We claim that $(x_{m_k,k})_{k\in\mathbb{N}}$ is an FD-sequence in $(\T(S),d_e)$. This follows from
\begin{align*}
&d_e(x_{m_k,k},x_{m_{k+1},k+1})
\leq
C_1 d_{\mathrm{WP}}(x_{m_k,k},x_{m_{k+1},k+1})\\
\leq&
\lim_{n\to\infty}
C_1\left(
d_{\mathrm{WP}}(x_{m_k,k},x_{n,k})+d_{\mathrm{WP}}(x_{n,k},x_{n,k+1})+d_{\mathrm{WP}}(x_{n,k+1},x_{m_{k+1},k+1})
\right)\\
\leq&
C_1\left(
2^{-k}+\bar{d}_{\mathrm{WP}}(\xi_k,\xi_{k+1})+2^{-k-1}
\right),
\end{align*}
in conjunction with $(\xi_k)$ being a Cauchy sequence with respect to $\bar d_\mathrm{WP}$. 

Next observe that
\begin{align*}
&\lim_{k\to\infty}
d_e(x_{m_k,k},x_k)\\
\leq
&\lim_{k\to\infty}\lim_{n\to\infty}
\left(d_e(x_{m_k,k},x_{n,k})+d_e(x_{n,k},x_n)+d_e(x_n,x_k)\right)\\
\leq
&\lim_{k\to\infty}\lim_{n\to\infty}
\left(
C_1d_{\mathrm{WP}}(x_{m_k,k},x_{n,k})+d_e(x_{n,k},x_n)+C_1 d_{\mathrm{WP}}(x_n,x_k)
\right)\\
=
&\lim_{k\to\infty}
(C_1 2^{-k}+\bar{d_e}(\xi_k,\xi))+
\lim_{k\to\infty}\lim_{n\to\infty}
C_1 d_{\mathrm{WP}}(x_n,x_k)=0.
\end{align*}
Since the FD-sequences $(x_{m_k,k})_{k\in\mathbb{N}}$ and $(x_k)_{k\in\mathbb{N}}$ satisfy $d_e(x_{m_k,k},x_k)\to0$, \cref{lem:fdiswp,lem:fd-equi:wp-equi} tell us that these two sequences are also  equivalent Cauchy sequences in the Weil--Petersson metric. Hence,
\begin{align*}
&\lim_{k\to\infty}\bar{d}_{\mathrm{WP}}(\Pi(\xi_k),\Pi(\xi))
=
\lim_{k\to\infty}\lim_{n\to\infty}
d_{\mathrm{WP}}(x_{n,k},x_n)\\
\leq
&\lim_{k\to\infty}\lim_{n\to\infty}
\left(d_{\mathrm{WP}}(x_{n,k},x_{m_k,k})
+d_{\mathrm{WP}}(x_{m_k,k},x_k)
+d_{\mathrm{WP}}(x_k,x_n)
\right)
=0.
\end{align*}
Thus, we have shown that $\Pi$ is continuous, hence is a homeomorphism.\vspace{1em}

\noindent\textsc{Step~5: establishing uniqueness.} Consider a homeomorphic extension $\Pi'$ of $\mathrm{id}:(\T(S),d_e)\to(\T(S),d_{\mathrm{WP}})$. Then $\Pi'\circ\Pi^{-1}$ gives a homeomorphic extension of $\mathrm{id}:(\T(S),d_{\mathrm{WP}})\to(\T(S),d_{\mathrm{WP}})$, which we know, from the uniqueness of continuous extensions of Lipschitz maps to maps between Cauchy completions, must be the identity map on $(\overline{\T(S)}^{\mathrm{WP}},\bar{d}_{\mathrm{WP}})$. Therefore, $\Pi'$ is $\Pi$.
\end{proof}

\subsection{Completion via symmetrisation}

We end this section by explaining what happens when we complete a symmetrisation of the earthquake metric. We first note that there is no canonical choice of symmetrisation for an asymmetric metric $d$.
The most obvious examples are the \emph{sum} and \emph{max} symmetrisations, defined by
\begin{gather*}
d^{(1)}(x,y):=d^{\mathrm{sum}}(x,y)=\tfrac{1}{2}(d(x,y)+d(y,x))\quad\text{and}\\
d^{(\infty)}(x,y)=d^{\max}(x,y):=\max\{d(x,y),d(y,x)\},\end{gather*}
but there are also other symmetrisations.

\begin{definition}[metric symmetrisation]\label{defn:metricsymm}
Indeed, we can interpolate between them via the following family of metrics
\begin{align}
d^{(p)}(x,y):=
\left\{\begin{array}{rl}
\sqrt[p]{\tfrac{1}{2}(d(x,y)^p+d(y,x)^p)},&\quad\text{ for }p\in[1,\infty),\\
\max\{d(x,y),d(y,x)\},&\quad\text{ for }p=\infty.
\end{array}
\right.
\end{align}
\end{definition}

We can also symmetrise within the Finsler category by symmetrising the Finsler norm of the metric. 

\begin{proposition}\label{prop:finslersymm}
Given a Finsler metric space $(X,d)$ where $d$ is the path metric induced by the Finsler norm ${\|\cdot\|}$. For every $p\in[1,\infty]$, the function ${\|\cdot\|}^{(p)}:T_xX\to[0,\infty)$,
\begin{align*}
\|v\|^{(p)}:=
\left\{
\begin{array}{rl}
\sqrt[p]{\tfrac{1}{2}(\| v\|^p+\|-v\|^p)},&\quad\text{ for }p\in[1,\infty),\\
\max\{\|v\|,\|-v\|\},&\quad\text{ for }p=\infty.
\end{array}
\right.
\end{align*}
defines a (symmetric) norm on  $T_xX$ for every $x\in X$.
\end{proposition}

\begin{proof}
We first note that the properties
\begin{itemize}
\item $\| v\|^{(p)}=0$ if and only if $v=0$, and
\item $\| \lambda v\|^{(p)}= |\lambda| \| v\|^{(p)}$ for all $v\in T_xX$ and for all $\lambda\in\mathbb{R}$
\end{itemize}
trivially follow from ${\|\cdot\|}$ being a weak norm. Since ${\|\cdot\|}^{(p)}$ is positively homogeneous, verifying that it satisfies the triangle inequality is equivalent to showing that ${\|\cdot\|}^{(p)}$ is a convex function. For $p\neq\infty$, consider arbitrary $v,w\in T_x X$ with $\| v\|^{(p)}=\| w\|^{(p)}=1$, for any $t\in[0,1]$,
\begin{align*}
&(\|tv+(1-t)w\|^{(p)})^p\\
=
&\tfrac{1}{2}\left(\|tv+(1-t)w\|^p+\|t(-v)+(1-t)(-w)\|^p\right)\\
\leq
&\tfrac{1}{2}\left((t\|v\|+(1-t)\|w\|)^p+(t\|-v\|+(1-t)\|-w\|)^p\right).
\end{align*}
By Jensen's inequality,
\begin{align*}
(t\|v\|+(1-t)\|w\|)^p
%&=
%(t^{\frac{1}{p}}\|v\|\cdot t^{1-\frac{1}{p}}+(1-t)^{\frac{1}{p}}\|w\|\cdot (1-t)^{1-\frac{1}{p}})^p\\
%&\leq
%(t\|v\|^{p}+(1-t)\|w\|^p)(t+1-t)^{p(1-\frac{1}{p})}\\
%&=
\leq
t\|v\|^{p}+(1-t)\|w\|^p.
\end{align*}
The requisite convexity follows:
\begin{align*}
(\|tv+(1-t)w\|^{(p)})^p
\leq
&\tfrac{1}{2}\left(t\|v\|^p+(1-t)\|w\|^p+t\|-v\|^p+(1-t)\|-w\|^p\right)\\
=
&t(\|v\|^{(p)})^p+(1-t)(\|w\|^{(p)})^p=1.
\end{align*}
The case for $p=\infty$ is straightforward.
\end{proof}

\begin{definition}[Finsler symmetrisation]\label{defn:finslersymm}
 Let $(X,d)$ be a Finsler manifold with path metric $d$ induced by a Finsler norm $\|\cdot\|$. We refer to the path metric for the Finsler norm ${\|\cdot\|}^{(p)}$, where $p\in[1,\infty]$, as the $p$-th \emph{Finsler symmetrisation} of $d$. We denote such a path metric by $D^{(p)}$.
\end{definition}

\begin{proposition}\label{prop:DpvsD1}
For every $p\in[1,\infty]$, we have $D^{(p)}\geq D^{(1)}$.
\end{proposition}

\begin{proof}
This is clear for $p\infty$. For $p\in[1,\infty)$, Jensen's inequality asserts that
\begin{align*}
&(\|tv+(1-t)w\|^{(p)})^p\\
=
&\tfrac{1}{2}\left(\|tv+(1-t)w\|^p+\|t(-v)+(1-t)(-w)\|^p\right)\\
\geq
&(\tfrac{1}{2}\left(\|tv+(1-t)w\|+\|t(-v)+(1-t)(-w)\|\right))^p.
\end{align*}
Therefore $\|\cdot\|^{(p)}\geq\|\cdot\|^{(1)}$, and hence $D^{(p)}\geq D^{(1)}$.
\end{proof}

\begin{remark}
The last author and Troyanov gave (in \cite{PT-Harmonic}) a geometric description of the transformation of the Finsler norm for the $D^{(1)}$ symmetrisation of the distance function; this is based on an operation called the ``harmonic symmetrisation" at the level of the indicatrix. \cite[Lemma~4.5]{PT-Harmonic} asserts that $D^{(1)}\geq d^{(1)}$. However, we are unaware of any such relationships between $D^{(p)}$ and $d^{(p)}$ for general $p\in(1,\infty]$.
\end{remark}

In light of \cref{prop:DpvsD1}, this means that $D^{(p)}\geq d^{(1)}$ for all $p\in[0,\infty]$ and so by \cref{prop:FDmorphisms} there is a natural $1$-Lipschitz extension of the identity map between the respective FD-completions of $(\T(S),D^{(p)})$ and $(\T(S),d^{(p)})$. We show that such an extension map is homeomorphic:

\begin{theorem}\label{thm:universalcompletion}
The identity map $\mathrm{id}:\T(S)\to\T(S)$ induces a unique homeomorphic extension between any two of the following:
\begin{itemize}
\item
the forward FD-completion of $(\T(S),d_e)$;
\item
the backward FD-completion of $(\T(S),d_e)$;
\item
the forward FD-completion of $(\T(S),d_e^\#)$;
\item
the backward FD-completion of $(\T(S),d_e^\#)$;
\item
the (Cauchy) completion of $(\T(S),d_{\mathrm{WP}})$;
\item
the (Cauchy) completion of $(\T(S),d_e^{(p)})$, for any $p\in[1,\infty]$; and
\item
the (Cauchy) completion of $(\T(S),D_e^{(p)})$, for any $p\in[1,\infty]$.
\end{itemize}
\end{theorem}

\begin{proof}
We begin with a few observations.
\begin{itemize}
\item
As \cref{rmk:backwardmetric} explains, the forward FD-completion of $(\T(S),d_e)$ is the same asymmetric metric space as the backward FD-completion of $(\T(S),d_e^\#)$.
\item
Similarly, the backward FD-completion of $(\T(S),d_e)$ is identical to the forward FD-completion of $(\T(S),d_e^\#)$.

\item
By \cref{thm:lr:wp:completion}, the forward FD-completion of $(\T(S),d_e)$, the backward FD-completion of $(\T(S),d_e)$, and the completion of $(\T(S),d_{\mathrm{WP}})$ all admit unique homeomorphic extensions of the identity map.

\item
The symmetrised metrics $d_e^{(p)}$, for $p\in[1,\infty]$, are all bi-Lipschitz equivalent to each other, because $d_e^{(\infty)}(x,y)\leq \sqrt[p]{2}d_e^{(p)}(x,y)\leq \sqrt[p]{2}d_e^{(\infty)}(x,y)$. Hence, they all have homeomorphic Cauchy completions.

\item
Likewise, the symmetrised metrics $D_e^{(p)}$, for $p\in[1,\infty]$, are also all bi-Lipschitz equivalent to each other, because $\|v\|^{(\infty)}\leq \sqrt[p]{2}\|v\|^{(p)}\leq \sqrt[p]{2}\|v\|^{(\infty)}$. Hence, they also all induce homeomorphic Cauchy completions.
\end{itemize}
All of these things combined mean that we only need to establish unique homeomorphic extension in two cases:
\begin{enumerate}
\item
between $(\T(S),d_e^{(\infty)})$ and $(\T(S),d_{\mathrm{WP}})$, and
\item
between $(\T(S),D_e^{(\infty)})$ and $(\T(S),d_{\mathrm{WP}})$.
\end{enumerate}
Since $d_e\leq d_e^{(\infty)}\leq C_1 d_{\mathrm{WP}}$ and $d_e\leq D_e^{(\infty)}\leq C_1 d_{\mathrm{WP}}$, and $d_e^{(\infty)}$ and $D_e^{(\infty)}$ are Busemannian (symmetric) metrics, the statements and proofs of \cref{lem:fdiswp}, \cref{lem:fd-equi:wp-equi} and of \cref{thm:lr:wp:completion} still hold with the earthquake metric supplanted by either $d_e^{(\infty)}$ or $D_e^{(\infty)}$.
\end{proof}

%The proof is divided into two lemmas.
%\begin{lemma}\label{lem:e:to:wp}
%If $(x_n)$ in $\T(S)$ is a Cauchy sequence with respect to the symmetrized metric {\color{red}$d_S$} of the earthquake metric, then it is also a Cauchy sequence with respect to the Weil--Petersson metric.
%\end{lemma}

%\begin{proof}
%The proof is similar to that of \cref{lem:fdiswp}.
%\end{proof}

%\begin{lemma}\label{lem:wp:to:e}
%If $(x_n)$ in $\T(S)$ is a Cauchy sequence with respect to the Weil--Petersson metric, then it is also a Cauchy sequence with respect to the symmetrized metric $d_S$ of the  earthquake metric.
%\end{lemma}

%\begin{proof}
%This follows directly from \cref{thm:e:wp}. 
%\end{proof}

%\begin{proof}[Proof of \cref{thm:completion}]
%By  \cref{lem:e:to:wp} and \cref{lem:wp:to:e}, the identity map  $\mathrm{Id}:\T(S)\to \T(S)$ induces a bijection $\Pi:\overline{\T(S)}^S\to\overline{\T(S)}^\mathrm{WP}$. By \cref{cor:e<wp}, $\Pi^{-1}$ is continuous. Applying an argument similar to the proof of \cref{lem:lr:wp:homeo}, we see that $\Pi$ is also continuous. This completes the proof.
%\end{proof}

 %Moreover, the homeomorphisms the above theorem are all induced by the identity  from $\T(S)$ to itself. 
 %\HP{It is better to say that these completions are the same, or isometric.}
 %\cref{thm:lr:completion} follows from \cref{thm:lr:wp:completion} and \cref{thm:s:wp:completion}.

%=============
%=============
%=============
%=============
\newpage
\section{Coarse comparison with the Weil--Petersson metric}
\label{s:vsWP}

Although the earthquake metric is bounded above by some constant multiple of the Weil--Petersson (\cref{thm:e:wp}), the two metrics are \emph{not} bi-Lipschitz (\cref{prop:notbilipschitz}). Such fine-scale difference is dramatically felt as one approaches the completion locus. Nevertheless, the fact that the completions of these two metrics topologically agree suggests that the coarse geometry of Teichm\"uller space with respect to the earthquake metric and the Weil--Petersson metric should see no such difference. 
To be more precise, the following holds.\medskip

\noindent\textbf{\cref{thm:coarse:geometry}.} \textit{The identity map on Teichm\"uller space is a quasi-isometry between $(\T(S),d_e)$ and $(\T(S),d_{\mathrm{WP}})$. Specifically, there are constants $D_1,D_2$ which depend only on the topology of $S$, such that
\begin{align*}
\text{for any } x,y\in\T(S),\quad 
&D_1 \dwp (x,y)-D_2\leq d_e(x,y)\leq C \dwp (x,y)
\end{align*}
where $C$ is the constant from \cref{cor:e<wp}.}

\subsection{Pants graph} 

\begin{definition}[elementary moves for pants decompositions]
%A \emph{pants decomposition} of $S$ is a maximal collection of distinct isotopy classes of disjoint essential simple closed curves on $S$. 
Two pants decompositions $P$ and $P'$ of $S$ are related by an \emph{elementary move} if $P'$ can be obtained from $P$ by replacing (one isotopy class of) a simple closed curve $\alpha\in P$ by a different one $\alpha'$ which minimally intersects $\alpha$, and keeping the other curves unchanged.
\end{definition}

\begin{definition}[pants graph]
The \emph{pants graph} $\mathcal{P}(S)$ is a graph whose vertices are pants decompositions and where two pants decompositions $P$ and $P'$ are connected by an edge if they are related by an elementary move.  We metrise $\P(S)$ with a path metric $d_{\mathcal{P}}$ by setting the length of each edge to be $1$.
\end{definition}

Brock showed in \cite{Brock2003} that the pants graph is a coarse combinatorial model for $(\T(S),\dwp)$. We bridge the earthquake metric and the Weil--Petersson metric via the pants graph. We appeal to Brock's work and his ideas in our arguments.

Let $L$ be a positive constant greater than the  Bers constant. Then, every hyperbolic surface $x\in\T(S)$ admits a pants decomposition $P$ such that $\ell_{\alpha}(m)< L$ for every $\alpha\in P$.

\begin{theorem}[\cite{Brock2003}, Theorem 1.1]\label{thm:Brock}
Let $\pi:\T(S)\to\mathcal{P}(S)$ be a map which sends $x\in\T(S)$ to any pants decomposition $P$ satisfying  $\ell_{\alpha}(m)\leq L$ for every $\alpha\in P$.  The map $\pi:(\T(S),\dwp)\to(\mathcal{P}(S),d_{\mathcal{P}})$ is a quasi-isometry.
\end{theorem}

Proving the main result of this section reduces to the following lemma.

\begin{lemma}\label{lem:e:pants:lower}
For any $\pi:\T(S)\to\mathcal{P}(S)$ defined as above, there exist constants $C'$ and $D'$ such that for any $x,y\in\T(S)$, 
\begin{align*}
d_e(x,y)\geq C' d_\mathcal{P}(\pi(x),\pi(y))-D'.
\end{align*}
\end{lemma}

\begin{proof}[Proof of \cref{thm:coarse:geometry} assuming \cref{lem:e:pants:lower}]
	By \cref{cor:e<wp},
\begin{align*}
d_e(x,y)\leq C_1 \dwp(x,y)\quad \text{ and }\quad d_e(y,x)\leq C_1 \dwp (y,x)=\dwp(x,y)
\end{align*}
for any $x,y\in\T(S)$. The converse follows from \cref{lem:e:pants:lower} and \cref{thm:Brock}.
\end{proof}

\subsection{Proof of \cref{lem:e:pants:lower}} 
 For each pants decomposition $P$, we set
 \begin{align*}
 	V_P(L):=\{x\in\T(S)\mid  \ell_\alpha(x)< L \text{ for each } \alpha\in P\}.
 \end{align*}
Since $L$ is greater than the Bers constant , $\{V_P(L)\}$ is an open cover  of the \tec space.   

\begin{lemma}
[\cite{Brock2003}, Lemma 3.3]\label{lem:Brock} Given $L'>L$, there exists a constant $D_2$ depending on $L'$ such that if $V_P(L')\cap V_{P'}(L')\neq \emptyset $ then
$
d_{\P}(P,P')\leq D_2.
$
\end{lemma}

\begin{lemma}\label{lem:cover:J}
	Given $L'>L$, there exists an integer $J$ such that for every compact path $\beta\subset\T(S)$ of unit length with respect to the earthquake metric, there exist pants decompositions $P_1,\cdots, P_J$ such that $\beta\subset V_{P_1}(L')\cup V_{P_2}(L')\cup \cdots \cup V_{P_J}(L'). $
\end{lemma}
 \begin{proof}
By the choice of $L$ and the compactness of $\beta$, we see that  there exist pants decompositions $P_1,P_2,\cdots, P_k$ such that $\beta$ is covered  by the union of $V_{P_1}(L)$, $V_{P_2}(L)$, $\cdots$, $V_{P_k}(L)$.  Since $L'>L$, this implies that
$
\beta\subset \bigcup_{i=1}^k V_{P_i}(L').
$
For any pants decomposition $P$, let $\overline{V_{P}(L')}$  be  the closure of $V_{P_i}(L')$ in the Weil-Petersson completion. Let $\mathbf{DT}(P)$ be the subgroup of the mapping class group generated by Dehn twists around simple closed curves in $P$. Then the quotients $\overline{V_{P}(L')}/\mathbf{DT}(P)$ and $\overline{V_{P}(L)}/\mathbf{DT}(P)$ are compact. Since the earthquake metric is mapping class group invariant and extends to the Weil-Petersson completion (\cref{thm:metric:extension}), it follows that  there exists a constant $c>0$ such that for all $P\in\P(S)$, we have 
\[d_e(\partial V_P(L),\partial V_P(L'))\geq c\quad\text{ and }\quad d_e(\partial V_P(L'),\partial V_P(L))\geq c.
\]
Therefore, each component of $\beta\cap (V_{P_i}(L')\setminus V_{P_i}(L))$ which meets $\overline{V_{P_i}(L)}$ has length at least $c$. Thus, if $J$ is the smallest integer larger than $1/c$,  then we may cover $\beta$ by a subcovering of $\{V_{P_i}(L')\mid 1\leq i\leq k\}$ which has at most $J$ elements.
\end{proof}

\begin{proof}[Proof of \cref{lem:e:pants:lower}]
Let $x$ and $y$ be arbitrary points in $\T(S)$. Let $p$ be a path connecting $x$ to $y$ whose earthquake metric length $\mathrm{L}(p)$ satisfies 
$
|\mathrm{L}(p)-d_e(x,y)|\leq 1.
$
Let $K$ be the smallest integer larger than $\mathrm{L}(p)$.  We subdivide $p$ into segments $p_1,p_2,\cdots,p_K$, each of which has length at most one.  By \cref{lem:cover:J} there exists a collection $\{P_j\mid 1\leq j\leq KJ\}$ of pants decompositions such that
$
p\subset \bigcup_{1\leq j\leq KJ} V_{P_j}(L').$

Combined with \cref{lem:Brock}, this implies that $\pi(P)$ and $\pi(P')$ can be connected by a path in the pants graph with length at most $(KJ+1)D_2$.  In particular,
\begin{align*}
d_{\P}(\pi(x),\pi(y))
&\leq D_2(KJ+1)\\
&\leq D_2J(\ell(\sigma)+1)+D_2 \\
&\leq D_2Jd_e(m,m') +(2J+1)D_2.
\end{align*}
This completes the proof.
\end{proof}

\newpage
\appendix

\section{Topology of Finsler metrics of low regularity}

For an asymmetric metric space, we can  define a forward topology and a backward topology, with sub-bases respectively consisting of forward and backward open balls of the metric.
In general these two topologies may be different. Busemann studied these notions in Chapter~1 of \cite{Busemann-synthetic}, and gave conditions under which the forward and backward topologies coincide with the topology of the symmetrised metric 
$
d^{\max}(x,y)=\max\{d(x,y),d(y,x)\}.
$

In this appendix, after making precise what we mean by a Finsler (asymmetric) metric, we shall consider the topology of such metrics. We show that the induced topology is the same as the topology of the underlying differentiable manifold.  This will show in particular that  the topology induced on Teichm\"uller space  by the earthquake metric coincides with the usual topology.

For the remainder of this appendix, we take the following setup: 
\begin{itemize}
\item
$X$ is a $C^1$-manifold;

\item
there is a continuous function ${\|\cdot\|}:TX\to[0,\infty)$ called the Finsler metric (see \cref{defn:finslermetric}) on $X$, and we denote the restriction of ${\|\cdot\|}$ to a tangent space $T_x X$ by ${\|\cdot\|}_x$; and

\item
$d$ is the induced path metric of ${\|\cdot\|}$ on $X$.
\end{itemize}

Let $x$ be an \emph{arbitrary} point of $X$ and $(U,\phi)$ a  local chart around the point $x$ such that $\phi(x)=0$ and that the image of $\phi:U\to \R^m$ is the unit ball. Let ${\|\cdot\|}'$ be the pullback to $T U$ of the Euclidean norm on $\R^m$ via $\phi$.  Let $d'$ be the metric induced by ${\|\cdot\|}'$ on $U\subset M$. 

\begin{lemma}\label{lem:comparison}
There exists a positive constant $C=C(U,\phi)\geq 1$   such that for any $y$ with $d'(x,y)\leq\tfrac{3}{4}$ and  any $v\in T_y X$, we have
\begin{equation*}
\frac{1}{C}\leq \frac{\|v\|_y}{\|v\|_y'}\leq C.
\end{equation*}
In particular, $d(y,z)\leq C d'(y,z)$ for any $y,z\in U$ with $d'(x,y)\leq\tfrac{3}{4}$ and $d'(x,z)\leq\tfrac{3}{4}$.
\end{lemma}

\begin{proof}
The existence of $C$ follows from ${\|\cdot\|}$ being continuous, the ratio of norms being homogeneous of degree $0$, and the compactness of the unit tangent bundle over $\{y\in U\mid d'(x,y)\leq\tfrac{3}{4}\}$. The distance bound uses the geodesic convexity of $\{y\in U\mid d'(x,y)\leq\tfrac{3}{4}\}$ with respect to $d'$.
\end{proof}

\begin{lemma}\label{cor:Thurston:earthquake2}
For any $R\in(0,\tfrac{3}{4})$, if $d'(x,y)\geq R $, then both $d(x,y),d(y,x)\geq \tfrac{R}{C}$, where $C$ is the constant from \cref{lem:comparison}.
\end{lemma}

\begin{proof}
Consider the open subset
\[
V:=\{z\in U\mid d'(x,z)< R\}.
\]
By \cref{lem:comparison}, for any piecewise $C^1$ path $p$ connecting $x$ to any $y$ such that $d(x,y)\geq \tfrac{R}{C}$,  the length $\mathrm{L}(p)$ of $p$ with respect to ${\|\cdot\|}$  and the length $\mathrm{L}'(p)$ with respect to ${\|\cdot\|}'$ satisfy:
\begin{align*}
\mathrm{L}(p)
\geq\mathrm{L}( p\cap V) 
\geq\tfrac{1}{C}\mathrm{L}'( p\cap V)\geq \tfrac{R}{C}.
\end{align*}
Consequently,  $d(x,y)\geq \frac{R}{C}$.  Similarly, we have $d(y,x)\geq \frac{R}{C}$. 
\end{proof}
%Let $B_e(m,\epsilon):=\{m'\in\T(S): d_e(m',m)<\epsilon\}$ and $B^R_e(m,\epsilon):=\{m'\in\T(S): d_e(m,m')<\epsilon\}$

From \cref{cor:Thurston:earthquake2}, we deduce the following:
\begin{corollary}\label{cor:distinguishes}
A Finsler metric distinguishes points. In other words, for any $x$ and $y$ in $X$, we have  $d(x,y)=0$ if and only if $x=y$. 
\end{corollary}

\begin{lemma}\label{lem:earthquake:compact}
	There exists $\epsilon_0>0$ such that both $\{y\in X\mid d(x,y)\leq \epsilon_0\}$ and $\{y\in X \mid d(y,x)\leq \epsilon_0\}$ are compact.
\end{lemma}
\begin{proof}
	By \cref{cor:Thurston:earthquake2},  there exists $\delta>0$ such that for any $y\in  X$ with $d'(x,y)\geq \tfrac{1}{2}$, we have $d(x,y)\geq \delta$ and $d(y,x)\geq \delta$. This implies that  both closed sets $\{y\in X \mid d(x,y)\leq \tfrac{\delta}{2}\}$ and $\{y\in X\mid d(y,x)\leq \tfrac{\delta}{2}\}$ are contained in the compact subset $\{y\in U\mid d'(x,y)\leq \tfrac{1}{2}\}$, hence are  also compact.  The lemma follows by letting $\epsilon_0$ be any positive constant smaller than $\tfrac{\delta}{2}$.
\end{proof}

\begin{lemma}\label{lem:local:comparison}
There exists a neighbourhood $V\subset U$ of $x$  such that
\begin{align*}
\tfrac{1}{2C} d'(y,z)
\leq 
d(y,z)\leq C d'(y,z)
\end{align*}
holds for any $y,z\in V$, where $C$ is the constant from \cref{lem:comparison}.
\end{lemma}

\begin{proof}
Consider the compact subset $K_0:=\{y\in X:d(x,y)\leq \epsilon_0\}$, where $\epsilon_0$ is the constant from \cref{lem:earthquake:compact}. By \cref{lem:comparison}, there exists a constant $C$ such that  
\begin{align}\label{eq:Th:e:upper}
d(y,z)\leq Cd'(y,z)
\end{align}
 holds for any $y,z$ in  $K_0$. Now we set
\begin{align*}
K:=\{y\in X \mid d(x,y)\leq \tfrac{\epsilon_0}{5}, d'(x,y)\leq \tfrac{4\epsilon_0}{5C}\}.
\end{align*}
Consider two arbitrary points $y,z$ in $K$. We have
\begin{align*}
d(y,z)\leq Cd'(y,z)\leq\tfrac{2\epsilon_0}{5}.
\end{align*}
Any piecewise $C^1$-path $p$ in $X$ connecting $y$ to $z$ with length (with respect to $d$) less than $2d(y,z)$, which is at most $\frac{4\epsilon_0}{5}$,  is contained in $K_0$, since the distance from $x$ of any point $u$ on $p$ satisfies
\[
\epsilon_0=
\tfrac{\epsilon_0}{5}+\tfrac{4\epsilon_0}{5}
\geq
d(x,y)+d(y,u)\geq d(x,u).
\]
By \cref{lem:comparison},  the length $\mathrm{L}'(p)$ of $p$ with respect to $d'$ and the length $\mathrm L(p)$ of the same path with respect to $d$ satisfy
$\mathrm{L}'(p)\leq C \mathrm{L}(p)\leq 2Cd(y,z).$
Hence, 
\begin{equation}\label{eq:Th:e:lower}
d'(y,z)\leq \mathrm{L}'(p)\leq 2Cd(y,z).
\end{equation} Let $V$ be the interior of $K$.  The lemma follows from \cref{eq:Th:e:upper} and \cref{eq:Th:e:lower}. 
\end{proof}

\begin{corollary}
There exists a neighbourhood $V\subset U$ of $x$  such that
\begin{align*}
\tfrac{1}{4C^2} d(z,y)
\leq 
d(y,z)\leq 4C^2 d(z,y)
\end{align*}
holds for every $y,z\in V$, where $C$ is the constant from \cref{lem:comparison}.
\end{corollary}

\begin{proposition}
\label{prop:finslertopology}
The forward topology induced by $d$, the backward topology induced by $d$, and the topology induced by various symmetrisation of $d$ (\cref{defn:metricsymm} and \cref{defn:finslersymm})  are all the same as the underlying topology of $X$ as a manifold.
\end{proposition}

\begin{proof}
This follows directly from \cref{lem:local:comparison}.
\end{proof}

\begin{theorem}
\label{prop:finslerbusemann}
Finsler manifolds are Busemannian metric spaces.
\end{theorem}

\begin{proof}
We need to verify that an arbitrary Finsler metric $d$ satisfies
\begin{itemize}
\item $\text{for any }x,y\in X$, $d(x,y)=0\Leftrightarrow x=y$;
\item $\text{for any }x,y,z\in X$, $d(x,z)\leq d(x,y)+d(y,z)$; and
\item $\text{for any }  x\in X$ and $\text{any sequence } (x_n)$ in $X$, $d(x_n,x)\to 0$ if and only if $d(x,x_n)\to 0$.
\end{itemize}
The second property is a general property of path metrics, and the first and the third properties follow from the local bi-Lipschitz comparison between any Finsler metric and the Euclidean metric, as given by \cref{lem:local:comparison}.
\end{proof}

\newpage
\section{FD-completions of Asymmetric Metric Spaces}
\label{appendix:FD}

In this appendix, we develop the theory of \emph{FD-completions} of asymmetric metric spaces, a notion we defined in \cref{sec:fdcompletion}.

\subsection{Metrising the FD-completion}

We first show that the forward FD-completion $\overline{X}$ may be naturally topologised/metrised as follows.

\begin{theorem}[metrising the forward FD-completion]
\label{thm:asymmetricmetric}
Let $\overline{X}$ denote the forward FD-completion of an asymmetric metric space $(X,d)$. Then, $d$ induces a function $\bar{d}:\overline{X}\times\overline{X}\to[0,\infty]$ defined by
\[
\bar{d}([x_n],[y_m])
:=
\inf_{(x'_n)\in[x_n],\ (y'_m)\in[y_m]}\liminf_{k\to\infty} d(x'_k,y'_k),
\]
where the infimum is taken over all forward FD-sequences $(x'_n)$ and $(y'_m)$ respectively representing the forward FD-equivalence classes $[x_n]$ and $[y_m]$. The function $\bar{d}$ defines an asymmetric metric on $\overline{X}$, and we refer to $\bar{d}$ as the \emph{forward metric extension} of $d$.
\end{theorem}

Before proving the above result, we first establish a useful lemma:

\begin{lemma}
\label{lem:infimalrepsequence}
Given $[x_n],[y_m]$, there exist forward FD-sequences $(x'_n)$ and $(y'_m)$ respectively representing $[x_n]$ and $[y_m]$ such that
\[
\bar{d}([x_n],[y_m])
:=
\inf_{(\hat{x}_n)\in[x_n],\ (\hat{y}_m)\in[y_m]}\liminf_{k\to\infty} d(\hat{x}_k,\hat{y}_k)
=
\lim_{k\to\infty} d(x'_k,y'_k).
\]
\end{lemma}

\begin{proof}
Let $(x_n^1)_{n\in\mathbb{N}},(x_n^2)_{n\in\mathbb{N}},\ldots$ and $(y_m^1)_{m\in\mathbb{N}},(y_m^2)_{m\in\mathbb{N}},\ldots$ be two sequences of forward FD-sequences respectively representing $[x_n]$ and $[y_m]$, which realise the infimisation process for $\bar{d}([x_n],[y_m])$, i.e.\ 
\[
\bar{d}([x_n],[y_m])
=
\lim_{i\to\infty}
\liminf_{k\to\infty} d(x_k^i,y_k^i).
\]
By a straightforward application of the triangle inequality, since subsequences of (forward) FD-sequences are also (forward) FD-sequences, we can replace each pair $(x_k^i),(y_k^i)$ by a pair of subsequences where $\underset{k\to\infty}{\liminf}\ d(x_k^i,y_k^i)=\underset{k\to\infty}{\lim} d(x_k^i,y_k^i)$. By possibly truncating the head of each sequence, we further ensure that
\begin{align}
\text{for all } j\geq i,\quad\left|\lim_{k\to\infty} d(x_k^i,y_k^i) - d(x_j^i,y_j^i)\right|<\frac{1}{i}.\label{eq:asymapproximation}
\end{align}
We build an ``intermediate'' FD-sequence $(\tilde{x}_n)$ and $(\tilde{y}_n)$ as follows. 
\begin{itemize}
\item
Step~1. Choose $\tilde{x}_1=x^1_{n_1}$, $\tilde{x}_2=x^2_{n_2}$, $\tilde{y}_1=y^1_{n_1}$, and $\tilde{y}_2=y^2_{n_2}$ with $n_2>n_1>1$, so that 
\begin{itemize}
\item $d(\tilde{x}_1,\tilde{x}_2)<\tfrac{1}{2}$, $\sum_{k\geq n_2} d(x^2_k,x^2_{k+1})<\tfrac{1}{2^2}$, and
\item $d(\tilde{y}_1,\tilde{y}_2)<\tfrac{1}{2}$, $\sum_{k\geq n_2} d(y^2_k,y^2_{k+1})<\tfrac{1}{2^2}$.
\end{itemize}
 This is always possible because $(x^1_n)$ and $(x^2_n)$ (resp. $(y^1_n)$ and $(y^2_n)$) are FD-equivalent. We informally refer to this as \emph{hopping} from $(x^1_n)$ and $(y^1_n)$ to $(x^2_n)$ and $(y^2_n)$. In what follows, we always ensure that the $n_k$ in $x^i_{n_k}$ is greater than $k$. 

\item
Step~2. We hop from $(x^2_n)$ and $(y^2_n)$ back to $(x^1_n)$ and $(y^1_n)$ by choosing $\tilde{x}_3=x^2_{n_3}$, $\tilde{x}_4=x^1_{n_4}$, $\tilde{y}_3=y^2_{n_3}$ and $\tilde{y}_4=y^1_{n_4}$, with respectively increased indices $n_4>n_3>n_2>n_1$, chosen so that
\begin{itemize}
\item $d(\tilde{x}_3,\tilde{x}_4)<\tfrac{1}{2^3}$, $\sum_{k\geq n_4} d(x^1_k,x^1_{k+1})<\tfrac{1}{2^4}$, and
\item $d(\tilde{y}_3,\tilde{y}_4)<\tfrac{1}{2^3}$ and $\sum_{k\geq n_4} d(y^1_k,y^1_{k+1})<\tfrac{1}{2^4}$.
\end{itemize}
Since $n_3>n_2$, we see that $d(\tilde x_2,\tilde x_3)\leq \sum_{k\geq n_2} d(x^2_k,x^2_{k+1})<\tfrac{1}{2^2} $. Similarly, we have $d(\tilde y_2,\tilde y_3)\leq \frac{1}{2^2}$.
\item
Step~3. We hop from $(x^1_n)$ and $(y^1_n)$ to $(x^3_n)$ and $(y^3_n)$ by choosing $\tilde{x}_5=x^1_{n_5}$,  $\tilde{x}_6=x^3_{n_6}$, $\tilde{y}_5=y^1_{n_5}$,  and $\tilde{y}_6=x^3_{n_6}$, with $n_6>n_5>n_4$, so that 
\begin{itemize}
\item $d(\tilde{x}_5,\tilde{x}_6)<\tfrac{1}{2^5}$,  $\sum_{k\geq n_6} d(x^3_k,x^3_{k+1})<\tfrac{1}{2^6}$, and
\item $d(\tilde{y}_5,\tilde{y}_6)<\tfrac{1}{2^5}$ and $\sum_{k\geq n_6} d(y^3_k,y^3_{k+1})<\tfrac{1}{2^6}$.
\end{itemize}
Since $n_5>n_4$, we see that $d(\tilde x_4,\tilde x_5)\leq \sum_{k\geq n_4} d(x^1_k,x^1_{k+1})<\tfrac{1}{2^4} $. Similarly, we have $d(\tilde y_4,\tilde y_5)\leq \frac{1}{2^4}$.

\item
Step~4. We hop from $(x^3_n)$ and $(y^3_n)$  back to $(x^1_n)$ and $(y^1_n)$ by choosing $\tilde{x}_7=x^3_{n_7}$, $\tilde{x}_8=x^1_{n_8}$, $\tilde{y}_7=y^3_{n_7}$ and $\tilde{y}_8=y^1_{n_8}$ with increased indices $n_8>n_7>n_6$, so that
\begin{itemize}
\item  $d(\tilde{x}_7,\tilde{x}_8)<\tfrac{1}{2^7}$, $\sum_{k\geq n_8} d(x^1_k,x^1_{k+1})<\tfrac{1}{2^8}$, and
\item  $d(\tilde{y}_7,\tilde{y}_8)<\tfrac{1}{2^7}$ and $\sum_{k\geq n_8} d(y^1_k,y^1_{k+1})<\tfrac{1}{2^8}$.
\end{itemize}
Since $n_7>n_6$, we see that $d(\tilde x_6,\tilde x_7)\leq \sum_{k\geq n_2} d(x^3_k,x^3_{k+1})<\tfrac{1}{2^6} $. Similarly, we have $d(\tilde y_6,\tilde y_7)\leq \frac{1}{2^6}$.
\item
We repeat the above steps ad infinitum, doing ``rounds'' of hops from $(x^1_n)$ and $(y^1_n)$ to  new sequences $(x^i_n)$ and $(y^i_n)$ and then hopping back to $(x^1_n)$ and $(y^1_n)$.
\end{itemize}
This process yields two forward FD-sequences $(\tilde{x}_n)$  and  $(\tilde{y}_n)$ because 
 $\sum_{n}^\infty d(\tilde{x}_n,\tilde{x}_{n+1})<\sum_{n}\tfrac{1}{2^n}<\infty$ and $\sum_{n}^\infty d(\tilde{y}_n,\tilde{y}_{n+1})<\sum_{n}\tfrac{1}{2^n}<\infty$. 
 
Define $(x'_i)$ by $x'_i=\tilde{x}_{4i+2}=x^{i+2}_{n_{4i+2}}$ and $(y'_j)$ by $y'_j=\tilde{y}_{4j+2}=y^{j+2}_{m_{4j+2}}$. Since $(x'_i)$ and $(y'_j)$ are subsequences of FD-sequences, they must also be FD-sequences. Moreover, $(x'_i)$ is FD-equivalent to $(\tilde{x}_n)$, which in turn shares a subsequence with $(x^1_n)$ and is hence FD-equivalent to $(x^1_n)$. Thus, $(x'_i)\in[x_n]$. Likewise, $(y'_j)\in[y_m]$. Finally, since $n_{i+1}>n_i>\cdots >n_1>1$, we have $n_i\geq i+2$. Hence,  \cref{eq:asymapproximation} ensures that
\begin{align*}
\lim_{i\to\infty} d(x'_i,y'_i)
=
\lim_{i\to\infty} d(x^{i+2}_{n_{4i+2}},y^{i+2}_{n_{4i+2}})
=
\lim_{i\to\infty} \lim_{k\to\infty} d(x^{i+2}_{k},y^{i+2}_{k})
=
\bar{d}([x_n],[y_m]).
\end{align*}

\end{proof}

\begin{proof}[Proof of \cref{thm:asymmetricmetric}]
We verify the three axioms for asymmetric metrics. To begin with, $\bar{d}([x_n],[x_n])=0$ because $d(x_n,x_n)=0$. 

(i) We next show that if $\bar{d}([x_i],[y_j])=\bar{d}([y_j],[x_i])=0$, then $[x_i]=[x_j]$. By \cref{lem:infimalrepsequence}, there are forward FD-sequences $(x_i),(x'_i)$ which both represent $[x_i]$ and forward FD-sequences $(y_j),(y'_j)$ which both represent $[y_j]$ so that
\begin{align}
\bar{d}([x_i],[y_j])=\underset{k\to\infty}{\lim} d(x_k,y_k)=0,
\text{ and }
\bar{d}([y_j],[x_i])=\underset{k\to\infty}{\lim} d(y'_k,x'_k)=0.\label{eq:tendsto0}
\end{align}
 Since $(x_i)$ and $(x_i')$ (resp. $(y_i)$ and $(y_i')$) are FD-equivalent, after taking subsequences if necessary, we may assume the interlacings $x_1,x_1',x_2,x_2',\cdots$ and $y_1,y_1',y_2,y_2',\cdots$ are forward FD-sequences. To see this is possible, we first take an interlacing $x_{i_1},x'_{i_2},x_{i_3},x_{i_4}',\cdots$ so that it is a forward FD-sequence. We next take an interlacing $y_{i_{j_1}}, y_{i_{j_2}}',y_{i_{j_3}}, y_{i_{j_4}}',y_{i_{j_5}}, y_{i_{j_6}}',\cdots$ of $(y_{i_k})$ and $(y_{i_k}')$ so that the interlacing is a forward FD-sequence. Finally, we replace $(x_k)$, $(x_k')$, $(y_k)$, and $(y_k')$ by $(x_{i_{j_{2k}}})$, $(x_{i_{j_{2k+1}}}')$, $(y_{i_{j_{2k}}})$, and $(y_{i_{j_{2k+1}}}')$. The resulting sequences satisfy the desired property. Hence,
\[
\lim_{k\to\infty} d(x_k,{y}'_k)
\leq
\lim_{k\to\infty} 
d(x_{k},y_{k})
+
d(y_{k},y'_{k})
=0.
\]
Similarly,
\[
\lim_{k\to\infty} d({y}'_k,{x}_{k+1})
\leq
\lim_{k\to\infty} 
d(y'_{k},x'_{k})
+
d(x'_{k},x_{k+1})
=0.
\]
We now choose a sequence of increasing indices $(n_k)$ such that 
\begin{equation*}
	d({x}_{n_k},{y}'_{n_k})<\frac{1}{2^k},\quad \sum_{n_k\leq j\leq n_{k+1} }d(x_j,x_{j+1})<\frac{1}{2^k},\quad d(y_{n_k}',x_{n_k+1})<\frac{1}{2^k}.
\end{equation*}
Hence 
 $\sum_{k=1}^\infty d({x}_{n_k},{y}'_{n_k})<\infty$ and 
 \begin{align*}
 	\sum_{k=1}^\infty d({y}'_{n_k},{x}_{n_{k+1}}))&\leq	\sum_{k=1}^\infty \left( d(y_{n_k}',x_{n_k+1})+\sum_{n_k\leq j\leq n_{k+1}} d(x_j,x_{j+1})) \right) \\
 	&<\sum_{k=1}^\infty \left(\frac{1}{2^k}+  \frac{1}{2^k}\right )<\infty.
 \end{align*}
Thus, the interlacing given by ${x}_{n_1},{y}'_{n_1},{x}_{n_2},{y}'_{n_2},\ldots$ is an FD-sequence, showing that $(x_i)$ and $(y'_j)$ are FD-equivalent.

(ii) Finally, we show that $\bar{d}$ satisfies the triangle inequality. Consider arbitrary elements $[x_i],[y_j],[z_k]\in \overline{X}$. We invoke \cref{lem:infimalrepsequence} to obtain the following FD-sequences:
\begin{itemize}
\item
$(x_i)$ and $(y_j)$ are respective representatives for $[x_i]$ and $[y_j]$, so that $\bar{d}([x_i],[y_j])=\underset{i\to\infty}{\lim} d(x_i,y_i)$;

\item
$(y'_j)$ and $(z'_k)$ are respective representatives for $[y_j]$ and $[z_k]$, so that $\bar{d}([y_j],[z_k])=\underset{i\to\infty}{\lim} d(y'_i,z'_i)$.
\end{itemize}

Since $(y_j)$ and $(y'_j)$ are FD-equivalent, after taking subsequences if necessary, we may assume that the interlacing $y_1,y_1',y_2,y_2'\cdots$ is a forward FD-sequence. In particular, 
\begin{align*}
\lim_{j\to\infty} d(y_j,y'_{j})
=0.
\end{align*}
 The triangle inequality tells us that
\begin{align*}
\liminf_{i\to\infty}
d(x_i,z_{i}')
&\leq
\lim_{i\to\infty}
(d(x_i,y_i)+d(y_i,y'_{i})+d(y'_{i},z'_{i}))\\
&=
\bar{d}([x_i],[y_j])+0+\bar{d}([y_j],[z_k]).
\end{align*}
By definition,
\begin{align*}
\bar{d}([x_i],[z_k])
\leq
\liminf_{i\to\infty}
d(x_i,z_{i}')
\leq
\bar{d}([x_i],[y_j])+\bar{d}([y_j],[z_k]),\text{ as desired.}
\end{align*}

\end{proof}

\begin{remark}[backward metric extension]
\label{rmk:backwardmetric}
\cref{thm:asymmetricmetric} holds \emph{specifically} for the forward FD-completion, and there is a similar construction of a standard metric for the backward FD-completion. However, in order to obtain an analogous triangle inequality based on arguments like those in the latter steps of the proof of \cref{thm:asymmetricmetric} (as well as the latter half of \cref{lem:infimalrepsequence}), we define the \emph{backward metric extension} of $d$ as:
\[
\inf_{(x'_n)\in[x_n],\ (y'_m)\in[y_m]}\liminf_{k\to\infty} d(y'_k,x'_k).
\]
Note in particular, that this makes the backward FD-completion of a metric space $(X,d)$ identical to the forward FD-completion of $(X^\#,d^\#)$.
\end{remark}

\subsection{Natural inclusion}

Recall from \cref{lem:natinclude} that the map $\iota: X\to\overline{X}$ sending $x\in X$ to the FD-equivalence class represented by the constant FD-sequence $(x)_{n\in\mathbb{N}}$ naturally identifies $X$ with a subset of $\overline{X}$. We show:

\begin{proposition}[dense inclusion]
\label{prop:natisoinclude}
The natural inclusion map $\iota: X\to\overline{X}$ is an isometric embedding from $(X,d)$ to $(\overline{X},\bar{d})$. Moreover, $\iota(X)$ is \emph{forward dense} in $\overline{X}$, meaning that every point in $\overline{X}$ is the forward limit of a sequence of points $(\iota(x_n))_{n\in\mathbb{N}}$ in $\iota(X)\subset\overline{X}$.
\end{proposition}

\begin{proof}
Given two arbitrary FD-equivalence classes $[x]=\iota(x)$ and $[y]=\iota(y)$, let $(x_n)$ and $(y_n)$ denote FD-sequences (obtained via \cref{lem:infimalrepsequence}), which respectively  represent $[x]$ and $[y]$ and satisfy
\[
\bar{d}([x],[y])=\lim_{n\to\infty} d(x_n,y_n).
\]
Since the constant sequence $(x)_{n\in\mathbb{N}}$ and the sequence $(x_n)_{n\in\mathbb{N}}$ are FD-equivalent, they have an interlacing which is an FD-sequence, and hence there is a subsequence $(x_{n_k})_{k\in\mathbb{N}}$ such that
\begin{align}
\sum_{k} d(x,x_{n_k}),\ 
\sum_{k} d(x_{n_k},x)<\infty.\label{eq:finsum}
\end{align}
Replace both $(x_n)$ and $(y_n)$ with their respective subsequences indexed by $(n_k)_{k\in\mathbb{N}}$, and observe that \cref{eq:finsum} implies that
\[
\lim_{n\to\infty} 
d(x,x_n)
=0
\quad\text{and}\quad
\lim_{n\to\infty} 
d(x_n,x)=0.
\]
Since the updated $(y_n)$ is a subsequence of the former, they are FD-equivalent, hence there is an interlacing between $(y_n)_{n\in\mathbb{N}}$ and $(y)_{n\in\mathbb{N}}$ which is an FD-sequence. We similarly find a subsequence $(y_{n_k})_{k\in\mathbb{N}}$ satisfying
\begin{align}
\sum_{k} d(y,y_{n_k}),\ 
\sum_{k} d(y_{n_k},y)<\infty.\label{eq:finsum2}
\end{align}
Update both $(x_n)$ and $(y_n)$ again, and observe that \cref{eq:finsum2} yields
\[
\lim_{n\to\infty} 
d(y,y_n)
=0
\quad\text{and}\quad
\lim_{n\to\infty} 
d(y_n,y)=0.
\]
Then,
\begin{align*}
d(x_n,y_n)
&\leq d(x_n,x)+d(x,y)+d(y,y_n)\\
&\leq d(x_n,x)+d(x,x_n)+d(x_n,y_n)+d(y_n,y)+d(y,y_n).
\end{align*}
Taking the limit as $n\to\infty$, the sandwich theorem tells us that
\[
\bar{d}([x],[y])
=
\lim_{n\to\infty} d(x_n,y_n)
=
\lim_{n\to\infty} d(x,y)
=
d(x,y).
\]
Thus, $\iota$ is an isometric embedding. Density is straightforward: given an arbitrary element $\xi\in\overline{X}$ represented by an FD sequence $(x_k)_{k\in\mathbb{N}}$, we have
\begin{align*}
\lim_{k\to\infty}
\bar{d}(\iota(x_k),\xi)
\leq
\lim_{k\to\infty}
\lim_{n\to\infty}
d(x_k,x_n)
\leq
\lim_{k\to\infty}
\sum_{j=k}^\infty d(x_k,x_{k+1}).
\end{align*}
The right hand side is equal to $0$ because $(x_k)_{k\in\mathbb{N}}$ is an FD-sequence. Thus, $\bar{d}(\iota(x_k),\xi)\to0$, as desired.
\end{proof}

\begin{remark}
Analogously applying the above arguments, we obtain that $\iota^\#: X\to\overline{X}^\#$ is an isometric embedding from $(X,d^\#)$ to $(\overline{X}^\#,\bar{d}^\#)$ with backward dense image.
\end{remark}

\subsection{FD-Completeness}

\begin{theorem}
\label{thm:FDcomplete}
The forward FD-completion of any asymmetric metric space $(X,d)$ is forward FD-complete. Likewise, the backward FD-completion is backward FD-complete.
\end{theorem}

\begin{proof}
We only need to verify this for the forward FD-completion, as the backward FD-completion is equal to the forward FD-completion of the reverse metric space $(X,d^\#)$. Consider an arbitrary (forward) FD-sequence $(\xi_k)$ in $(\overline{X},\bar{d})$. Our goal is to show that there exists $\xi\in\overline{X}$ such that $\bar{d}(\xi_k,\xi)\to0$. We do so by constructing a specific FD-sequence in $(X,d)$ and showing that its corresponding FD-equivalence class satisfies the desired property.

By \cref{lem:infimalrepsequence}, for each $i\geq1$, we find FD-sequences $(\hat x_{n,i})$ and $(x'_{n,i})$ representing  $\xi_i$ such that
\begin{equation*}
	\bar d(\xi_{i},\xi_{i+1})=\lim_{n\to\infty} d(\hat x_{n,i},x'_{n,i+1}).
\end{equation*}
Since both $(\hat x_{n,i})$ and $(x'_{n,i})$ represent $\xi_i$, they are FD-equivalent. After taking subsequences, we may assume that for each $i\geq 1$ the interlacing 
\begin{equation*}
	\hat x_{1,i},x'_{1,i},\hat x_{2,i},x_{2,i}',\cdots 
\end{equation*}
is a forward FD-sequence. In particular, $\lim\limits_{n\to\infty}d( x_{n,i}', \hat x_{n+1,i}')=0$. 
Hence, 

\begin{align*}
	 \lim_{n\to\infty} d(\hat x_{n,i},\hat x_{n,i+1})&\leq  \lim_{n\to\infty} \left(d(\hat x_{n,i}, x_{n,i+1}')+d( x_{n,i+1}', \hat x_{n+1,i+1})\right)\\
	 &= \lim_{n\to\infty} d(\hat x_{n,i}, x_{n,i+1}') =\bar d(\xi_i,\xi_{i+1}).
\end{align*}
On the other hand, the definition of $\bar d$ implies that 
\begin{equation*}
	\bar d(\xi_i,\xi_{i+1})\leq  \lim_{n\to\infty} d(\hat x_{n,i},\hat x_{n,i+1}).
\end{equation*}
Hence, for each $i\geq1$
\begin{equation*}
	\bar d(\xi_i,\xi_{i+1})=  \lim_{n\to\infty} d(\hat x_{n,i},\hat x_{n,i+1}).
\end{equation*}

To simplify notation, we set $x_{m,k}:=\hat x_{m,k}$.
Take a sequence of (strictly) increasing indices $(m_k)$ satisfying 
\begin{align}
\text{for all } n\geq m_k,\quad \left|d(x_{n,k},x_{n,k+1})-\bar{d}(\xi_k,\xi_{k+1})\right|\leq2^{-k},
\label{eq:firstcondition}
\end{align}
and
\begin{equation}\label{eq:secondcondition}
	\sum_{j=m_k}^\infty d(x_{j,k},x_{j+1,k})\leq2^{-k}.
\end{equation}

We verify that the sequence $(x_{m_k,k})_{k\in\mathbb{N}}$ is an FD-sequence. \cref{eq:firstcondition} and \cref{eq:secondcondition} tell us that
\begin{align*}
	&d(x_{m_k,k},x_{m_{k+1},k+1})\\
	\leq&
d(x_{m_k,k},x_{{m_k+1},k})+\cdots +d(x_{m_{k+1}-1,k},x_{m_{k+1},k})+ d(x_{m_{k+1},k},x_{m_{k+1},k+1})\\
\leq & 2^{-k}+ \bar d (\xi_k,\xi_{k+1})+2^{-k}=\bar d(\xi_{k},\xi_{k+1})+2^{-k+1}.
\end{align*}
Hence, 
\begin{equation*}
	\sum_{k=1}^\infty d(x_{m_k,k},x_{m_{k+1},k+1})\leq \sum_{k=1}^\infty (\bar d(\xi_{k},\xi_{k+1})+2^{-k+1})<\infty,
\end{equation*}
where we use the assumption that $(\xi_k)$ is a forward FD-sequence in $(\overline{X},\bar d)$.

Let $\xi\in\overline{X}$ be the FD-equivalence class for $(x_{m_k,k})$. We complete this proof by showing that $\bar{d}(\xi_k,\xi)\to0$.

To calculate the desired limit, we observe that
\begin{align*}
\lim_{k\to\infty}
\bar{d}(\xi_k,\xi)
\leq
\lim_{k\to\infty}\lim_{n\to\infty}
d(x_{m_n,k},x_{m_n,n})
\end{align*}
is true because $(x_{m_n,k})$ is a subsequence of $(x_{n,k})$ for each $k$,  and hence is FD-equivalent to $(x_{n,k})$. Finally, since $m_n>n$, 
\begin{align*}
\lim_{k\to\infty}\lim_{n\to\infty}
d(x_{m_n,k},x_{m_n,n})
&\leq
\lim_{k\to\infty}\lim_{n\to\infty}
\sum_{j=k}^{n-1}
d(x_{m_n,j},x_{m_n,j+1})\\
&\leq
\lim_{k\to\infty}
\sum_{j=k}^{n-1}
(\bar{d}(\xi_j,\xi_{j+1})+2^{-j}) \qquad(\text{by } \eqref{eq:firstcondition})\\
&
=0  \qquad(\text{since } \sum_{k=1}^\infty \bar d(\xi_k,\xi_{k+1})<\infty ),
\end{align*}
as required. Therefore, $(\overline{X},\bar{d})$ is forward FD-complete.
\end{proof}

\begin{remark}[non-idempotence]
\label{rmk:example}
Although the forward FD-completion is forward FD-complete, we emphasise that successive applications of FD-completions may indefinitely add more and more points to the metric space. Consider the following example.
Let $X=\mathbb{N}$ and 
\begin{align*}
d(m,n)=\left\{
\begin{array}{rl}
\frac{1}{m}-\tfrac{1}{n},&\quad\text{ if }m\leq n,\\
1,&\quad\text{ otherwise.}
\end{array}\right.
\end{align*}
Successive FD-completions will add a string of ``$\infty$''s, which we denote by $\infty_1,\infty_2,\infty_3,\ldots$ so that $d(\infty_m,\infty_n)=0$ if $m\geq n$, but $1$ otherwise. There may be ways of addressing this by, for example, applying infinitely many FD-completions. This is a delicate discussion, however, for one should take into account the purpose that such a completion should serve, and we leave this for future studies.
\end{remark}

\subsection{FD-completion generalises Cauchy completion}

For the purposes of this subsection, we assume that $(X,d)$ is symmetric. The goal is to show that FD-completion naturally generalises Cauchy completion within this restricted context. We first lay down some groundwork:

\begin{lemma}
\label{lem:fdiscauchy}
Any FD-sequence in a symmetric metric space $(X,d)$ is also Cauchy in $(X,d)$.
\end{lemma}

\begin{proof}
Let $(x_n)$ be an FD-sequence. For any $\epsilon>0$, take $N$ to be some integer such that $\sum_{k=N}^\infty d(x_k,x_{k+1})<\epsilon$. Then, for any $m>n\geq N$, the distance $d(x_n,x_m)$ satisfies
\[
d(x_n,x_m)
\leq
d(x_n,x_{n+1})+\cdots+d(x_{m-1},x_m)
\leq
\sum_{k=N}^\infty d(x_k,x_{k+1})<\epsilon.
\]
\end{proof}

\begin{lemma}
\label{lem:surjective}
Any Cauchy sequence in a symmetric metric space $(X,d)$ has a subsequence which is an FD-sequence in $(X,d)$.
\end{lemma}

\begin{proof}
Given a Cauchy sequence $(x_n)$, let $(m_k)_{k\in\mathbb{N}}$ be an increasing sequence of integers such that $\forall i,j\leq m_k$, $d(x_i,x_j)<2^{-k}$. Then, $(x_{m_k})_{k\in\mathbb{N}}$ is an FD-sequence, as $\sum_{k=1}^\infty d(x_{m_k},x_{m_{k+1}})<1$.
\end{proof}

\begin{lemma}
\label{lem:congruence}
Two FD-sequences $(x_n)$ and $(y_n)$ in a symmetric metric space $(X,d)$ are FD-equivalent if and only if they are Cauchy equivalent.
\end{lemma}

\begin{proof}
First assume that $(x_n)$ and $(y_n)$ are FD-equivalent. For any $\epsilon>0$, there exists some $N$ such that the following two inequalities hold:
\[
\sum_{k=N}^\infty d(x_k,x_{k+1}) <\frac{\epsilon}{3}\quad\text{and}\quad\sum_{k=N}^\infty d(y_k,y_{k+1})
<\frac{\epsilon}{3}.
\]
For any $m\geq N$, we can find some successive pair of terms $x_{i_m},y_{j_m}$ in an interlacing for $(x_n)$ and $(y_n)$ such that $i_m,j_m> m$ and $d(x_{i_m},y_{j_m})<\frac{\epsilon}{3}$ (since $(x_n)$ and $(y_n)$ are FD-equivalent). Then,
\begin{align*}
d(x_m,y_m)\leq 
\sum_{k=m}^{i_m-1} d(x_k,x_{k+1})
+d(x_{i_m},y_{j_m})
+\sum_{k=m}^{j_m-1} d(y_{k+1},y_k)<\epsilon.
\end{align*}
Hence, $\underset{m\to\infty}{\lim} d(x_m,y_m)=0$ and we see that $(x_n)$ and $(y_n)$ (which are necessarily Cauchy, by \cref{lem:fdiscauchy}) are Cauchy equivalent.

For the converse, assume that $\underset{k\to\infty}{\lim} d(x_k,y_k)=0$. Choose an increasing sequence of indices $(m_k)_{k\in\mathbb{N}}$ such that
\[
\forall n\geq m_k,\quad
d(x_n,y_n)<2^{-k}.
\]
Then the sequence $x_{m_1},y_{m_1},y_{m_2},x_{m_2},x_{m_3},y_{m_3},y_{m_4},x_{m_4},\ldots$ is an FD-sequence because its distance-series is bounded above by
\[
\sum_{k=1}^\infty(d(x_k,x_{k+1})+d(y_k,y_{k+1})+2^{-k})<\infty.
\]
Hence, this interlacing induces an FD-equivalence between $(x_n)$ and $(y_n)$.
\end{proof}

\begin{lemma}
\label{lem:fwdbwdsame}
The forward and backward FD-completions of a symmetric metric space $(X,d)$ are canonically isometric symmetric metric spaces.
\end{lemma}

\begin{proof}
We first note that the symmetry of $d$ as a metric means that the set of forward and backward FD-sequences agree, which in turn means that the two completions canonically identify with each other. Then, it is clear from the definition of $\bar{d}$ in \cref{thm:asymmetricmetric} and \cref{rmk:backwardmetric} that the two metric extensions agree and are symmetric.
\end{proof}

\cref{lem:fwdbwdsame} allows us to simply refer to ``the FD-completion of $(X,d)$'', and so allows us to pose the following:

\begin{theorem}[FD generalises Cauchy]
\label{thm:Fd=cauchy}
For any symmetric metric space $(X,d)$, the identity map $\mathrm{id}:X\to X$ uniquely extends to a continuous map between the FD-completion of $(X,d)$ and its Cauchy completion. Moreover, this extension map is an isometric homeomorphism between $(\overline{X},\bar{d})$ and the Cauchy completion.
\end{theorem}

\begin{proof}
We first introduce a little notation. We denote the FD-equivalence class of an FD-sequence $(x_n)\in\overline{X}$ by $[x_n]_{FD}$. We use $(\widehat{X},\hat{d})$ to denote the Cauchy completion, and write Cauchy-equivalence classes as ${[x_n]_C}$. Since FD-sequences are Cauchy (\cref{lem:fdiscauchy}), the obvious candidate for an extension map $\bar{\mathrm{id}}:\overline{X}\to\widehat{X}$ is given by $\bar{\mathrm{id}}([x_n]_{FD})={[x_n]_C}$. For this map to be well-defined, we require that FD-equivalent FD-sequences are also Cauchy equivalent, and this is precisely given by the ``only if'' direction in \cref{lem:congruence}. Furthermore, the ``if'' direction of \cref{lem:congruence} shows that $\bar{\mathrm{id}}$ is injective, and \cref{lem:surjective} shows that $\bar{\mathrm{id}}$ is surjective, as a Cauchy sequence is Cauchy equivalent to any of its subsequences. Thus, $\bar{\mathrm{id}}:\overline{X}\to\widehat{X}$ is a well-defined bijection between these two sets.

\cref{thm:asymmetricmetric} (in conjunction with \cref{lem:infimalrepsequence}, if one so wishes) shows that $\bar{d}([x_n]_{FD},[y_n]_{FD})=\hat{d}({[x_n]_C},{[y_n]_C})$, and thus $\bar{\mathrm{id}}$ is an isometric homeomorphism. This means in particular that $\bar{\mathrm{id}}$ is a \emph{continuous} extension of $\mathrm{id}$. All that remains is to show that $\bar{\mathrm{id}}$ is the unique continuous extension. Consider a continuous extension $\phi:\overline{X}\to\widehat{X}$. Take an arbitrary element $\xi=[x_n]_{FD}\in\overline{X}$, represented by an FD-sequence $(x_n)$. Let $\mathbf{x}_k$ be the constant FD-sequence represented by $x_k$. 
Then, the sequence of FD-equivalence classes of constant FD-sequences $\mathbf{x}_1,\mathbf{x}_2,\mathbf{x} _3,\ldots$ converges to $\xi$ in the sense that 
\[
\lim_{k\to\infty}\bar{d}(\mathbf{x}_k,\xi)=0.
\]
Since $\phi$ is continuous, limits are preserved and
\[
\lim_{k\to\infty}\hat{d}(\phi(\mathbf{x}_k),\phi(\xi)))=0.
\]
On the other hand, since $[x_n]_C$ is the Cauchy equivalence class represented by the sequence $(x_n)=(\phi(x_n))$ and $\phi(\mathbf{x}_k)$ is the constant Cauchy equivalence class represented by $x_k=\phi(x_k)$, we see that 
\begin{equation*}
	\lim_{k\to\infty} \hat{d}(\phi(\mathbf{x}_k), [x_n]_C)=0.
\end{equation*}
Therefore,
\[
\phi(\xi)=[x_n]_C.
\]
Consequently,  $\phi=\bar{\mathrm{id}}$. This completes the proof.
\end{proof}

\subsection{FD-extensibility}

We return now to the general setting where $(X,d)$ is an asymmetric metric space. Thanks to \cref{thm:asymmetricmetric} and \cref{prop:natisoinclude}, FD-completion can be regarded as a process that takes in $(X,d)$ and outputs a larger metric space $(\overline{X},\bar{d})$. It is natural to wonder if this process can be framed more category theoretically. However, even for symmetric metric spaces, one cannot always extend a continuous map between two (symmetric) metric spaces to a continuous map between their completions. Rather, only maps which take Cauchy sequences to Cauchy sequences produce a well-defined extension. We do something analogous for FD-completion.

\begin{definition}[FD-extensibility]
A map $f:(X,d_X)\to(Y,d_Y)$ is called \emph{forward FD-extensible} if the image of every forward FD-sequence
\begin{itemize}
\item
contains at least one subsequence which is a forward FD-sequence in $(Y,d_Y)$, and
\item
all such subsequences are forward FD-equivalent in $(Y,d_Y)$.
\end{itemize}
Backward FD-extensibility is defined analogously.
\end{definition}

\begin{proposition}[standard extension]
Any FD-extensible map $f:(X,d_X)\to(Y,d_Y)$ induces a natural extension map $\bar{f}:\overline{X}\to \overline{Y}$. We refer to $\bar{f}$ as the \emph{standard extension} of $f$.
\end{proposition}

\begin{proof}
Since $f$ is FD-extensible, we specify the map $\bar{f}$ as follows: given an FD-sequence $(x_n)$ representing an FD-equivalence class $[x_n]$, we define $\bar{f}([x_n])$ by sending $[x_n]$ to the unique FD-equivalence class represented by FD-subsequences of $(f(x_n))_{n\in\mathbb{N}}$. This is the obvious candidate for $\bar{f}$. We first show that this defines a well-defined map. Consider $(x_n)$ and $(x'_n)$ to be FD-sequences representing the same FD-equivalence class $[x_n]$. Let $(f(x_{n_i}))_{i\in\mathbb{N}}$ be an FD-sequence representing the unique FD-equivalence class among subsequences of $(f(x_n))$ and likewise let $(f(x'_{n_j}))_{j\in\mathbb{N}}$ be an FD-sequence representing the unique FD-equivalence class among subsequences of $(f(x'_n))$. Our present goal is to show that $(f(x_{n_i}))$ and $(f(x'_{n_j}))$ are FD-equivalent. Since $(x_{n_i})$ and $(x'_{n_j})$ are subsequences of $(x_n)$ and $(x'_n)$, they must be FD-equivalent, hence there is some interlacing $(\hat{x}_m)$ of $(x_{n_i})$ and $(x'_{n_j})$. Since $f$ is an FD-morphism, all FD-sequences of $(f(\hat{x}_m))$ are FD-equivalent. In particular, the sequences of $(f(\hat{x}_m))$ comprised of terms in $(f(x_{n_i}))$ and $(f(x'_{n_j}))$ both have finite distance series, hence they are in the same FD-equivalence class. This means that $[f(x_{n_i})]=[f(x'_{n_j})]$, hence $\bar{f}$ is well-defined. The fact that $\bar{f}$ is an extension is clear from the construction.
\end{proof}

Given the naturality properties inherent in the definition of the standard extension, one wonders if the standard extension is ``canonical'' in some formal sense --- in the symmetric setting, uniqueness is forced by requiring continuity. However, to attempt to phrase such a question in the full generality of asymmetric metric spaces, one needs to consider the subtleties of \emph{which} topology to use. We skirt this issue by considering Busemannian metric spaces:

\begin{proposition}
\label{prop:uniqueextension}
Consider asymmetric metric spaces $(X,d_X)$ and $(Y,d_Y)$ whose respective FD-completions $\phi:(\overline{X},\bar{d}_X)$, $(\overline{Y},\bar{d}_Y)$ are Busemannian metric spaces. (Note that this also makes $(X,d_X)$ and $(Y,d_Y)$ Busemannian, endowing them also with a standard topology.) Then, any FD-extensible map $f:X\to Y$ admits at most one continuous extension. Furthermore, whenever such a continuous extension exists, it is necessarily the standard extension $\bar{f}$.
\end{proposition}

\begin{proof}
The argument is essentially the same as the proof of uniqueness in \cref{thm:Fd=cauchy}. Let $\phi:\overline{X}\to\overline{Y}$ denote a continuous extension of $f:X\to Y$. Given an arbitrary element $\xi\in\overline{X}$, let $(x_n)$ be an FD-sequence representing $\xi$. Since $f$ is FD-extensible, we can replace $(x_n)$ with a subsequence so that $(f(x_n))$ is an FD-sequence in $(Y,d_Y)$. Let $\xi_k\in\overline{X}$ denote the FD-equivalence class represented by the constant FD sequence $[x_k]$. Then $\xi_k\to\xi$ because
\begin{align*}
\lim_{k\to\infty} \bar{d}_X(\xi_k,\xi)
\leq
\lim_{k\to\infty}\lim_{n\to\infty} d_X(x_k,x_n)
\leq
\lim_{k\to\infty}
\sum_{j=k}^\infty d_X(x_k,x_{k+1})=0.
\end{align*}
Continuous maps between Busemannian spaces preserve limits, hence
\begin{align}
\lim_{k\to\infty}
\bar{d}_Y(\phi(\xi_k),\phi(\xi))=0.\label{eq:triineq1}
\end{align}
Now, since $(f(x_n))_{n\in\mathbb{N}}$ is an FD-sequence in $(Y,d_Y)$, its FD-equivalence class $[f(x_n)]_{n\in\mathbb{N}}$ defines an element in $\overline{Y}$. We observe that
\begin{align*}
\lim_{k\to\infty}
\bar{d}_Y(\phi(\xi_k),[f(x_n)]_{n\in\mathbb{N}})
&=
\lim_{k\to\infty}
\bar{d}_Y([f(\xi_k)]_{n\in\mathbb{N}},[f(x_n)]_{n\in\mathbb{N}})\notag\\
&\leq
\lim_{k\to\infty}\lim_{n\to\infty}
d_Y(f(x_k),f(x_n))\notag\\
&\leq
\lim_{k\to\infty}
\sum_{j=k}^\infty d_Y(f(x_j),f(x_{j+1}))
=0.
\end{align*}
Since $(\overline{Y},\bar{d}_Y)$ is Busemannian, this means
\begin{align}
\lim_{k\to\infty}
\bar{d}_Y([f(x_n)]_{n\in\naturals},\phi(\xi_k))=0.\label{eq:triineq2}
\end{align}
Combining \cref{eq:triineq1} and \cref{eq:triineq2}, we see that
\[
\bar{d}_Y([f(x_n)]_{n\in \naturals},\phi(\xi))
\leq
\lim_{k\to\infty}
(\bar{d}_Y([f(x_n)]_{n\in \naturals},\phi(\xi_k))
+
\bar{d}_Y(\phi(\xi_k),\phi(\xi)))=0.
\]
Again using the fact that $(\overline{Y},\bar{d}_Y)$ is Busemannian, we conclude that the continuity of $\phi$ forces $\phi(\xi_k)=[f(x_n)]_{n\in\naturals}$,  hence $\phi$ must be the standard extension.
\end{proof}

\begin{remark}
It is possible to weaken the conditions of the previous lemma so that $(\overline{X},\bar{d}_X)$ is not necessarily Busemannian, and to specify continuity with respect to the backward topology on $(\overline{X},\bar{d}_X)$.
\end{remark}

FD-extensibility is a difficult property to verify without some stronger condition in place, such as Lipschitz continuity. Indeed, Lipschitz continuity is a very natural condition for preserving FD-sequences, hence also FD-equivalence.

\begin{proposition}
\label{prop:FDmorphisms}
Every $K$-Lipschitz map $f:(X,d_X)\to(Y,d_Y)$ is FD-extensible, and induces a standard $K$-Lipschitz map $\bar{f}:(\overline{X},\bar{d}_X)\to(\overline{Y},\bar{d}_Y)$. 
\end{proposition}

\begin{proof}
Since the image of an FD-sequence under a $K$-Lipschitz map is an FD-sequence and all subsequences are in the same FD-equivalence class, $K$-Lipschitz maps are FD-extensible. \cref{lem:infimalrepsequence} then suffices to ensure that $\bar{f}$ is $K$-Lipschitz. 

\end{proof}

\subsection{Categorification of FD-completion}

\cref{prop:FDmorphisms} motivates the following categorical definition.

\begin{definition}[$\mathcal{AML}$ and the FD-completion functor]
\label{defn:aml}
Let $\mathcal{AML}$ denote the category whose objects are asymmetric metric spaces and whose morphisms are Lipschitz maps. For $K\in[1,\infty)$, we similarly define $\mathcal{AML}_K$ to be the subcategory consisting of the same objects, and with morphisms given by $K$-Lipschitz maps. Note that $K\geq1$ is needed to allow for identity morphisms.
\end{definition}

\begin{proposition}\label{prop:compatibility}
The FD-completion defines an endofunctor on $\mathcal{AML}$, as well as on $\mathcal{AML}_K$ for any $K\in[1,\infty)$.
\end{proposition}

\begin{proof}
This is a corollary of \cref{thm:asymmetricmetric} and \cref{prop:FDmorphisms}.
\end{proof}
 
\begin{remark}
It is tempting to wonder if one can enrich $\mathcal{AML}$ by taking a  class of FD-extensible maps as morphisms. This needs a fairly involved discussion which is orthogonal to our present purpose, and deserves independent future investigation.
\end{remark}

 \begin{question}
Constructions based on \cref{rmk:example} might suggest that a 
naive reading of the universal property of metric completions (either in terms of isometric or uniformly continuous maps) seems to fail because pathological asymmetric metric spaces might admit isometric self-maps which are equal to the identity on a dense subset, but not everywhere. Is there a remedy for this either by adjusting the universal property or by imposing additional conditions on the class of asymmetric metric spaces considered?
\end{question}

\begin{remark}
There are other approaches to completing asymmetric metric spaces, see \cite{Algom-Kfir2020,BvBR1998,KS2002,Schellekens2002,Vickers2005,Vickers2009}, and the closest that we have seen is Algom-Kfir's \cite{Algom-Kfir2020}. Both her construction and ours are motivated by simple geometric considerations, and are more accessible. It would be interesting to see to what extent our constructions differ, and what differences they detect in the geometry of asymmetric metric spaces.
\end{remark}

\vspace{8em}

%=====================
\newpage
\bibliographystyle{acm}
\bibliography{earthquake}

\begin{thebibliography}{10}

\bibitem{abikoff1977degenerating}
{\sc Abikoff, W.}
\newblock Degenerating families of {R}iemann surfaces.
\newblock {\em Annals of Mathematics 105}, 1 (1977), 29--44.

\bibitem{Abikoff}
{\sc Abikoff, W.}
\newblock {\em The real analytic theory of {Teichm{\"u}ller} space}, vol.~820
  of {\em Lect. Notes Math.}
\newblock Springer, Cham, 1980.

\bibitem{AJP-History}
{\sc A'Campo, N., Ji, L., and Papadopoulos, A.}
\newblock On the early history of teichm\"uller and moduli spaces.
\newblock In {\em Lipman Bers, a life in mathematics}. Providence, RI: American
  Mathematical Society (AMS), 2015, pp.~175--262.

\bibitem{Algom-Kfir2020}
{\sc Algom-Kfir, Y.}
\newblock The metric completion of outer space.
\newblock {\em Geom. Dedicata 204\/} (2020), 191--230.

\bibitem{Barbot-Fillastre}
{\sc Barbot, T., and Fillastre, F.}
\newblock Quasi-fuchsian co-minkowski manifolds.
\newblock In {\em In the Tradition of Thurston (ed. K. Ohshika and A.
  Papadopoulos)}. Springer International Publishing, 2020, pp.~645--703.

\bibitem{bers_simultaneous}
{\sc Bers, L.}
\newblock Simultaneous uniformization.
\newblock {\em Bull. Amer. Math. Soc. 66\/} (1960), 94--97.

\bibitem{BvBR1998}
{\sc Bonsangue, M.~M., van Breugel, F., and Rutten, J. J. M.~M.}
\newblock Generalized metric spaces: completion, topology, and powerdomains via
  the {Y}oneda embedding.
\newblock {\em Theoret. Comput. Sci. 193}, 1-2 (1998), 1--51.

\bibitem{BST2017}
{\sc Bonsante, F., Seppi, A., and Tamburelli, A.}
\newblock On the volume of anti--de {S}itter maximal globally hyperbolic
  three-manifolds.
\newblock {\em Geom. Funct. Anal. 27}, 5 (2017), 1106--1160.

\bibitem{Brock2003}
{\sc Brock, J.~F.}
\newblock The {W}eil-{P}etersson metric and volumes of 3-dimensional hyperbolic
  convex cores.
\newblock {\em J. Amer. Math. Soc. 16}, 3 (2003), 495--535.

\bibitem{BMW2012}
{\sc Burns, K., Masur, H., and Wilkinson, A.}
\newblock The {W}eil-{P}etersson geodesic flow is ergodic.
\newblock {\em Ann. of Math. (2) 175}, 2 (2012), 835--908.

\bibitem{burothesis}
{\sc Buro, G.}
\newblock {\em Les g\'{e}od\'{e}siques des m\'{e}triques finsl\'{e}riennes et
  pseudo-finsl\'{e}riennes faibles de basse r\'{e}gularit\'{e}}.
\newblock Phd thesis, \'{E}cole Polytechnique F\'{e}d\'{e}rale de Lausanne,
  Lausanne, February 2023.
\newblock Available at
  \url{https://epfl.swisscovery.slsp.ch/permalink/41SLSP_EPF/1g1fbol/alma99116909360005516}.

\bibitem{Busemann1932}
{\sc Busemann, H.}
\newblock {\"U}ber die geometrien, in denen die ``kreise mit unendlichem
  radius'' die k{\"u}rzesten linien sind.
\newblock {\em Math. Ann. 106\/} (1932), 140--160.

\bibitem{busemann1944}
{\sc Busemann, H.}
\newblock Local metric geometry.
\newblock {\em Trans. Amer. Math. Soc. 56\/} (1944), 200--274.

\bibitem{Busemann-synthetic}
{\sc Busemann, H.}
\newblock {\em Recent synthetic differential geometry}, vol.~54 of {\em Ergeb.
  Math. Grenzgeb.}
\newblock Springer-Verlag, Berlin, 1970.

\bibitem{Busemann1932-translation}
{\sc Busemann, H.}
\newblock On spaces with convex spheres and the parallel postulate.
\newblock In {\em Herbert Busemann, Selected Works, I, translation by A.
  A'Campo-Neuen of ``{\"U}ber die Geometrien, in denen die `Kreise mit
  unendlichem Radius' die k{\"u}rzesten Linien sind"}, A.~Papadopoulos, Ed.
  Springer, 2017, pp.~150--172.

\bibitem{Buser2010}
{\sc Buser, P.}
\newblock {\em Geometry and spectra of compact {R}iemann surfaces}.
\newblock Modern Birkh\"{a}user Classics. Birkh\"{a}user Boston, Ltd., Boston,
  MA, 2010.
\newblock Reprint of the 1992 edition.

\bibitem{calderonfarre}
{\sc Calderon, A., and Farre, J.}
\newblock Shear-shape cocycles for measured laminations and ergodic theory of
  the earthquake flow.
\newblock {\em arXiv preprint arXiv:2102.13124\/} (2021).

\bibitem{DLRT}
{\sc Dumas, D., Lenzhen, A., Rafi, K., and Tao, J.}
\newblock Coarse and fine geometry of the {T}hurston metric.
\newblock {\em Forum Math. Sigma 8\/} (2020), Paper No. e28, 58.

\bibitem{epstein1987convex}
{\sc Epstein, D. B.~A., and Marden, A.}
\newblock Convex hulls in hyperbolic space, a theorem of {S}ullivan, and
  measured pleated surfaces.
\newblock In {\em Analytical and geometric aspects of hyperbolic space
  ({C}oventry/{D}urham, 1984)}, vol.~111 of {\em London Math. Soc. Lecture Note
  Ser.} Cambridge Univ. Press, Cambridge, 1987, pp.~113--253.

\bibitem{faltings1983real}
{\sc Faltings, G.}
\newblock Real projective structures on {R}iemann surfaces.
\newblock {\em Compositio Mathematica 48}, 2 (1983), 223--269.

\bibitem{FLPtranslation}
{\sc Fathi, A., Laudenbach, F., and Po\'{e}naru, V.}
\newblock {\em Thurston's work on surfaces}, vol.~48 of {\em Mathematical
  Notes}.
\newblock Princeton University Press, Princeton, NJ, 2012.
\newblock Translated from the French 1979 original by Djun M. Kim and Dan
  Margalit.

\bibitem{FLP}
{\sc Fathi, A., Laundenbach, F., and Po{\'e}naru, V.}
\newblock Travaux de {Thurston} sur les surfaces. {S{\'e}minaire} {Orsay}.
\newblock Centre {National} de la {Recherche} {Scientifique}. {Ast{\'e}risque},
  No. 66-67, {Paris}, 1979.

\bibitem{goldman1987projective}
{\sc Goldman, W.~M.}
\newblock Projective structures with {F}uchsian holonomy.
\newblock {\em Journal of differential geometry 25}, 3 (1987), 297--326.

\bibitem{gromov1999metric}
{\sc Gromov, M.}
\newblock {\em Metric structures for Riemannian and non-Riemannian spaces}.
\newblock Springer.
\newblock With appendices by {\sc Katz}, Misha, {\sc Pansu}, Pierre and {\sc
  Semmes}, Stephen. English translation by {\sc Bates}, Sean Michael.

\bibitem{gunning1967lectures}
{\sc Gunning, R.~C.}
\newblock {\em Lectures on vector bundles over Riemann surfaces}, vol.~6.
\newblock Princeton university press, 1967.

\bibitem{hamenstadt2003length}
{\sc Hamenst\"{a}dt, U.}
\newblock Length functions and parameterizations of {T}eichm\"{u}ller space for
  surfaces with cusps.
\newblock {\em Ann. Acad. Sci. Fenn. Math. 28}, 1 (2003), 75--88.

\bibitem{harvey1974chabauty}
{\sc Harvey, W.}
\newblock Chabauty spaces of discrete groups.
\newblock In {\em Discontinuous groups and Riemann surfaces (Proc. Conf., Univ.
  Maryland, College Park, Md., 1973)\/} (1974), pp.~239--246.

\bibitem{hejhal1975monodromy}
{\sc Hejhal, D.~A.}
\newblock Monodromy groups and linearly polymorphic functions.
\newblock {\em Acta Mathematica 135}, 1 (1975), 1--55.

\bibitem{HOP}
{\sc Huang, Y., Ohshika, K., and Papadopoulos, A.}
\newblock The infinitesimal and global {T}hurston geometry of {T}eichm{\"u}ller
  space.
\newblock {\em arXiv:2111.13381\/} (2021).

\bibitem{IT}
{\sc Imayoshi, Y., and Taniguchi, M.}
\newblock {\em An introduction to {T}eichm\"{u}ller spaces}.
\newblock Springer-Verlag, Tokyo, 1992.
\newblock Translated and revised from the Japanese by the authors.

\bibitem{JP-Historical}
{\sc Ji, L., and Papadopoulos, A.}
\newblock Historical development of {Teichm{\"u}ller} theory.
\newblock {\em Arch. Hist. Exact Sci. 67}, 2 (2013), 119--147.

\bibitem{Ker}
{\sc Kerckhoff, S.~P.}
\newblock {The Nielsen realization problem}.
\newblock {\em Annals of Mathematics 117}, 2 (1983), 235 -- 265.

\bibitem{Ker-ICM}
{\sc Kerckhoff, S.~P.}
\newblock {The geometry of Teichm\"uller space}.
\newblock {\em Proc. Int. Congr. Math., Warszawa, 1983 1\/} (1984), 665 -- 678.

\bibitem{kerckhoff1985earthquakes}
{\sc Kerckhoff, S.~P.}
\newblock Earthquakes are analytic.
\newblock {\em Commentarii Mathematici Helvetici 60}, 1 (1985), 17--30.

\bibitem{kerckhoff1992lines}
{\sc Kerckhoff, S.~P.}
\newblock Lines of minima in {T}eichm{\"u}ller space.
\newblock {\em Duke Mathematical Journal 65}, 2 (1992), 187--213.

\bibitem{kou_complexlength}
{\sc Kourouniotis, C.}
\newblock Complex length coordinates for quasi-{F}uchsian groups.
\newblock {\em Mathematika 41}, 1 (1994), 173--188.

\bibitem{KS2002}
{\sc K\"{u}nzi, H.~P., and Schellekens, M.~P.}
\newblock On the {Y}oneda completion of a quasi-metric space.
\newblock vol.~278. 2002, pp.~159--194.
\newblock Mathematical foundations of programming semantics (Boulder, CO,
  1996).

\bibitem{LRT2012}
{\sc Lenzhen, A., Rafi, K., and Tao, J.}
\newblock Bounded combinatorics and the {L}ipschitz metric on {T}eichm\"{u}ller
  space.
\newblock {\em Geometriae Dedicata 159\/} (2012), 353--371.

\bibitem{maskit1969class}
{\sc Maskit, B.}
\newblock On a class of {K}leinian groups.
\newblock {\em Ann. Acad. Sci. Fenn. Ser. A I 442\/} (1969), 8 pp.

\bibitem{masur1976extension}
{\sc Masur, H.}
\newblock The extension of the {W}eil-{P}etersson metric to the boundary of
  {T}eichmuller space.
\newblock {\em Duke Mathematical Journal 43}, 3 (1976), 623--634.

\bibitem{MT2017}
{\sc Matveev, V., and Troyanov, M.}
\newblock The {M}yers-{S}teenrod theorem for {F}insler manifolds of low
  regularity.
\newblock {\em Proceedings of the American Mathematical Society 145}, 6 (2017),
  2699--2712.

\bibitem{mcmullen1998complex}
{\sc McMullen, C.}
\newblock Complex earthquakes and {T}eichm{\"u}ller theory.
\newblock {\em Journal of the American Mathematical Society 11}, 2 (1998),
  283--320.

\bibitem{mennucci2013asymmetric}
{\sc Mennucci, A.~C.}
\newblock On asymmetric distances.
\newblock {\em Analysis and Geometry in Metric Spaces 1}, 2013 (2013),
  200--231.

\bibitem{Mirzakhani2008}
{\sc Mirzakhani, M.}
\newblock Ergodic theory of the earthquake flow.
\newblock {\em Int. Math. Res. Not. 2008\/} (2008), 39.

\bibitem{mondello_convergence}
{\sc Mondello, G.}
\newblock A criterion of convergence in the augmented {T}eichm\"{u}ller space.
\newblock {\em Bull. Lond. Math. Soc. 41}, 4 (2009), 733--746.

\bibitem{mumford}
{\sc Mumford, D.}
\newblock A remark on {M}ahler's compactness theorem.
\newblock {\em Proc. Amer. Math. Soc. 28\/} (1971), 289--294.

\bibitem{OP}
{\sc Ohshika, K., and Papadopoulos, A.}
\newblock {Hom{\'e}omorphismes et nombre d'intersection}.
\newblock {\em Comptes Rendus de l'Acad{\'e}mie des Sciences, Math{\'e}matique
  356}, 8 (2018), 899--902.

\bibitem{okumura1996global}
{\sc Okumura, Y.}
\newblock Global real analytic length parameters for {T}eichm{\"u}ller spaces.
\newblock {\em Hiroshima Mathematical Journal 26}, 1 (1996), 165--179.

\bibitem{Pan}
{\sc Pan, H.}
\newblock Local rigidity of the {T}eichm{\"u}ller space with the {T}hurston
  metric.
\newblock {\em Science China Mathematics 66\/} (2023), 1751--1766.

\bibitem{PanWolf2022}
{\sc Pan, H., and Wolf, M.}
\newblock Ray structures on {T}eichm\"uller space.
\newblock {\em arXiv:2206.01371v2\/} (June 2022).

\bibitem{PapadopoulosSu2015}
{\sc Papadopoulos, A., and Su, W.}
\newblock On the {F}insler structure of {T}eichm\"{u}ller's metric and
  {T}hurston's metric.
\newblock {\em Expo. Math. 33}, 1 (2015), 30--47.

\bibitem{PT-Harmonic}
{\sc Papadopoulos, A., and Troyanov, M.}
\newblock Harmonic symmetrization of convex sets and of {Finsler} structures,
  with applications to {Hilbert} geometry.
\newblock {\em Expo. Math. 27}, 2 (2009), 109--124.

\bibitem{Riera2005}
{\sc Riera, G.}
\newblock A formula for the {W}eil-{P}etersson product of quadratic
  differentials.
\newblock {\em J. Anal. Math. 95\/} (2005), 105--120.

\bibitem{royden}
{\sc Royden, H.~L.}
\newblock Automorphisms and isometries of {T}eichm\"{u}ller space.
\newblock In {\em Advances in the {T}heory of {R}iemann {S}urfaces ({P}roc.
  {C}onf., {S}tony {B}rook, {N}.{Y}., 1969)}, vol.~No. 66 of {\em Ann. of Math.
  Stud.} Princeton Univ. Press, Princeton, NJ, 1971, pp.~369--383.

\bibitem{Schellekens2002}
{\sc Schellekens, M.~P.}
\newblock Extendible spaces.
\newblock {\em Appl. Gen. Topol. 3}, 2 (2002), 169--184.

\bibitem{schmutz1993parametrisierung}
{\sc Schmutz, P.}
\newblock Die parametrisierung des teichm{\"u}llerraumes durch geod{\"a}tische
  l{\"a}ngenfunktionen.
\newblock {\em Commentarii Mathematici Helvetici 68}, 1 (1993), 278--288.

\bibitem{tan1994}
{\sc Tan, S.~P.}
\newblock Complex {F}enchel-{N}ielsen coordinates for quasi-{F}uchsian
  structures.
\newblock {\em Internat. J. Math. 5}, 2 (1994), 239--251.

\bibitem{Theret}
{\sc Th{\'e}ret, G.}
\newblock {Divergence et parall{\'e}lisme des rayons d'{\'e}tirement
  cylindriques}.
\newblock {\em Algebraic \& Geometric Topology 10}, 4 (2010), 2451--2468.

\bibitem{ThE}
{\sc Thurston, W.~P.}
\newblock Earthquakes in two-dimensional hyperbolic geometry.
\newblock In {\em Low-dimensional topology and {K}leinian groups
  ({C}oventry/{D}urham, 1984)}, vol.~112 of {\em London Math. Soc. Lecture Note
  Ser.} Cambridge Univ. Press, Cambridge, 1986, pp.~91--112.

\bibitem{Thurston-BAMS}
{\sc Thurston, W.~P.}
\newblock On the geometry and dynamics of diffeomorphisms of surfaces.
\newblock {\em Bull. Am. Math. Soc., New Ser. 19}, 2 (1988), 417--431.

\bibitem{ThM}
{\sc Thurston, W.~P.}
\newblock {Minimal stretch maps between hyperbolic surfaces}.
\newblock {\em arXiv.org math.GT\/} (01 1998).
\newblock Published in the Collected Works of William P. Thurston, Vol. I,
  American Mathematical Society, Providence, RI, 2022, p. 533-585.

\bibitem{Torkaman2023}
{\sc Torkaman, T.}
\newblock Intersection number, length, and systole on compact hyperbolic
  surfaces.
\newblock {\em arXiv:2306.09249\/} (2023).

\bibitem{Vickers2005}
{\sc Vickers, S.}
\newblock Localic completion of generalized metric spaces. {I}.
\newblock {\em Theory Appl. Categ. 14\/} (2005), No. 15, 328--356.

\bibitem{Vickers2009}
{\sc Vickers, S.}
\newblock Localic completion of generalized metric spaces. {II}.
  {P}owerlocales.
\newblock {\em J. Log. Anal. 1\/} (2009), Paper 11, 48.

\bibitem{Wolpert1979}
{\sc Wolpert, S.}
\newblock The length spectra as moduli for compact {R}iemann surfaces.
\newblock {\em Ann. of Math. (2) 109}, 2 (1979), 323--351.

\bibitem{Wolpert1982}
{\sc Wolpert, S.}
\newblock The {F}enchel-{N}ielsen deformation.
\newblock {\em Ann. of Math. (2) 115}, 3 (1982), 501--528.

\bibitem{wolpert1986thurston}
{\sc Wolpert, S.}
\newblock Thurston's {R}iemannian metric for {T}eichm{\"u}ller space.
\newblock {\em Journal of Differential Geometry 23}, 2 (1986), 143--174.

\bibitem{Wolpert2008}
{\sc Wolpert, S.}
\newblock Behavior of geodesic-length functions on {T}eichm\"{u}ller space.
\newblock {\em J. Differential Geom. 79}, 2 (2008), 277--334.

\end{thebibliography}

\end{document}